\documentclass[9pt]{amsart}
\usepackage{amsmath}
\usepackage{amsxtra}
\usepackage{amscd}
\usepackage{amsthm}
\usepackage{amsfonts}
\usepackage{amssymb}
\usepackage{eucal}
\usepackage[all]{xy}
\usepackage{graphicx}
\usepackage[usenames]{color}
\usepackage[hidelinks]{hyperref}

\usepackage{tikz-cd}

\newtheorem{cor}[subsubsection]{Corollary}
\newtheorem{lem}[subsubsection]{Lemma}
\newtheorem{prop}[subsubsection]{Proposition}

\newtheorem{conj}[subsubsection]{Conjecture}
\newtheorem{thm}[subsubsection]{Theorem}
\newtheorem{mainthm}[subsubsection]{Main Theorem}

\theoremstyle{definition}

\theoremstyle{remark}
\newtheorem{rem}[subsubsection]{Remark}

\newcommand{\thmref}[1]{Theorem~\ref{#1}}

\newcommand{\secref}[1]{Sect.~\ref{#1}}
\newcommand{\lemref}[1]{Lemma~\ref{#1}}
\newcommand{\propref}[1]{Proposition~\ref{#1}}
\newcommand{\corref}[1]{Corollary~\ref{#1}}
\newcommand{\conjref}[1]{Conjecture~\ref{#1}}
\newcommand{\remref}[1]{Remark~\ref{#1}}

\numberwithin{equation}{section}

\newcommand{\nc}{\newcommand}
\nc{\renc}{\renewcommand}
\nc{\ssec}{\subsection}
\nc{\sssec}{\subsubsection}
\nc{\on}{\operatorname}

\nc\ol{\overline}
\nc\wt{\widetilde}
\nc\tboxtimes{\wt{\boxtimes}}
\nc\tstar{\wt{\star}}
\nc{\alp}{\alpha}

\nc{\ZZ}{{\mathbb Z}}
\nc{\NN}{{\mathbb N}}
\nc{\OO}{{\mathbb O}}
\renc{\SS}{{\mathbb S}}
\nc{\DD}{{\mathbb D}}
\nc{\GG}{{\mathbb G}}

\nc{\Fq}{{\mathbb F}_q}
\nc{\Fqb}{\ol{{\mathbb F}_q}}
\nc{\Ql}{\ol{{\mathbb Q}_\ell}}
\nc{\id}{\text{id}}
\nc\X{\mathcal X}

\nc{\Hom}{\on{Hom}}
\nc{\Lie}{\on{Lie}}
\nc{\Loc}{\on{Loc}}
\nc{\Pic}{\on{Pic}}
\nc{\Bun}{\on{Bun}}
\nc{\IC}{\on{IC}}
\nc{\Aut}{\on{Aut}}
\nc{\rk}{\on{rk}}
\nc{\Sh}{\on{Sh}}
\nc{\Perv}{\on{Perv}}
\nc{\pos}{{\on{pos}}}
\nc{\Conv}{\on{Conv}}
\nc{\Sph}{\on{Sph}}
\nc{\Sym}{\on{Sym}}
\nc{\BunBb}{\overline{\Bun}_B}
\nc{\BunNb}{\overline{\Bun}_N}
\nc{\BunTb}{\overline{\Bun}_T}
\nc{\BunBbm}{\overline{\Bun}_{B^-}}
\nc{\BunBbel}{\overline{\Bun}_{B,el}}
\nc{\BunBbmel}{\overline{\Bun}_{B^-,el}}
\nc{\Buno}{\overset{o}{\Bun}}
\nc{\BunPb}{{\overline{\Bun}_P}}
\nc{\BunPtm}{\wt{\Bun}_{P^-}}
\nc{\BunBM}{\Bun_{B(M)}}
\nc{\BunBMb}{\overline{\Bun}_{B(M)}}
\nc{\BunPbw}{{\widetilde{\Bun}_P}}
\nc{\BunBP}{\widetilde{\Bun}_{B,P}}
\nc{\GUb}{\overline{G/U}}
\nc{\GUPb}{\overline{G/U(P)}}

\nc{\Hhom}{\underline{\on{Hom}}}
\nc\syminfty{\on{Sym}^{\infty}}
\nc\lal{\ol{\lambda}}
\nc\xl{\ol{x}}
\nc\thl{\ol{\theta}}
\nc\nul{\ol{\nu}}
\nc\mul{\ol{\mu}}
\nc\Sum\Sigma
\nc{\oX}{\overset{\circ}{X}{}}
\nc{\hl}{\overset{\leftarrow}h{}}
\nc{\hr}{\overset{\rightarrow}h{}}
\nc{\M}{{\mathcal M}}
\nc{\N}{{\mathcal N}}
\nc{\F}{{\mathcal F}}
\nc{\D}{{\mathcal D}}
\nc{\Q}{{\mathcal Q}}
\nc{\Y}{{\mathcal Y}}
\nc{\G}{{\mathcal G}}
\nc{\E}{{\mathcal E}}
\nc{\CalC}{{\mathcal C}}
\nc\Dh{\widehat{\D}}

\nc{\C}{{\mathcal C}}
\nc{\K}{{\mathcal K}}
\renewcommand{\H}{{\mathcal H}}

\nc{\T}{{\mathcal T}}
\nc{\V}{{\mathcal V}}
\renc{\P}{{\mathcal P}}
\nc{\A}{{\mathcal A}}
\nc{\B}{{\mathcal B}}
\nc{\U}{{\mathcal U}}

\nc{\Gr}{{\on{Gr}}}

\nc{\frn}{{\check{\mathfrak u}(P)}}

\nc{\fC}{\mathfrak C}
\nc{\p}{\mathfrak p}
\nc{\q}{\mathfrak q}
\nc\f{{\mathfrak f}}

\nc{\qo}{{\mathfrak q}}
\nc{\po}{{\mathfrak p}}
\nc{\s}{{\mathfrak s}}
\nc\w{\text{w}}

\renewcommand{\mod}{{\on{-mod}}}

\nc\mathi\iota
\nc\Spec{\on{Spec}}
\nc\Mod{\on{Mod}}
\nc{\tw}{\widetilde{\mathfrak t}}
\nc{\pw}{\widetilde{\mathfrak p}}
\nc{\qw}{\widetilde{\mathfrak q}}
\nc{\jw}{\widetilde j}

\nc{\grb}{\overline{\Gr}}
\nc{\I}{\mathcal I}

\nc{\lambdach}{{\check\lambda}}
\nc{\Lambdach}{{\check\Lambda}{}}
\nc{\much}{{\check\mu}}
\nc{\omegach}{{\check\omega}}
\nc{\nuch}{{\check\nu}}
\nc{\etach}{{\check\eta}}
\nc{\alphach}{{\check\alpha}}
\nc{\oblvtach}{{\check\oblvta}}
\nc{\rhoch}{{\check\rho}}
\nc{\ch}{{\check h}}

\nc{\Hb}{\overline{\H}}

\nc{\BA}{{\mathbb{A}}}
\nc{\BC}{{\mathbb{C}}}
\nc{\BE}{{\mathbb{E}}}
\nc{\BF}{{\mathbb{F}}}
\nc{\BG}{{\mathbb{G}}}
\nc{\BM}{{\mathbb{M}}}
\nc{\BO}{{\mathbb{O}}}
\nc{\BD}{{\mathbb{D}}}
\nc{\BL}{{\mathbb{L}}}
\nc{\Bl}{{\mathbb{l}}}
\nc{\BN}{{\mathbb{N}}}
\nc{\BP}{{\mathbb{P}}}
\nc{\BQ}{{\mathbb{Q}}}
\nc{\BR}{{\mathbb{R}}}
\nc{\BZ}{{\mathbb{Z}}}
\nc{\BS}{{\mathbb{S}}}

\nc{\CA}{{\mathcal{A}}}
\nc{\CB}{{\mathcal{B}}}

\nc{\CE}{{\mathcal{E}}}
\nc{\CF}{{\mathcal{F}}}
\nc{\CH}{{\mathcal{H}}}

\nc{\CL}{{\mathcal{L}}}
\nc{\CC}{{\mathcal{C}}}
\nc{\CG}{{\mathcal{G}}}
\nc{\CM}{{\mathcal{M}}}
\nc{\CN}{{\mathcal{N}}}
\nc{\CK}{{\mathcal{K}}}
\nc{\CO}{{\mathcal{O}}}
\nc{\CP}{{\mathcal{P}}}
\nc{\CQ}{{\mathcal{Q}}}
\nc{\CR}{{\mathcal{R}}}
\nc{\CS}{{\mathcal{S}}}
\nc{\CT}{{\mathcal{T}}}
\nc{\CU}{{\mathcal{U}}}
\nc{\CV}{{\mathcal{V}}}
\nc{\CW}{{\mathcal{W}}}
\nc{\CX}{{\mathcal{X}}}
\nc{\CY}{{\mathcal{Y}}}
\nc{\CZ}{{\mathcal{Z}}}
\nc{\CI}{{\mathcal{I}}}
\nc{\cD}{{\mathcal{D}}}
\nc{\ocD}{\overset{\circ}{\mathcal{D}}}

\nc{\csM}{{\check{\mathcal A}}{}}
\nc{\oM}{{\overset{\circ}{\mathcal M}}{}}
\nc{\obM}{{\overset{\circ}{\mathbf M}}{}}
\nc{\oCA}{{\overset{\circ}{\mathcal A}}{}}
\nc{\obA}{{\overset{\circ}{\mathbf A}}{}}
\nc{\ooM}{{\overset{\circ}{M}}{}}
\nc{\osM}{{\overset{\circ}{\mathsf M}}{}}
\nc{\vM}{{\overset{\bullet}{\mathcal M}}{}}
\nc{\nM}{{\underset{\bullet}{\mathcal M}}{}}
\nc{\oD}{{\overset{\circ}{\mathcal D}}{}}
\nc{\obD}{{\overset{\circ}{\mathbf D}}{}}
\nc{\oA}{{\overset{\circ}{\mathbb A}}{}}
\nc{\op}{{\overset{\bullet}{\mathbf p}}{}}
\nc{\cp}{{\overset{\circ}{\mathbf p}}{}}
\nc{\oU}{{\overset{\bullet}{\mathcal U}}{}}
\nc{\oZ}{{\overset{\circ}{\mathcal Z}}{}}
\nc{\ofZ}{{\overset{\circ}{\mathfrak Z}}{}}
\nc{\oF}{{\overset{\circ}{\fF}}}
\nc{\oC}{\overset{\circ}{C}}

\nc{\fa}{{\mathfrak{a}}}
\nc{\fb}{{\mathfrak{b}}}
\nc{\fc}{{\mathfrak{c}}}
\nc{\fch}{{\mathfrak{ch}}}
\nc{\fd}{{\mathfrak{d}}}
\nc{\ff}{{\mathfrak{f}}}
\nc{\fg}{{\mathfrak{g}}}
\nc{\fgl}{{\mathfrak{gl}}}
\nc{\fh}{{\mathfrak{h}}}
\nc{\fj}{{\mathfrak{j}}}
\nc{\fl}{{\mathfrak{l}}}
\nc{\fm}{{\mathfrak{m}}}
\nc{\fn}{{\mathfrak{n}}}
\nc{\fu}{{\mathfrak{u}}}
\nc{\fp}{{\mathfrak{p}}}
\nc{\fr}{{\mathfrak{r}}}
\nc{\fs}{{\mathfrak{s}}}
\nc{\ft}{{\mathfrak{t}}}
\nc{\fT}{{\mathfrak{T}}}
\nc{\fz}{{\mathfrak{z}}}
\nc{\fsl}{{\mathfrak{sl}}}
\nc{\hsl}{{\widehat{\mathfrak{sl}}}}
\nc{\hgl}{{\widehat{\mathfrak{gl}}}}
\nc{\hg}{{\widehat{\mathfrak{g}}}}
\nc{\htt}{{\widehat{\mathfrak{t}}}}
\nc{\chg}{{\widehat{\mathfrak{g}}}{}^\vee}
\nc{\hn}{{\widehat{\mathfrak{n}}}}
\nc{\chn}{{\widehat{\mathfrak{n}}}{}^\vee}

\nc{\fA}{{\mathfrak{A}}}
\nc{\fB}{{\mathfrak{B}}}
\nc{\fD}{{\mathfrak{D}}}
\nc{\fE}{{\mathfrak{E}}}
\nc{\fF}{{\mathfrak{F}}}
\nc{\fG}{{\mathfrak{G}}}
\nc{\fK}{{\mathfrak{K}}}
\nc{\fL}{{\mathfrak{L}}}
\nc{\fM}{{\mathfrak{M}}}
\nc{\fN}{{\mathfrak{N}}}
\nc{\fP}{{\mathfrak{P}}}
\nc{\fU}{{\mathfrak{U}}}
\nc{\fV}{{\mathfrak{V}}}
\nc{\fZ}{{\mathfrak{Z}}}

\nc{\bb}{{\mathbf{b}}}
\nc{\bc}{{\mathbf{c}}}
\nc{\bd}{{\mathbf{d}}}
\nc{\bbf}{{\mathbf{f}}}
\nc{\be}{{\mathbf{e}}}
\nc{\bg}{{\mathbf{g}}}
\nc{\bi}{{\mathbf{i}}}
\nc{\bj}{{\mathbf{j}}}
\nc{\bn}{{\mathbf{n}}}
\nc{\bp}{{\mathbf{p}}}
\nc{\bq}{{\mathbf{q}}}
\nc{\bu}{{\mathbf{u}}}
\nc{\bv}{{\mathbf{v}}}
\nc{\bx}{{\mathbf{x}}}
\nc{\bs}{{\mathbf{s}}}
\nc{\by}{{\mathbf{y}}}
\nc{\bw}{{\mathbf{w}}}
\nc{\bA}{{\mathbf{A}}}
\nc{\bK}{{\mathbf{K}}}
\nc{\bB}{{\mathbf{B}}}
\nc{\bC}{{\mathbf{C}}}
\nc{\bG}{{\mathbf{G}}}
\nc{\bD}{{\mathbf{D}}}
\nc{\bH}{{\mathbf{He}}}
\nc{\bM}{{\mathbf{M}}}
\nc{\bN}{{\mathbf{N}}}
\nc{\bO}{{\mathbf{O}}}
\nc{\bV}{{\mathbf{V}}}
\nc{\bW}{{\mathbf{Wh}}}
\nc{\bX}{{\mathbf{X}}}
\nc{\bY}{{\mathbf{Y}}}
\nc{\bZ}{{\mathbf{Z}}}
\nc{\bS}{{\mathbf{S}}}
\nc{\bT}{{\mathbf{T}}}

\nc{\sA}{{\mathsf{A}}}
\nc{\sB}{{\mathsf{B}}}
\nc{\sC}{{\mathsf{C}}}
\nc{\sD}{{\mathsf{D}}}
\nc{\sF}{{\mathsf{F}}}
\nc{\sG}{{\mathsf{G}}}
\nc{\sK}{{\mathsf{K}}}
\nc{\sM}{{\mathsf{M}}}
\nc{\sO}{{\mathsf{O}}}
\nc{\sR}{{\mathsf{R}}}
\nc{\sU}{{\mathsf{U}}}
\nc{\sW}{{\mathsf{W}}}
\nc{\sQ}{{\mathsf{Q}}}
\nc{\sP}{{\mathsf{P}}}
\nc{\sY}{{\mathsf{Y}}}
\nc{\sZ}{{\mathsf{Z}}}
\nc{\sfp}{{\mathsf{p}}}
\nc{\sfq}{{\mathsf{q}}}
\nc{\sr}{{\mathsf{r}}}
\nc{\sk}{{\mathsf{k}}}
\nc{\su}{{\mathsf{u}}}
\nc{\sv}{{\mathsf{v}}}
\nc{\sg}{{\mathsf{g}}}
\nc{\sff}{{\mathsf{f}}}
\nc{\sfb}{{\mathsf{b}}}
\nc{\sfc}{{\mathsf{c}}}
\nc{\sd}{{\mathsf{d}}}

\nc{\BK}{{\bar{K}}}

\nc{\tA}{{\widetilde{\mathbf{A}}}}
\nc{\tB}{{\widetilde{\mathcal{B}}}}
\nc{\tg}{{\widetilde{\mathfrak{g}}}}
\nc{\tG}{{\widetilde{G}}}
\nc{\TM}{{\widetilde{\mathbb{M}}}{}}
\nc{\tO}{{\widetilde{\mathsf{O}}}{}}
\nc{\tU}{{\widetilde{\mathfrak{U}}}{}}
\nc{\TZ}{{\tilde{Z}}}
\nc{\tx}{{\tilde{x}}}
\nc{\tbv}{{\tilde{\bv}}}
\nc{\tfP}{{\widetilde{\mathfrak{P}}}{}}
\nc{\tz}{{\tilde{\zeta}}}
\nc{\tmu}{{\tilde{\mu}}}

\nc{\urho}{\underline{\rho}}
\nc{\uB}{\underline{B}}
\nc{\uC}{{\underline{\mathbb{C}}}}
\nc{\ui}{\underline{i}}
\nc{\uj}{\underline{j}}
\nc{\ofP}{{\overline{\mathfrak{P}}}}
\nc{\oB}{{\overline{\mathcal{B}}}}
\nc{\og}{{\overline{\mathfrak{g}}}}
\nc{\oI}{{\overline{I}}}

\nc{\eps}{\varepsilon}
\nc{\hrho}{{\hat{\rho}}}

\nc{\one}{{\mathbf{1}}}
\nc{\two}{{\mathbf{t}}}

\nc{\Rep}{{\mathop{\operatorname{\rm Rep}}}}
\nc{\Tot}{{\mathop{\operatorname{\rm Tot}}}}
\nc{\Ker}{{\mathop{\operatorname{\rm Ker}}}}
\nc{\Hilb}{{\mathop{\operatorname{\rm Hilb}}}}
\nc{\End}{{\mathop{\operatorname{\rm End}}}}
\nc{\Ext}{{\mathop{\operatorname{\rm Ext}}}}
\nc{\CHom}{{\mathop{\operatorname{{\mathcal{H}}\it om}}}}
\nc{\GL}{{\mathop{\operatorname{\rm GL}}}}
\nc{\gr}{{\mathop{\operatorname{\rm gr}}}}
\nc{\Id}{{\mathop{\operatorname{\rm Id}}}}
\nc{\de}{{\mathop{\operatorname{\rm def}}}}
\nc{\length}{{\mathop{\operatorname{\rm length}}}}
\nc{\supp}{{\mathop{\operatorname{\rm supp}}}}

\nc{\Cliff}{{\mathsf{Cliff}}}
\nc{\Fl}{\on{Fl}}
\nc{\Fib}{{\mathsf{Fib}}}
\nc{\Coh}{{\on{Coh}}}
\nc{\QCoh}{{\on{QCoh}}}
\nc{\IndCoh}{{\on{IndCoh}}}
\nc{\FCoh}{{\mathsf{FCoh}}}

\nc{\reg}{{\text{\rm reg}}}

\nc{\cplus}{{\mathbf{C}_+}}
\nc{\cminus}{{\mathbf{C}_-}}
\nc{\cthree}{{\mathbf{C}_*}}
\nc{\Qbar}{{\bar{Q}}}
\nc\Eis{{\on{Eis}}}
\nc\Eisb{\ol\Eis{}}
\nc\Eisr{\on{Eis}^{rat}{}}
\nc\wh{\widehat}
\nc{\Def}{\on{Def_{\check{\fb}}(E)}}
\nc{\barZ}{\overline{Z}{}}
\nc{\barbarZ}{\overline{\barZ}{}}
\nc{\barpi}{\overline\pi}
\nc{\barbarpi}{\overline\barpi}
\nc{\barpip}{\overline\pi{}^+}
\nc{\barpim}{\overline\pi{}^-}

\nc{\fq}{\mathfrak q}

\nc{\fqb}{\ol{\fq}{}}
\nc{\fpb}{\ol{\fp}{}}
\nc{\fpr}{{\fp^{rat}}{}}
\nc{\fqr}{{\fq^{rat}}{}}

\nc{\hattimes}{\wh\otimes}

\nc{\bh}{{{\mathbf h}}}
\nc{\bk}{{{\mathbf k}}}
\nc{\bOmega}{{\overline{\Omega(\check \fn)}}}

\nc{\seq}[1]{\stackrel{#1}{\sim}}

\nc{\cT}{{\check{T}}}
\nc{\cG}{{\check{G}}}
\nc{\cM}{{\check{M}}}
\nc{\cB}{{\check{B}}}
\nc{\cP}{{\check{P}}}

\nc{\ct}{{\check{\mathfrak t}}}
\nc{\cg}{{\check{\fg}}}
\nc{\cb}{{\check{\fb}}}
\nc{\cn}{{\check{\fn}}}

\nc{\cLambda}{{\check\Lambda}}

\nc{\cla}{{\check\lambda}}
\nc{\cmu}{{\check\mu}}
\nc{\cnu}{{\check\nu}}
\nc{\ceta}{{\check\eta}}

\nc{\DefbE}{{\on{Def}_{\cB}(E_\cT)}}

\nc{\imathb}{{\ol{\imath}}}
\nc{\rlr}{\overset{\longrightarrow}{\underset{\longrightarrow}\longleftarrow}}

\nc{\oBun}{\overset{\circ}\Bun}
\nc{\oSht}{\overset{\circ}\Sht}
\nc{\LocSys}{\on{LocSys}}
\nc{\BunBbb}{\ol{\ol{Bun}}_B}
\nc{\BunBr}{\Bun_B^{rat}}
\nc{\BunBrp}{\Bun_B^{rat,polar}}
\nc{\BunTrp}{\Bun_T^{rat,polar}}
\nc{\BunNr}{\Bun_N^{rat}}
\nc{\BunNre}{\Bun_N^{enh,rat}}
\nc{\BunTr}{\Bun_T^{rat}}
\nc{\Vect}{\on{Vect}}
\nc{\Whit}{\on{Whit}}
\nc{\CTb}{\ol{\on{CT}}}
\nc{\Ran}{{\on{Ran}}}
\nc{\CTr}{\on{CT}^{rat}{}}
\nc\jmathr{\jmath^{rat}{}}
\nc{\ux}{\underline{x}}
\nc{\clambda}{{\check\lambda}}
\nc{\calpha}{{\check\alpha}}
\nc{\ind}{{\mathbf{ind}}}
\nc{\oblv}{{\mathbf{oblv}}}
\nc{\coeff}{\on{W-coeff}}
\nc{\Poinc}{\on{Poinc}}
\nc{\Dmod}{\on{D-mod}}
\nc{\dr}{\on{dR}}
\nc{\oCZ}{\overset{\circ}\CZ}
\nc{\KL}{\on{KL}}
\nc{\triv}{{\mathbf{triv}}}

\nc{\dgSch}{\on{DGSch}}
\nc{\pSch}{\on{pSch}}
\nc{\fSch}{\on{fSch}}
\nc{\Sch}{\on{Sch}}
\nc{\affdgSch}{\on{DGSch}^{\on{aff}}}
\nc{\affSch}{\on{Sch}^{\on{aff}}}
\nc{\paffSch}{\on{pSch}^{\on{aff}}}
\nc{\Sing}{\on{Sing}}
\nc{\inftygroup}{\infty\on{-Grpd}}
\renc{\dr}{{\on{dr}}}
\nc\Maps{\on{Maps}}
\nc\Res{\on{Res}}
\nc\bMaps{\mathbf{Maps}}
\nc{\ul}{\underline}
\nc{\bNP}{\mathbf{N(P)}}
\nc{\ofc}{\overset{\circ}\fch}
\nc{\ppart}{(\!(t)\!)}
\nc{\qqart}{[\![t]\!]}
\nc{\crit}{\on{crit}}
\nc{\DGCat}{\on{DGCat}}
\nc{\Shv}{\on{Shv}}
\nc{\bDelta}{\mathbf{\Delta}}
\nc{\genB}{{\overset{\on{gen}}\to B}}
\nc{\genP}{{\underset{\on{gen}}\longrightarrow P}}
\nc{\genN}{{\underset{\on{gen}}\longrightarrow N}}
\nc{\semiinf}{{\frac{\infty}{2}+\bullet}}
\nc{\mmod}{\on{-}\mathbf{mod}}
\nc{\AdFr}{\on{Ad}_{\on{Frob}}}
\nc{\Frob}{{\on{Frob}}}
\nc{\Tr}{\on{Tr}}
\nc{\Sht}{\on{Sht}}
\nc{\sfe}{\mathsf{e}}
\nc{\tCat}{{2\on{-Cat}}}
\nc{\tIndCoh}{{2\on{-IndCoh}}}
\nc{\tQCoh}{{2\on{-QCoh}}}
\nc{\uShv}{\underline{\Shv}}
\nc{\qLisse}{\on{QLisse}}
\nc{\Nilp}{{\on{Nilp}}}
\nc{\sotimes}{\overset{!}\otimes} 
\nc{\LS}{\on{LS}}
\nc{\Sat}{\on{Sat}}
\nc{\inftyCat}{{\infty\on{-Cat}}}

\begin{document}

\title[Geometric Langlands in positive characteristic]{Geometric Langlands in positive characteristic \\ from characteristic zero}

\author{Dennis Gaitsgory and Sam Raskin}

\dedicatory{To G\'erard Laumon, with gratitude for the math he brought into existence.} 

\begin{abstract}

We establish part of the statement of the geometric Langlands conjecture for $\ell$-adic sheaves
over a field of positive characteristic. Namely, we show that the category of automorphic sheaves
with nilpotent singular support is equivalent to the appropriately defined category of ind-coherent
sheaves on the \emph{union of some of the connected components of the} stack of Langlands 
parameters. 

\end{abstract}

\maketitle 

\tableofcontents

\section*{Introduction}

\ssec{What is done in this paper?}

\sssec{}

This paper can be considered as a sequel to both the [AGKRRV] and [GLC] series. Namely, 
in \cite[Conjecture 21.2.7]{AGKRRV1} we proposed a version of the geometric Langlands conjecture
(GLC) that makes sense in the context of $\ell$-adic sheaves (for curves over a field of any characteristic). 

\medskip

Namely, it says that the (derived) category 
\begin{equation} \label{e:geom side}
\Shv_{\Nilp}(\Bun_G)
\end{equation} 
of $\ell$-adic sheaves on the moduli stack $\Bun_G$ of principal $G$-bundles on a (smooth and complete)
curve $X$, \emph{with nilpotent singular support}, is equivalent to the category
$$\IndCoh_\Nilp(\LS^{\on{restr}}_\cG),$$
where:

\begin{itemize}

\item $\LS^{\on{restr}}_\cG$ is the prestack of $\cG$-local systems on $X$ \emph{with restricted variation}, 
introduced in \cite[Sect. 1.4]{AGKRRV1};

\medskip

\item The subscript ``Nilp" stands for the restriction on the singular support (as coherent sheaves), introduced in \cite[Sect. 11.1]{AG1}.

\end{itemize}

\begin{rem}

The fact that the subcategory \eqref{e:geom side} inside the ambient $\Shv(\Bun_G)$
is the ``right object to consider" as far as automorphic sheaves
are concerned was a discovery of G.~Laumon in his seminal paper \cite{Laum}. 

\medskip

There he conjectured that all Hecke eigensheaves must belong to this subcategory. This conjecture
was settled in \cite[Theorem 14.4.3]{AGKRRV1}. In fact, more is true: in {\it loc. cit.}, Theorem 14.4.4 
it was shown that ``any object of $\Shv(\Bun_G)$ that remotely looks like a Hecke eigensheaf" belongs to
$\Shv_{\Nilp}(\Bun_G)$. 

\end{rem} 

\sssec{}

Unfortunately, we still cannot prove the full GLC over a field of positive characteristic. Rather, in this paper we establish a partial result.
We construct a functor
\begin{equation} \label{e:Langlands functor Intro}
\BL^{\on{restr}}_\cG:\Shv_{\Nilp}(\Bun_G)\to \IndCoh_\Nilp(\LS^{\on{restr}}_\cG),
\end{equation} 
and we prove.

\begin{thm} \label{t:main(i) Intro}
The functor $\BL_\cG^{\on{restr}}$ factors via an equivalence
$$\Shv_{\Nilp}(\Bun_G)\overset{\sim}\to  \IndCoh_\Nilp({}'\!\LS^{\on{restr}}_\cG)\subset \IndCoh_\Nilp(\LS^{\on{restr}}_\cG),$$
where $'\!\LS^{\on{restr}}_\cG$ is the union of some of the connected components of $\LS^{\on{restr}}_\cG$.
\end{thm}

This is \thmref{t:ff GLC}(i) in the main body of the paper. 

\begin{rem}

We remind the reader that the connected components of $\LS^{\on{restr}}_\cG$ correspond bijectively to
\emph{semi-simple} $\cG$-local systems (two local systems lie in the same connected component if
and only if they have isomorphic semi-simplifications). 

\medskip

In particular, every irreducible $\cG$-local system lies in its own connected component (which is, however,
stacky and non-reduced). 

\end{rem}

\sssec{}

In addition, we prove: 

\begin{thm} \label{t:main(ii) Intro}
If $G=GL_n$, then the inclusion 
$$'\!\LS^{\on{restr}}_\cG\subset \LS^{\on{restr}}_\cG$$
is an equality.
\end{thm}

So at least for $G=GL_n$, Laumon's vision for the structure of what he called ``geometric Langlands correspondence" 
has been fully realized.

\begin{rem} \label{r:no sheaf}

One can say that for an arbitrary group $G$, we have not solved the most mysterious part of the
Langlands conjecture: we do not know that to an irreducible $\cG$-local system there corresponds
a non-zero Hecke eigensheaf.

\medskip

Over fields of characteristic $0$ we know this thanks to the Beilinson-Drinfeld construction
of eigensheaves for D-modules, using localization of modules over the affine Kac-Moody algebra at
the critical level\footnote{Formally speaking, this is not how the proof of GLC for D-modules given in 
\cite{GLC5} proceeds; however, it does crucially rely on the localization of KM-modules, albeit slightly
differently.}. 

\end{rem} 

\sssec{Assumptions on the characteristic}

The above results rely on the validity of \cite{AGKRRV1}, for which certain assumptions on the
characteristic of the ground field are needed, see Sects. 14.4.1 and D.1.1 in {\it loc. cit.}

\ssec{Function-theoretic applications}

The equivalence established in \thmref{t:main(i) Intro} allows us to deduce information about the classical theory
of automorphic functions (in the unramified case over function fields). 

\sssec{}

Namely, we let our ground field $k$ be the algebraic closure of a finite field $\BF_q$, and we assume that both
$X$ and $G$ are defined over $\BF_q$. In this case, the geometric objects involved in \eqref{e:Langlands functor Intro}
carry an automorphism, given by the action of Frobenius.

\medskip

Taking its categorical trace and applying some existing calculations (namely, ones from \cite{AGKRRV3} and \cite{BLR}), 
from \thmref{t:main(i) Intro} we obtain:

\begin{thm} \label{t:classical Intro}
There exists an identification of vector spaces
\begin{equation} \label{e:classical}
\on{Funct}_c(\Bun_G(\BF_q),\ol\BQ_\ell)\simeq \Gamma({}'\!\LS^{\on{arithm}}_\cG,\omega),
\end{equation}
where:

\begin{itemize}

\item $\on{Funct}_c(-,\ol\BQ_\ell)$ stands for the space of compactly supported functions;

\medskip

\item $\LS^{\on{arithm}}_\cG:=(\LS^{\on{restr}}_\cG)^\Frob$ is the moduli space of Weil local systems on $X$ with respect to $\cG$,
and ${}'\!\LS^{\on{arithm}}_\cG:=({}'\!\LS^{\on{restr}}_\cG)^\Frob$;

\smallskip

\item $\omega$ is the dualizing sheaf.

\end{itemize}

\end{thm}

This is stated as \corref{c:main} in the main body of the paper. 

\medskip

Moreover, it follows from the construction of the isomorphism
\eqref{e:classical} that it is compatible with the action of the excursion algebra
$$\CA_G:=\Gamma(\LS^{\on{arithm}}_\cG,\CO)$$
on both sides, and in particular, with the action of the Hecke operators. 

\begin{rem} 

Parallel to Remark \ref{r:no sheaf}, \thmref{t:classical Intro} does not settle the main mystery in classical Langlands:
outside the case of $G=GL_n$, we do not yet know that to an irreducible Langlands parameter there
corresponds a non-zero eigenfunction (if it existed, it would automatically be cuspidal, by the nature of
the isomorphism \eqref{e:classical}).

\end{rem} 

\begin{rem}

The stack $\LS^{\on{arithm}}_\cG$ is Calabi-Yau\footnote{In the sense that the determinant of the cotangent complex 
is the trivial line bundle.}, but is also highly derived in that its structure sheaf has 
non-trivial cohomology in infinitely many negative cohomological degrees. This implies that although there
is a canonical map
\begin{equation} \label{e:temp funct}
\CO_{\LS^{\on{arithm}}_\cG}\to \omega_{\LS^{\on{arithm}}_\cG},
\end{equation}
it is very far from being an isomorphism (although it is such over the \emph{quasi-smooth} locus).

\medskip

Ultimately, this phenomenon is responsible for the presence of the Arthur $SL_2$ in the
classification of automorphic functions. 

\medskip

Yet, we expect that in the map 
\begin{equation} \label{e:temp funct sec}
\Gamma(\LS^{\on{arithm}}_\cG,\CO_{\LS^{\on{arithm}}_\cG})\to \Gamma(\LS^{\on{arithm}}_\cG,\omega_{\LS^{\on{arithm}}_\cG}),
\end{equation}
induced by \eqref{e:temp funct}, both sides are classical vector spaces (i.e., are concentrated in cohomological degree $0$)
and the map \eqref{e:temp funct sec} itself is injective. Moreover, we expect that 
the intersection of the image of \eqref{e:temp funct sec} with the subspace of cuspidal
functions is equal to the space of \emph{tempered} cuspidal functions
(with respect to any isomorphism $\ol\BQ_\ell\simeq \BC$). We hope to take this up in a future work. 

\end{rem} 

\begin{rem}

The paper \cite{Ra2} proves an arithmetic result \textemdash{}
the non-existence of cusp forms with certain Langlands parameters \textemdash{}
conditional on GLC in 
characteristic $p$.
Although we only obtain partial results on GLC in 
this paper, our results suffice for the 
applications in \cite{Ra2}.

\end{rem}

\ssec{The methods}

\sssec{}

The first step is the construction of the functor \eqref{e:Langlands functor Intro}. 
It is here that the significance of the subcategory
$$\Shv_{\Nilp}(\Bun_G)\subset \Shv(\Bun_G)$$
comes to the fore.

\medskip

Namely, it turns out to be the maximal subcategory of $\Shv(\Bun_G)$ on which the action
of the Hecke functors factors through a monoidal action of the category
$$\QCoh(\LS^{\on{restr}}_\cG);$$
we refer to it as the ``spectral action".

\medskip

We recover the functor \eqref{e:Langlands functor Intro}
by requiring that it intertwines the actions of $\QCoh(\LS^{\on{restr}}_\cG)$ on the two sides 
and makes the following diagram commute:
$$
\CD
\Vect @>{\on{Id}}>> \Vect \\
@A{\on{coeff}^{\on{Vac}}}AA @AA{\Gamma^{\IndCoh}(\LS^{\on{restr}}_\cG,-)}A \\
\Shv_{\Nilp}(\Bun_G) @>{\BL^{\on{restr}}}>> \IndCoh_\Nilp(\LS^{\on{restr}}_\cG),
\endCD
$$
where the left vertical arrow is the functor of \emph{vacuum} Whittaker coefficient,
see \secref{sss:coeff expl}. 

\medskip

This follows verbatim the construction of the Langlands functor for D-modules and Betti sheaves
in \cite[Sect. 1]{GLC1}.

\sssec{}

As a next step, we show that the functor \eqref{e:Langlands functor Intro} is an equivalence 
for $\ell$-adic sheaves as long as the ground field over which we work has characteristic $0$.

\medskip

To do so, by the Lefschetz principle, we can replace the initial ground field by the field $\BC$ of complex
numbers. In the latter case, we compare the functor \eqref{e:Langlands functor Intro} for $\ell$-adic
sheaves with its counterpart for Betti sheaves, and we show that if the latter is an equivalence, then
so is the $\ell$-adic version. 

\medskip

Finally, we quote \cite{GLC1}, which says that the Betti version of \eqref{e:Langlands functor Intro}
\emph{is} an equivalence. This is obtained by combining the fact that the D-module version of 
\eqref{e:Langlands functor Intro} is an equivalence (which is the outcome of the [GLC] series)
and the Riemann-Hilbert correspondence. 

\sssec{}

Thus, our task is to deduce \thmref{t:main(i) Intro} for a field of positive characteristic 
from its validity for a field of characteristic $0$.  We achieve this by the following procedure. 

\medskip

Let $\sk$ be our ground field of positive characteristic (assumed algebraically closed). 
Let $\sR_0:=\on{Witt}(\sk)$ be the ring of Witt vectors of $\sk$, let $\sK_0$ denote the field of fractions of $\sR_0$ and let
$\sK$ denote the algebraic closure of $\sK_0$. Let $\sR$ denote the integral closure of $\sR_0$
in $\sK$. 

\medskip

Given a (smooth complete) curve $X_\sk$ over $\sk$, we \emph{choose} its extension to a (smooth complete) curve
$X_{\sR_0}$ over $\Spec(\sR_0)$. (Such an extension exists 
by a standard deformation theory argument.) Let $X_\sK$ be the base change of $X_{\sR_0}$ to $K$.

\medskip

Note that we can identify $\LS^{\on{restr}}_{\cG,\sk}$ with the union of some of the connected components 
of $\LS^{\on{restr}}_{\cG,\sK}$, see \secref{sss:Sp LS}.

\sssec{}

In \secref{s:Sp} we introduce a specialization functor
\begin{equation} \label{e:Sp init Intro}
\on{Sp}:\Shv(\Bun_{G,\sK})\to \Shv(\Bun_{G,\sk}),
\end{equation}
which essentially amounts to the functor of nearby cycles.

\medskip

We establish the following properties of this functor:

\begin{itemize}

\item It commutes with the Hecke functors;

\item It commutes with the functors of Eisenstein series;

\item It sends the vacuum Poincar\'e object\footnote{This is the object that corepresents the functor $\on{coeff}^{\on{Vac}}$.}  
$\on{Poinc}^{\on{Vac}}_{!,\sK}$ to the vacuum Poincar\'e object $\on{Poinc}^{\on{Vac}}_{!,\sk}$.

\end{itemize} 

The first of these properties implies that the functor $\on{Sp}$ sends the direct summand 
$$\Shv_\Nilp(\Bun_{G,\sK,\sk})\subset \Shv_\Nilp(\Bun_{G,\sK})$$
that has to do with $\LS^{\on{restr}}_{\cG,\sk}\subset \LS^{\on{restr}}_{\cG,\sK}$ to 
$\Shv_\Nilp(\Bun_{G,\sk})$, i.e., we obtain a functor
\begin{equation} \label{e:Sp Nilp Intro}
\on{Sp}:\Shv_\Nilp(\Bun_{G,\sK,\sk})\to \Shv_\Nilp(\Bun_{G,\sk}).
\end{equation} 

Combined with the other properties, one shows that the functor \eqref{e:Sp Nilp Intro} preserves compactness 
and makes the following diagram commute:
$$
\CD
\Shv_{\Nilp}(\Bun_{G,\sK,\sk}) @>{\BL^{\on{restr}}_\sK}>{\sim}> \IndCoh_\Nilp(\LS^{\on{restr}}_{\cG.\sk}) \\
@V{\on{Sp}}VV @VV{\on{Id}}V \\
\Shv_{\Nilp}(\Bun_{G,\sk}) @>>{\BL^{\on{restr}}_\sk}> \IndCoh_\Nilp(\LS^{\on{restr}}_{\cG.\sk}). 
\endCD
$$

\sssec{}

However, this is not quite enough to deduce \thmref{t:main(i) Intro}. What we need is another 
crucial property of the functor \eqref{e:Sp Nilp Intro}, which says that this functor is a \emph{Verdier quotient}.

\medskip

It turns out that there is a simple criterion for when a functor between dualizable categories 
is a Verdier quotient, see \lemref{l:criter for loc}. This is a general categorical assertion, but
it turns out that one can apply and check it in our situation; this is due to some rather
special properties of the category $\Shv_{\Nilp}(\Bun_G)$ established in  \cite{AGKRRV2},
most notably, the categorical K\"unneth formula.

\medskip

The fact that this criterion is satisfied for us follows from a certain geometric property of the initial functor 
\eqref{e:Sp init Intro}. Namely, this property says that the natural map 
\begin{equation} \label{e:diag Intro}
(\Delta_{\Bun_{G,\sk}})_!(\ul\sfe_{\Bun_{G,\sk}})\to \on{Sp}((\Delta_{\Bun_{G,\sK}})_!(\ul\sfe_{\Bun_{G,\sK}})),
\end{equation} 
is an isomorphism, where:

\begin{itemize}

\item $\Delta_\CY$ denotes the diagonal morphism of a stack $\CY$;

\item $\ul\sfe_\CY$ denotes the constant sheaf on $\CY$.

\end{itemize} 

In its turn, the fact that \eqref{e:diag Intro} is an isomorphism is equivalent to the ULA property of
$(\Delta_{\Bun_{G,\sR_0}})_!(\ul\sfe_{\Bun_{G,\sR_0}})$ with respect to the projection
$\Bun_{G,\sR_0}\to \Spec(\sR_0)$. 

\medskip

We verify the required ULA property using the Drinfeld-Lafforgue-Vinberg compactification $\ol{\Bun}_G$ of the diagonal
map of $\Bun_G$.

\sssec{Another technical result}

We now mention another result, \thmref{t:compactness}, established in this paper,
which is a crucial technical component for many other theorems that we 
prove.

\medskip

Namely, \thmref{t:compactness} says that the category $\Shv_{\Nilp}(\Bun_G)$ is
generated by objects that are compact in the ambient category $\Shv(\Bun_G)$.

\medskip

This result was conjectured in \cite{AGKRRV1}\footnote{This was used also as a hypothesis in \cite{AGKRRV3}
to show that the trace isomorphism $\Tr(\Frob,\Shv_\Nilp(\Bun_G))\simeq \on{Funct}_c(\Bun_G(\BF_q),\ol\BQ_\ell)$
reproduces the usual pointwise Frobenius map for constructible Weil sheaves.}, and it was proved in {\it loc. cit.}
when the ground field has characteristic $0$. In this paper we deduce it
in the positive characteristic case using the functor \eqref{e:Sp Nilp Intro}.

\ssec{Structure of the paper}

\sssec{}

In \secref{s:l-adic} we construct the Langlands functor \eqref{e:Langlands functor Intro}
and derive function-theoretic applications.

\sssec{}

In \secref{s:char 0} we prove that the functor \eqref{e:Langlands functor Intro}
is an equivalence over a ground field of characteristic $0$. Namely, we deduce
this from the validity of the Betti version of GLC. 

\sssec{}

In \secref{s:Sp} we stipulate the existence of the functor \eqref{e:Sp init Intro} with
some specified properties and deduce \thmref{t:main(i) Intro}.

\sssec{}

In \secref{s:construct funct} we construct the functor \eqref{e:Sp init Intro} and establish some 
of its expected properties. 

\sssec{}

In \secref{s:unit} we state \thmref{t:diag acycl}, which says that $(\Delta_{\Bun_{G,\sR_0}})_!(\ul\sfe_{\Bun_{G,\sR_0}})$
is ULA and show how this implies that the functor \eqref{e:Sp Nilp Intro} is a Verdier quotient.

\sssec{}

In \secref{s:ULA} we prove the local acyclicity theorems responsible for the required
properties of the functor \eqref{e:Sp init Intro}. There are four such theorems:

\medskip

\noindent{(i)} Acyclicity of kernels defining Hecke functors. This is proved 
using (what essentially is) the Bott-Samelson resolution;

\medskip

\noindent{(ii)} Acyclicity of kernels defining Eisenstein functors. This is proved
using local models, called \emph{Zastava} spaces.

\medskip

\noindent{(iii)} Acyclicity of $(\Delta_{\Bun_{G,\sR_0}})_!(\ul\sfe_{\Bun_{G,\sR_0}})$.
This is proved using local models for $\ol{\Bun}_G$, developed in \cite{Sch}. 

\medskip

\noindent{(iv)} Acyclicity of the vacuum Poincar\'e object $\on{Poinc}^{\on{Vac}}_{!,\sR_0}$. 

\medskip

In the present section we prove the first three of these theorems. The proofs are based on the 
\emph{contraction principle}, formulated in \propref{p:preserve ULA}. 

\sssec{}

In \secref{s:proof Poinc} we prove the acyclicity of the vacuum Poincar\'e object. We give two proofs. 
One is shorter, but it uses an additional assumption on the interaction of the characteristic of the field with $g(X)$ and $G$.

\medskip

The second proof uses a comparison between the !-Poincar\'e object $\on{Poinc}^{\on{Vac}}_!$ 
with its *-counterpart $\on{Poinc}^{\on{Vac}}_*$, expressed by \thmref{t:Poinc ! to *}. 

\sssec{}

In \secref{s:asymptotics} we prove \thmref{t:Poinc ! to *}, which says that the cone of the natural map
$$\on{Poinc}^{\on{Vac}}_!\to \on{Poinc}^{\on{Vac}}_*$$
belongs to the full subcategory generated by the essential images of the Eisenstein functors for
proper parabolic subgroups. 

\medskip

\thmref{t:Poinc ! to *} is of independent interest, and the proof that we give is useful as well:
it consists of studying what one may call the asymptotic behavior of the Whttaker sheaf
as we degenerate the character. 

\sssec{}

Finally, in \secref{s:L and Eis} we revisit the topic of the interaction of the Langlands functor 
with Eisenstein series. 

\medskip

So far we have not mentioned Eisenstein series in the introduction, but not surprisingly,
they form an integral part of the theory, and our ability to prove statements about the
Langlands correspondence often relies on having good control of the Eisenstein 
functor. 

\medskip

One particular aspect of this is the computation of the Whittaker coefficient
of Eisenstein series, which is performed in \thmref{t:Whit of Eis}.  

\sssec{Notations and conventions}

Notations and conventions in this paper are identical to those in \cite{GLC1}.

\ssec{Acknowledgements}

We dedicate this paper to G\'erard Laumon, who initiated the study of automorphic sheaves. 
In addition to that, the influence of his ideas in this area is all-pervasive: geometric Eisenstein
series, geometric Fourier transform, local Fourier-Mukai equivalence, to name just a few. 

\medskip

We are grateful to our collaborators on the [AGKRRV] and [GLC] projects:
D.~Arinkin, D.~Beraldo, J.~Campbell, L.~Chen, J.~Faergeman, D.~Kazhdan, K.~Lin, N.~Rozenblyum
and Y.~Varshavsky.

\medskip 

We are also grateful to D.~Ben-Zvi, T.~Feng, D.~Hansen, M.~Harris, V.~Lafforgue, W.~Sawin, P.~Scholze and X.~Zhu
for productive discussions. 

\medskip

This paper was written while the first author was visiting IHES and the second author was visiting the 
Max Planck Institute in Bonn. We wish to thank these institutions for providing excellent working conditions.

\medskip 

The work of the second author was supported by the Sloan Research Fellowship, as well as 
NSF grant DMS-2401526.

\section{The Langlands functor for \texorpdfstring{$\ell$}{ell}-adic sheaves} \label{s:l-adic}

In this section we will work over a ground field $k$, assumed algebraically closed,
but it may have either positive characteristic or characteristic $0$. Our sheaf
theory $\Shv(-)$ (see \cite[Sect. 1.1]{AGKRRV1} for what we mean by that)
will be that of (ind-)constructible $\ol\BQ_\ell$-adic \'etale sheaves;
so our field of coefficients $\sfe$ is $\ol\BQ_\ell$.

\medskip

On the geometric side, we consider the category $\Shv_\Nilp(\Bun_G)$, as defined
in \cite[Sect. 14.1]{AGKRRV1}.  

\medskip

On the spectral side we, we consider the category $\IndCoh_{\Nilp}(\LS^{\on{restr}}_\cG)$, 
as defined in \cite[Sect. 4.1]{AGKRRV1}.

\medskip

The goal of this section is to construct a functor
$$\BL_G^{\on{restr}}:\Shv_\Nilp(\Bun_G)\to \IndCoh_{\Nilp}(\LS^{\on{restr}}_\cG)$$
and state some results and conjectures pertaining to its properties. 

\ssec{Coarse version of the functor} \label{ss:coarse}

\sssec{}

We start by considering the object
$$\on{Poinc}_!^{\on{Vac}}\in \Shv(\Bun_G)^c.$$

It is constructed by the procedure of \cite[Sect. 3.3]{GLC1} (see \secref{sss:recall Poinc} below). 

\begin{rem}
Note that when $k$ has positive characteristic, $\on{Poinc}_!^{\on{Vac}}$ can be equivalently 
constructed by the procedure of \cite[Sect. 1.3.7]{GLC1}, replacing the exponential D-module by the 
Artin-Schreier sheaf, see \cite[Remark 3.3.6]{GLC1}.
\end{rem}

\sssec{}

Recall the functor
$$\sP:\Shv(\Bun_G)\to \Shv_{\Nilp}(\Bun_G)$$
of \cite[Sect. 15.4.5]{AGKRRV1}.

\medskip

Denote:
$$\on{Poinc}_{!,\Nilp}^{\on{Vac}}:=\sP(\on{Poinc}_!^{\on{Vac}})\in \Shv_{\Nilp}(\Bun_G).$$

\sssec{}

Recall now that, according to \cite[Theorem 14.3.2]{AGKRRV1}, the category $\Shv_{\Nilp}(\Bun_G)$ is
acted on by $\QCoh(\LS^{\on{restr}}_\cG)$. 

\medskip

We define the functor
$$\BL^{\on{restr},L}_{G,\on{temp}}:\QCoh(\LS^{\on{restr}}_\cG)\to \Shv_{\Nilp}(\Bun_G)$$
to be given by the action of $\QCoh(\LS^{\on{restr}}_\cG)$ on $\on{Poinc}_{!,\Nilp}^{\on{Vac}}$.

\medskip

We will prove:

\begin{prop} \label{p:L coarse L comp}
The functor $\BL^{\on{restr},L}_{G,\on{temp}}$ preserves compactness. 
\end{prop} 

The proof will be given in \secref{ss:L coarse L comp}.

\sssec{}

As stated, \propref{p:L coarse L comp} says that the functor $\BL^{\on{restr},L}_{G,\on{temp}}$ sends compacts
to compacts, when viewed as a functor with values in $\Shv_{\Nilp}(\Bun_G)$. The following assertion was 
stated as \cite[Conjecture 14.1.8]{AGKRRV1}; we will prove it in this paper (see \secref{sss:proof of compactness}): 

\begin{thm} \label{t:compactness}
The embedding 
$$\on{emb.Nilp}:\Shv_{\Nilp}(\Bun_G)\hookrightarrow \Shv(\Bun_G)$$
preserves compactness. 
\end{thm} 

\medskip

Thus, \propref{p:L coarse L comp}, combined with \thmref{t:compactness}, say that the functor $\BL^{\on{restr},L}_{G,\on{temp}}$,
when viewed as taking values in $\Shv(\Bun_G)$, also preserves compactness. 

\sssec{}

By \propref{p:L coarse L comp}, the functor $\BL^{\on{restr},L}_{G,\on{temp}}$ admits a continuous
right adjoint, which we will denote by $\BL^{\on{restr}}_{G,\on{coarse}}$.

\medskip

The functor $\BL^{\on{restr},L}_{G,\on{temp}}$ is $\QCoh(\LS^{\on{restr}}_\cG)$-linear by construction.
Hence, the functor $\BL^{\on{restr}}_{G,\on{coarse}}$ acquires a structure of right-lax linearity with respect 
to $\QCoh(\LS^{\on{restr}}_\cG)$. 

\medskip

Note, however, since $\QCoh(\LS^{\on{restr}}_\cG)$ is \emph{semi-rigid} as a symmetric monoidal category (see \cite[Appendix C]{AGKRRV1}
for what this means), we obtain that the right-lax linear structure on $\BL^{\on{restr}}_{G,\on{coarse}}$ is actually strict. 

\sssec{}

We will prove:

\begin{thm} \label{t:L sends comp to even coconn}
The functor $\BL^{\on{restr}}_{G,\on{coarse}}$ sends compact objects in $\Shv_{\Nilp}(\Bun_G)$
to bounded below (a.k.a. eventually coconnective) objects in $\QCoh(\LS^{\on{restr}}_\cG)$.
\end{thm} 

The proof of this theorem will be given in \secref{ss:L sends comp to even coconn}. 

\ssec{The Whittaker coefficient functor}

\sssec{}

Let 
$$\on{coeff}^{\on{Vac}}:\Shv(\Bun_G)\to \Vect$$
be the functor co-represented by $\on{Poinc}_!^{\on{Vac}}$. 

\medskip

When $\on{char}(k)$ is positive, it is given by \eqref{e:coeff formula}.
A variant of this holds when $\on{char}(k)=0$ using the material in \cite[Sect. 3.3]{GLC1}.

\sssec{}

Recall the functor
$$\Gamma_!(\LS^{\on{restr}}_\cG,-):\QCoh(\LS^{\on{restr}}_\cG)\to \Vect,$$
see \cite[Sect. 7.7]{AGKRRV1}. 

\begin{rem} 

Explicitly, the functor $\Gamma_!(\LS^{\on{restr}}_\cG,-)$ fits into the commutative diagram 
$$
\CD
\IndCoh(\LS^{\on{restr}}_\cG) @>{(\Upsilon_{\LS^{\on{restr}}_\cG})^\vee}>> \QCoh(\LS^{\on{restr}}_\cG) \\
@V{\Gamma^\IndCoh(\LS^{\on{restr}}_\cG,-)}VV @VV{\Gamma_!(\LS^{\on{restr}}_\cG,-)}V \\
\Vect @>{\on{Id}}>> \Vect.
\endCD
$$

This diagram is valid for any laft formal algebraic stack. Note, however, that for quasi-smooth\footnote{In fact, 
an appropriate eventual connectivity assumption suffices.}
formal algebraic stacks (such as $\LS^{\on{restr}}_\cG$), the top horizontal arrow is a Verdier 
quotient. 

\end{rem}

\sssec{}

We will prove:

\begin{prop}  \label{p:coeff}
The composition
$$\Shv_{\Nilp}(\Bun_G)\overset{\BL^{\on{restr}}_{G,\on{coarse}}}\longrightarrow \QCoh(\LS^{\on{restr}}_\cG)
\overset{\Gamma_!(\LS^{\on{restr}}_\cG,-)}\longrightarrow \Vect$$
identifies canonically with
\begin{equation} \label{e:coeff}
\Shv_{\Nilp}(\Bun_G)\overset{\on{emb.Nilp}}\hookrightarrow \Shv(\Bun_G) \overset{\on{coeff}^{\on{Vac}}}\longrightarrow \Vect.
\end{equation} 
\end{prop}

The proposition will be proved in \secref{sss:proof coeff}.

\begin{rem}
Recall (see \cite[Sect. 7.6.1]{AGKRRV1}) that $\QCoh(\LS^{\on{restr}}_\cG)$ is canonically self-dual, so that under this sel-fduality the object
$\CO_{\LS^{\on{restr}}_\cG}\in \QCoh(\LS^{\on{restr}}_\cG)$ corresponds to the functor $\Gamma_!(\LS^{\on{restr}}_\cG,-)$.
In particular, a $\CO_{\LS^{\on{restr}}_\cG}$-linear functor
$$\bC\to \QCoh(\LS^{\on{restr}}_\cG)$$
(for a $\QCoh(\LS^{\on{restr}}_\cG)$-linear category $\bC$) is uniquely recovered from the composition 
$$\bC\to \QCoh(\LS^{\on{restr}}_\cG) \overset{\Gamma_!(\LS^{\on{restr}}_\cG,-)}\longrightarrow \Vect.$$

\medskip

From here here we obtain that \propref{p:coeff} gives rise to the following characterization of the functor $\BL^{\on{restr}}_{G,\on{coarse}}$:
it is the unique functor
$$\Shv_{\Nilp}(\Bun_G)\to \QCoh(\LS^{\on{restr}}_\cG)$$
that satisfies:

\begin{itemize}

\item It is $\QCoh(\LS^{\on{restr}}_\cG)$-linear;

\item Its composition with $\Gamma_!(\LS^{\on{restr}}_\cG,-)$ is isomorphic to \eqref{e:coeff}. 

\end{itemize} 

\end{rem} 

\ssec{Construction of the Langlands functor} \label{ss:actual functor}

\sssec{} \label{sss:emb temp spec}

Let $\bu^{\on{spec}}$ denote the functor
\begin{multline} \label{e:u spec}
\QCoh(\LS^{\on{restr}}_\cG)\overset{\Upsilon_{\LS^{\on{restr}}_\cG}}\simeq \IndCoh_{\{0\}}(\LS^{\on{restr}}_\cG)
\overset{-\otimes \fl_{\LS^{\on{restr}}_\cG}}\simeq 
\IndCoh_{\{0\}}(\LS^{\on{restr}}_\cG)\hookrightarrow \\
\hookrightarrow \IndCoh_{\Nilp}(\LS^{\on{restr}}_\cG),
\end{multline}
where:

\begin{itemize}

\item The first arrow is given by tensoring by the dualizing sheaf $\omega_{\LS^{\on{restr}}_\cG}$;
of $\LS^{\on{restr}}_\cG$;

\item $\fl_{\LS^{\on{restr}}_\cG}$ is the graded line bundle\footnote{In fact, this line bundle is constant, i.e., 
is essentially a graded line over $\sfe$, see Remark \ref{r:consant line} below.} $\det(T^*(\LS^{\on{restr}}_\cG)^{-1})[-2(g-1)\dim(G)]$.

\end{itemize} 

\medskip

In what follows we will denote the composition of the first two arrows in \eqref{e:u spec} by $\Xi_{\LS^{\on{restr}}_\cG}$.
This is a functor that makes sense for any quasi-smooth formal scheme (resp., algebrac stack) $\CZ$. 
If $\CZ$ is an
actual scheme (resp., algebraic stack), then $\Xi_\CZ$ is the tautological functor 
$$\QCoh(\CZ)\hookrightarrow \IndCoh(\CZ),$$
whose essential image is $\IndCoh_{\{0\}}(\CZ)$.

\begin{rem}

We use the notation $\bu^{\on{spec},R}$, rather than the more common one, namely, 
$\Psi_{\Nilp,\{0\}}$, in order to avoid the clash with the symbol for the nearby cycles functor. 

\end{rem}

\begin{rem} \label{r:consant line}

The second arrow in \eqref{e:u spec} is introduced in order to make this functor compatible with the one in the de Rham and Betti versions. 
Note also that $\LS^{\on{restr}}_\cG$ is symplectic (the symplectic structure is constructed using a choice of an invariant form on $\cg$)
of dimension $[2(g-1)\dim(G)]$. Hence, the line bundle $\det(T^*(\LS^{\on{restr}}_\cG))$ is canonically constant.

\end{rem} 

\sssec{}

The functor $\bu^{\on{spec}}$ preserves compactness and is fully faithful. Let $\bu^{\on{spec},R}$
denote its right adjoint. 

\medskip

The functor $\bu^{\on{spec}}$ is $\QCoh(\LS^{\on{restr}}_\cG)$-linear by construction. Hence, the functor
$\bu^{\on{spec},R}$ acquires a right-lax linear structure. By the semi-rigidity of $\QCoh(\LS^{\on{restr}}_\cG)$, this right-lax structure is actually
strict. 

\sssec{}

From \thmref{t:L sends comp to even coconn}, as in \cite[Corollary 1.6.5]{GLC1}, we obtain:

\begin{cor} \label{c:existence of Langlands functor}
There exists a continuous functor 
$$\BL^{\on{restr}}_G: \Shv_{\Nilp}(\Bun_G)\to \IndCoh_{\Nilp}(\LS^{\on{restr}}_\cG),$$
uniquely characterized by the following properties:

\medskip

\noindent{\em(i)} The functor $\BL^{\on{restr}}_G$ sends compact objects in $\Shv_{\Nilp}(\Bun_G)$
to eventually coconnective objects in $\IndCoh_\Nilp(\LS^{\on{restr}}_\cG)$, i.e., to $\IndCoh_\Nilp(\LS^{\on{restr}}_\cG)^{>-\infty}$;

\smallskip

\noindent{\em(ii)} $(\bu^{\on{spec}})^R\circ \BL^{\on{restr}}_G\simeq \BL^{\on{restr}}_{G,\on{coarse}}$.

\end{cor} 

Furthermore, as in \cite[Proposition 1.7.2]{GLC1}, we obtain:

\begin{lem}
The functor $\BL^{\on{restr}}_G$ carries a unique $\QCoh(\LS^{\on{restr}}_\cG)$-linear structure, so that
the induced $\QCoh(\LS^{\on{restr}}_\cG)$-linear structure on $(\bu^{\on{spec}})^R\circ \BL^{\on{restr}}_G$ is
the natural $\QCoh(\LS^{\on{restr}}_\cG)$-linear structure on $\BL^{\on{restr}}_{G,\on{coarse}}$.
\end{lem} 

\sssec{}

We are now ready to state the main result of this paper: 

\begin{mainthm} \label{t:ff GLC} \hfill

\smallskip

\noindent{\em(i)} 
The functor $\BL^{\on{restr}}_G$ factors via an equivalence 
$$\Shv_{\Nilp}(\Bun_G)\overset{\sim}\to \IndCoh_\Nilp({}'\!\LS^{\on{restr}}_\cG)\hookrightarrow \IndCoh_\Nilp({}'\!\LS^{\on{restr}}_\cG)$$
where $'\!\LS^{\on{restr}}_\cG$ is the union of some of the connected components of 
$\LS^{\on{restr}}_\cG$.

\smallskip

\noindent{\em(ii)} If $\on{char}(k)=0$, then the inclusion $'\!\LS^{\on{restr}}_\cG\subset \LS^{\on{restr}}_\cG$ is an equality. 

\smallskip

\noindent{\em(iii)} For any $k$ and $G=GL_n$, the inclusion $'\!\LS^{\on{restr}}_\cG\subset \LS^{\on{restr}}_\cG$ is an equality. 

\end{mainthm} 

Of course, we believe that the statement of \thmref{t:ff GLC}(i) can be strengthened:

\begin{conj} \label{c:GLC}
The inclusion $'\!\LS^{\on{restr}}_\cG\subset \LS^{\on{restr}}_\cG$ is always as equality.
\end{conj}

\ssec{Langlands functor and Eisenstein series} \label{ss:Eis}

\sssec{}

Let $P^-$ be a standard (negative) parabolic in $G$ and let $M$ be its Levi quotient. Consider the Eisenstein
functor
$$\Eis^-_!:\Shv(\Bun_M)\to \Shv(\Bun_G),$$
see \cite[Sect. 8.1]{GLC3}. 

\medskip

Note that according to the conventions of \cite[Sect. 8.1.3]{GLC3}, the definition of $\Eis^-_!$
includes a cohomological shift, see \eqref{e:Eis formula}.

\medskip

We claim:

\begin{prop} \label{p:Eis preserve Nilp}
The functor $\Eis^-_!$ sends $\Shv_\Nilp(\Bun_M)$ to $\Shv_\Nilp(\Bun_G)$.
\end{prop}

\begin{rem}

The proof given below uses a spectral description of the subcategory $\Shv_\Nilp(\Bun_G)$.
One can, however, give a purely geometric argument proving \propref{p:Eis preserve Nilp}:
Namely, one can estimate the singular support of objects of the form $\Eis^-_!(\CF)$ using 
the following assertion:

\medskip

\noindent{\it The singular support of $\jmath_!(\ul\sfe_{\Bun_{P^-}})\in \Shv(\ol\Bun_{P^-})$
is contained in the union of the conormal of the strata (for the natural stratification of 
$\ol\Bun_{P^-}$).}

\medskip

This assertion can be proved using Zastava spaces in a way similar to the manipulation
involved in the proof of \thmref{t:IC acycl} below.

\end{rem}

\begin{proof}

Let $\CZ$ be a prestack over $\sfe$ mapping to $\LS^{\on{restr}}_\cM$. Let
$$\on{Hecke}(\CZ,\Shv(\Bun_M))$$
be the corresponding category of Hecke eigensheaves, see \cite[Sect. 15.2]{AGKRRV1}.

\medskip

Recall that according to \cite{BG2} (the case of $P=B)$ and its generalization to an arbitrary parabolic in
\cite[Theorem 1.6.5.2]{FH}, the functor
\begin{equation} \label{e:Eis on Hecke one}
\on{Hecke}(\CZ,\Shv(\Bun_M))\overset{\oblv_{\on{Hecke}}}\longrightarrow
\QCoh(\CZ)\otimes \Shv(\Bun_M) \overset{\on{Id}\otimes \Eis^-_!}\longrightarrow 
\QCoh(\CZ)\otimes \Shv(\Bun_G)
\end{equation} 
factors canonically as
\begin{multline} \label{e:Eis on Hecke two}
\on{Hecke}(\CZ,\Shv(\Bun_M))\overset{(\sfq^{-,\on{spec}})^*\otimes \on{Id}}\longrightarrow 
\on{Hecke}(\LS^{\on{restr}}_{\cP^-}\underset{\LS^{\on{restr}}_\cM}\times \CZ,\Shv(\Bun_M))
\overset{\on{Hecke}(\CZ,\Eis^-_!)}\longrightarrow \\
\to \on{Hecke}(\LS^{\on{restr}}_{\cP^-}\underset{\LS^{\on{restr}}_\cM}\times \CZ,\Shv(\Bun_G)) 
\overset{\oblv_{\on{Hecke}}}\longrightarrow  
\QCoh(\LS^{\on{restr}}_{\cP^-}\underset{\LS^{\on{restr}}_\cM}\times \CZ)\otimes \Shv(\Bun_G)
\overset{(\sfq^{-,\on{spec}})_*\otimes \on{Id}}\longrightarrow \\
\to \QCoh(\CZ)\otimes \Shv(\Bun_G).
\end{multline} 

Take $\CZ=\LS^{\on{restr}}_\cM$, and recall that in this case the composition 
\begin{equation} \label{e:Nilp as Hecke}
\on{Hecke}(\LS^{\on{restr}}_\cM,\Shv(\Bun_M))\overset{\oblv_{\on{Hecke}}}\longrightarrow
\QCoh(\LS^{\on{restr}}_\cM)\otimes \Shv(\Bun_M) \overset{\Gamma_!(\LS^{\on{restr}}_\cM,-)\otimes \on{Id}}\longrightarrow 
\Shv(\Bun_M)
\end{equation}
is fully faithful with essential image $\Shv_\Nilp(\Bun_M)$ (see \cite[Proposition 15.5.3(a)]{AGKRRV1}). 

\medskip 

Under the identification \eqref{e:Nilp as Hecke}, the original functor $\Eis^-_!$ identifies with the 
composition of \eqref{e:Eis on Hecke two} with
$$\QCoh(\LS^{\on{restr}}_\cM)\otimes \Shv(\Bun_G) \overset{\Gamma_!(\LS^{\on{restr}}_\cM,-)\otimes \on{Id}}
\longrightarrow \Shv(\Bun_G).$$

\medskip

The desired assertion follows now from the fact that for any $\CZ'$, the inclusion
$$\on{Hecke}(\CZ',\Shv_\Nilp(\Bun_G))\hookrightarrow \on{Hecke}(\CZ',\Shv(\Bun_G))$$
is an equality, see again \cite[Proposition 15.5.3(a)]{AGKRRV1}.

\end{proof} 

\sssec{}

Consider the diagram
$$\LS^{\on{restr}}_\cG  \overset{\sfp^{-,\on{spec}}}\longleftarrow 
\LS^{\on{restr}}_{\cP^-} \overset{\sfq^{-,\on{spec}}}\longrightarrow \LS^{\on{restr}}_{\cM},$$
where we note that the morphism $\sfq^{-,\on{spec}}$ is a relative algebraic stack and is quasi-smooth. 

\medskip

The spectral Eisenstein functor 
$$\Eis^{-,\on{spec}}:\IndCoh(\LS^{\on{restr}}_\cM)\to \IndCoh(\LS^{\on{restr}}_\cG).$$
is defined to be 
$$(\sfp^{-,\on{spec}})^\IndCoh_*\circ (\sfq^{-,\on{spec}})^{\IndCoh,*}.$$

\medskip

As in \cite[Proposition 13.2.6]{AG1}, one shows that the functor $\Eis^{-,\on{spec}}$ sends 
$$\IndCoh_\Nilp(\LS^{\on{restr}}_\cM)\to \IndCoh_\Nilp(\LS^{\on{restr}}_\cG).$$

\sssec{}

In \secref{s:L and Eis} we will prove:

\begin{thm} \label{t:L and Eis}
The diagram
\begin{equation} \label{e:L and Eis}
\CD
\Shv_\Nilp(\Bun_M) @>{\BL^{\on{restr}}_M}>> \IndCoh_\Nilp(\LS^{\on{restr}}_\cM) \\
@V{\Eis^-_{!,\rho_P(\omega_X)}[\delta_{(N^-_P)_{\rho_P(\omega_X)}}]}VV  @VV{\Eis^{-,\on{spec}}}V \\
 \Shv_\Nilp(\Bun_G) @>>{\BL^{\on{restr}}_G}> \IndCoh_\Nilp(\LS^{\on{restr}}_\cG)
\endCD
\end{equation} 
commutes, where:

\begin{itemize}

\item The functor $\Eis^-_{!,\rho_P(\omega_X)}$ is the precomposition of $\Eis^-_!$ with the translation functor 
by $\rho_P(\omega_X)\in \Bun_{Z_M}$;

\item $\delta_{(N^-_P)_{\rho_P(\omega_X)}}=\dim(\Bun_{(N^-_P)_{\rho_P(\omega_X)}})$, see \cite[Theorem 10.1.2]{GLC3}.

\end{itemize}

\end{thm}

\ssec{Consequences for the classical theory of automorphic functions}

\sssec{}

We now specialize to the case when $k=\ol\BF_q$, but $X$ and $G$ (and hence also $\Bun_G$)
are defined over $\BF_q$. The stack $\LS_\cG^{\on{restr}}$ carries an automorphism induced by
the geometric Frobenius on $X$.

\medskip

The following assertion was stated in \cite{AGKRRV1} as a corollary of {\it loc. cit.}, Conjecture 24.6.9, and 
the latter was proved in \cite[Sect. 6.4.13]{BLR}: 

\begin{thm} \label{t:Tr Frob spec}
The inclusion
$$\IndCoh_\Nilp(\LS^{\on{restr}}_\cG)\hookrightarrow \IndCoh(\LS^{\on{restr}}_\cG)$$
induces an \emph{isomorphism}
$$\Tr(\Frob,\IndCoh_\Nilp(\LS^{\on{restr}}_\cG)) \overset{\sim}\to \Tr(\Frob,\IndCoh(\LS^{\on{restr}}_\cG)).$$
\end{thm}

\sssec{}

A standard computation implies that 
$$\Tr(\Frob,\IndCoh(\LS^{\on{restr}}_\cG)) \simeq \Gamma^\IndCoh((\LS^{\on{restr}}_\cG)^\Frob,\omega_{(\LS^{\on{restr}}_\cG)^\Frob}).$$

Recall the notation
$$\LS^{\on{arithm}}_\cG:=(\LS^{\on{restr}}_\cG)^\Frob.$$

According to \cite[Theorem 24.1.4]{AGKRRV1}, $\LS^{\on{arithm}}_\cG$ is a quasi-compact algebraic stack locally
almost of finite type. 

\medskip

Hence, \thmref{t:Tr Frob spec} implies that we have a canonical isomorphism
\begin{equation} \label{e:Tr Frob spec}
\Tr(\Frob,\IndCoh_\Nilp(\LS^{\on{restr}}_\cG)) \simeq \Gamma(\LS^{\on{arithm}}_\cG,\omega_{\LS^{\on{arithm}}_\cG}),
\end{equation}
where by a slight abuse of notation, we denote by the same symbol 
$$\omega_{\LS^{\on{arithm}}_\cG}\in \QCoh(\LS^{\on{arithm}}_\cG)$$
the image of
$$\omega_{\LS^{\on{arithm}}_\cG}\in \IndCoh(\LS^{\on{arithm}}_\cG)$$
under the functor
$$(\Upsilon_{\LS^{\on{arithm}}_\cG})^\vee:\IndCoh(\LS^{\on{arithm}}_\cG)\to \QCoh(\LS^{\on{arithm}}_\cG).$$

\sssec{}

Let
$$'\!\LS^{\on{restr}}_\cG\subset \LS^{\on{restr}}_\cG$$
be as in \thmref{t:ff GLC}(i). This is a $\Frob$-invariant substack. Set
$$'\!\LS^{\on{arithm}}_\cG:=({}'\!\LS^{\on{restr}}_\cG)^\Frob.$$

It formally follows from \thmref{t:Tr Frob spec} that we have
\begin{equation} \label{e:Tr Frob spec prime}
\Tr(\Frob,\IndCoh_\Nilp({}'\!\LS^{\on{restr}}_\cG)) \simeq \Gamma({}'\!\LS^{\on{arithm}}_\cG,\omega_{'\!\LS^{\on{arithm}}_\cG}).
\end{equation}

\sssec{}

Combining isomorphism \eqref{e:Tr Frob spec prime} with \thmref{t:ff GLC}(i) and \cite[Theorem 0.2.6]{AGKRRV3}, we obtain:

\begin{cor} \label{c:main}
There exists a canonical isomorphism
$$\on{Funct}_c(\Bun_G(\BF_q),\ol\BQ_\ell)\simeq \Gamma({}'\!\LS^{\on{arithm}}_\cG,\omega_{'\!\LS^{\on{arithm}}_\cG}).$$
\end{cor} 

\sssec{}

From \corref{c:main} we obtain the following:

\begin{cor} \label{c:mult one}
Let $G$ be semi-simple, and let $\sigma$ be an irreducible Weil $\cG$-local system. Then the space of automorphic
functions on which the algebra of excursion operators acts by the character corresponding
to $\sigma$ is at most one-dimensional (and in the latter case is spanned by a cuspidal 
function). 
\end{cor} 

\begin{proof}

Recall that by \cite[Theorem 24.1.6]{AGKRRV1}, an irreducible local system $\sigma$
gives rise to a connected component $\LS^{\on{arithm}}_{\cG,\sigma}$ isomorphic to
$\on{pt}/\on{Aut}(\sigma)$, where $\on{Aut}(\sigma)$ is a finite group. In particular\footnote{Note also that $\LS^{\on{arithm}}_\cG$ is canonically Calabi-Yau, 
i.e., the determinant line bundle of its cotangent bundle is trivialized. Hence, the restriction of $\omega_{\LS^{\on{arithm}}_\cG}$
to the quasi-smooth locus of $\LS^{\on{arithm}}_\cG$ is canonically isomorphic to the restriction of 
$\CO_{\LS^{\on{arithm}}_\cG}$.},
$\omega_{\LS^{\on{arithm}}_{\cG,\sigma}}\simeq \CO_{\LS^{\on{arithm}}_{\cG,\sigma}}$. 

\medskip

Recall now that the excursion algebra identifies with
$$H^0(\Gamma(\LS^{\on{arithm}}_\cG,\CO_{\LS^{\on{arithm}}_\cG})).$$

\medskip

The map 
$$\LS^{\on{arithm}}_{\cG,\sigma}\to \Spec(\Gamma(\LS^{\on{arithm}}_\cG,\CO_{\LS^{\on{arithm}}_{\cG,\sigma}})),$$
induced by 
$$\LS^{\on{arithm}}_\cG\to \Spec(\Gamma(\LS^{\on{arithm}}_\cG,\CO_{\LS^{\on{arithm}}_\cG})),$$
identifies with 
$$\on{pt}/\on{Aut}(\sigma)\to \on{pt}.$$

\medskip

This makes the assertion obvious. 

\end{proof}

\begin{rem}
If \conjref{c:GLC} holds, then $'\!\LS^{\on{arithm}}_\cG$ is all of $\LS^{\on{arithm}}_\cG$. 

\medskip

In particular, in this case in \corref{c:mult one}, the corresponding eigenspace is
exactly one-dimensional. 

\end{rem}

\ssec{Proof of Propositions \ref{p:L coarse L comp} and \ref{p:coeff}} \label{ss:L coarse L comp}

\sssec{}

Write $\LS^{\on{restr}}_\cG$ as a union of its connected components
$$\LS^{\on{restr}}_\cG=\underset{\alpha}\sqcup\, \CZ_\alpha.$$

As in \cite[Sect. 21.1]{AGKRRV1}, we can write each $\CZ_\alpha$ as a (countable) colimit
$$\CZ_\alpha\simeq \underset{n}{\on{colim}}\, Z_{\alpha,n},$$
where:

\begin{itemize}

\item Each $Z_{\alpha,n}$ is a quasi-smooth algebraic stack;

\smallskip

\item Each $i_{\alpha,n}:Z_{\alpha,n}\to \CZ_\alpha$ is a \emph{regular} closed embedding that 
induces an isomorphism at the reduced level.

\end{itemize}

\sssec{}

The category $\QCoh(\CZ_\alpha)$ is compactly generated by objects of the form
$$(i_{\alpha,n})_*(\CO_{Z_{\alpha,n}})\otimes \CE,$$
where $\CE$ is a dualizable object in $\QCoh(\LS^{\on{restr}}_\cG)$. 

\medskip

In fact, we can take $\CE$ to be the pullback of a vector bundle under the evaluation map
$$\on{ev}_x:\LS^{\on{restr}}_\cG\to \on{pt}/\cG$$
corresponding to some chosen point $x\in X$. 

\sssec{}

We now begin the proof of \propref{p:L coarse L comp}. 

\medskip

In the above notations, it suffices to show that the functor
$\BL^{\on{restr},L}_{G,\on{temp}}$ sends each $(i_{\alpha,n})_*(\CO_{Z_{\alpha,n}})$ to
a compact object of $\Shv_\Nilp(\Bun_G)$.

\sssec{}

Consider the category
$$\QCoh(Z_{\alpha,n})\underset{\QCoh(\LS^{\on{restr}}_\cG)}\otimes \Shv_\Nilp(\Bun_G).$$

We have a pair of $\QCoh(\LS^{\on{restr}}_\cG)$-linear functors
$$((i_{\alpha,n})_*\otimes \on{Id}):
\QCoh(Z_{\alpha,n})\underset{\QCoh(\LS^{\on{restr}}_\cG)}\otimes \Shv_\Nilp(\Bun_G)\rightleftarrows \Shv_\Nilp(\Bun_G):
((i_{\alpha,n})^!\otimes \on{Id}).$$

\sssec{} \label{sss:n-projector}

Let 
$$\sP^{\on{enh}}_{Z_{\alpha,n}}:\Shv(\Bun_G)\to \QCoh(Z_{\alpha,n})\underset{\QCoh(\LS^{\on{restr}}_\cG)}\otimes \Shv_\Nilp(\Bun_G).$$
be as in \cite[Sect. 15.3.2]{AGKRRV1}. I.e., this is the left adjoint to forgetful functor
$$\QCoh(Z_{\alpha,n})\underset{\QCoh(\LS^{\on{restr}}_\cG)}\otimes \Shv_\Nilp(\Bun_G)
\overset{((i_{\alpha,n})_*\otimes \on{Id})}\longrightarrow \Shv_\Nilp(\Bun_G)\overset{\on{emb.Nilp}}\hookrightarrow \Shv(\Bun_G),$$
see \cite[Corollary 13.5.4]{AGKRRV1}.

\medskip

Unwinding the construction (see \cite[Sect. 15]{AGKRRV1}), we obtain that
$$\BL^{\on{restr},L}_{G,\on{temp}}((i_{\alpha,n})_*(\CO_{Z_{\alpha,n}}))\simeq
((i_{\alpha,n})_*\otimes \on{Id})(\sP^{\on{enh}}_{Z_{\alpha,n}}(\on{Poinc}_!^{\on{Vac}})).$$

\sssec{}

Now the object $\sP^{\on{enh}}_{Z_{\alpha,n}}(\on{Poinc}_!^{\on{Vac}})$ is compact in 
$\QCoh(Z_{\alpha,n})\underset{\QCoh(\LS^{\on{restr}}_\cG)}\otimes \Shv_\Nilp(\Bun_G)$
(being the value of a left adjoint on a compact object). 

\medskip

Finally, the functor $(i_{\alpha,n})_*\otimes \on{Id}$ preserves compactness since it admits a continuous right 
adjoint.

\qed[\propref{p:L coarse L comp}]

\sssec{Proof of \propref{p:coeff}} \label{sss:proof coeff} 

Note that 
$$\Gamma_!(\QCoh(\LS^{\on{restr}}_\cG,-)\simeq \underset{\alpha}\oplus\, \Gamma_!(\CZ_\alpha,(-)|_{\CZ_\alpha}),$$
while the functor 
$$\Gamma_!(\CZ_\alpha,-):\QCoh(Z_\alpha)\to \Vect$$
can be written as
$$\underset{n}{\on{colim}}\, \CHom_{\QCoh(\LS^{\on{restr}}_\cG)}((i_{\alpha,n})_*(\CO_{Z_{\alpha,n}}),-).$$

Thus, we can rewrite the composition 
$$\Shv_{\Nilp}(\Bun_G)\overset{\BL^{\on{restr}}_{G,\on{coarse}}}\longrightarrow \QCoh(\LS^{\on{restr}}_\cG)
\overset{\Gamma_!(\CZ_\alpha,-)}\longrightarrow \Vect$$
as
$$\underset{n}{\on{colim}}\, \CHom_{\Shv_\Nilp(\Bun_G)}((i_{\alpha,n})_*(\CO_{Z_{\alpha,n}}),\BL^{\on{restr}}_{G,\on{coarse}}(-)),$$
and hence by adjunction as 
$$\underset{n}{\on{colim}}\, \CHom_{\Shv_\Nilp(\Bun_G)}(\BL^{\on{restr},L}_{G,\on{temp}}\circ (i_{\alpha,n})_*(\CO_{Z_{\alpha,n}}),-).$$

By \secref{sss:n-projector}, we can rewrite the latter expression as
$$\underset{n}{\on{colim}}\, \CHom_{\Shv_\Nilp(\Bun_G)}\left(((i_{\alpha,n})_*\otimes \on{Id})(\sP^{\on{enh}}_{Z_{\alpha,n}}(\on{Poinc}_!^{\on{Vac}})),-\right),$$
and further as 
$$\underset{n}{\on{colim}}\, \CHom_{\QCoh(Z_{\alpha,n})\underset{\QCoh(\LS^{\on{restr}}_\cG)}\otimes\Shv_\Nilp(\Bun_G)}
\left(\sP^{\on{enh}}_{Z_{\alpha,n}}(\on{Poinc}_!^{\on{Vac}}),((i_{\alpha,n})^!\otimes \on{Id})(-)\right),$$
and again by adjunction
$$\underset{n}{\on{colim}}\, \CHom_{\Shv(\Bun_G)}\left(\on{Poinc}_!^{\on{Vac}},
\on{emb.Nilp}\circ ((i_{\alpha,n})_*\otimes \on{Id})\circ ((i_{\alpha,n})^!\otimes \on{Id})(-)\right).$$

Since $\on{Poinc}_!^{\on{Vac}}$ is compact, the latter expression identifies with
$$\CHom_{\Shv(\Bun_G)}\left(\on{Poinc}_!^{\on{Vac}},\on{emb.Nilp}\left(\underset{n}{\on{colim}}\, 
((i_{\alpha,n})_*\otimes \on{Id})\circ ((i_{\alpha,n})^!\otimes \on{Id})(-)\right)\right).$$

The required assertion follows now from the fact that the natural transformation 
$$\underset{\alpha}\oplus\, \underset{n}{\on{colim}}\, 
((i_{\alpha,n})_*\otimes \on{Id})\circ ((i_{\alpha,n})^!\otimes \on{Id})\to \on{Id}$$
on $\Shv_\Nilp(\Bun_G)$ is an isomorphism.

\qed[\propref{p:coeff}]

\ssec{Proof of \thmref{t:L sends comp to even coconn}} \label{ss:L sends comp to even coconn}

The proof will largely follow \cite[Sect. 2]{GLC1}. 

\sssec{}

For the proof we will assume the validity of \thmref{t:compactness}, which will be proved independently. 

\medskip

Assuming this theorem and using \cite[Sect. 2.2]{GLC2},  we obtain that if $\CM\in \Shv_\Nilp(\Bun_G)$ is compact, then
$$\on{emb.Nilp}(\CM)\in \Shv(\Bun_G)$$
is bounded below.

\medskip

Hence, it is enough to show that the functor $\BL^{\on{restr}}_{G,\on{coarse}}$ has a cohomological amplitude bounded 
on the left, i.e., there exists an integer $d$ such that $\BL^{\on{restr}}_{G,\on{coarse}}[-d]$ is left t-exact.

\sssec{}

Choose a point $x\in X$, and set
$$\LS_\cG^{\on{restr},\on{rigid}_x}:=\LS_\cG^{\on{restr}}\underset{\on{pt}/\cG}\times \on{pt},$$
where $\LS_\cG^{\on{restr}}\to \on{pt}/\cG$ is the evaluation map $\on{ev}_x$. 

\medskip

According to \cite[Theorem 1.6.3]{AGKRRV1}, the prestack $\LS_\cG^{\on{restr},\on{rigid}_x}$ is a disjoint union 
of formal affine schemes. In particular, the functor
$$\Gamma_!(\LS_\cG^{\on{restr},\on{rigid}_x},-):\QCoh(\LS_\cG^{\on{restr},\on{rigid}_x})\to \Vect$$
is t-exact and conservative. 

\sssec{}

Note also that we have a canonical isomorphism 
$$\Gamma_!(\LS_\cG^{\on{restr},\on{rigid}_x},\pi^*(-))\simeq
\Gamma_!(\LS_\cG^{\on{restr}},\on{ev}_x^*(R_\cG)\otimes (-)),$$
where:

\begin{itemize}

\item $\pi$v denotes the projection $\LS_\cG^{\on{restr},\on{rigid}_x}\to \LS_\cG^{\on{restr}}$;

\item $R_\cG\in \Rep(\cG)\simeq \QCoh(\on{pt}/\cG)$ is the regular representation.

\end{itemize}

\medskip

Hence, we obtain that it is enough to show that the functors
$$\Gamma_!\left(\LS_\cG^{\on{restr}},\on{ev}_x^*(V)\otimes \BL^{\on{restr}}_{G,\on{coarse}}(-)\right), \quad V\in \Rep(\cG)^\heartsuit.$$
have cohomological amplitudes uniformly bounded on the left. 

\sssec{}

By \propref{p:coeff}, we rewrite the above functor as 
\begin{equation} \label{e:coeff Hecke}
\on{coeff}^{\on{Vac}}\circ \on{H}_{V,x}(-),
\end{equation}
where $\on{H}_{V,x}$ is the Hecke endofunctor of $\Shv(\Bun_G)$ corresponding to the chosen $x$ and $V$. 

\medskip

We will show that the functors \eqref{e:coeff Hecke} (for $V\in \Rep(\cG)$) 
have cohomological amplitudes uniformly bounded on the left on all of $\Shv(\Bun_G)$ (and not just $\Shv_\Nilp(\Bun_G)$). 

\sssec{}

Consider the \emph{tempered subcategory}
$$\Shv(\Bun_G)_{\on{temp}}\overset{\bu}\hookrightarrow \Shv(\Bun_G)$$
as defined in \cite[Sect. 7]{FR} (using the choice of $x\in X$). The embedding $\bu$ admits a continuous right adjoint, denoted $\bu^R$. 

\medskip

The category $\Shv(\Bun_G)_{\on{temp}}$ carries a uniquely defined t-structure for which the functor $\bu^R$ is 
t-exact, see \cite[Sect. 7.2]{FR}.

\medskip

Another key feature of this t-structure is that the Hecke functors $\on{H}_{V,x}$ descend to $\Shv(\Bun_G)_{\on{temp}}$ 
and are t-exact on it (for $V\in \Rep(\cG)^\heartsuit$), see \cite[Theorem 7.1.0.1]{FR}.

\sssec{}

Note that since $\on{Poinc}_!^{\on{Vac}}\in \Shv(\Bun_G)_{\on{temp}}$, the functor $\on{coeff}^{\on{Vac}}$ factors 
canonically as
$$\Shv(\Bun_G) \overset{\bu^R}\to  \Shv(\Bun_G)_{\on{temp}} \overset{\on{coeff}^{\on{Vac}}_{\on{temp}}}\longrightarrow
\Vect.$$

Hence, it is enough to show that the functors
$$\on{coeff}^{\on{Vac}}_{\on{temp}}\circ \on{H}_{V,x}(-):\Shv(\Bun_G)_{\on{temp}} \to \Vect$$
have cohomological amplitudes uniformly bounded on the left.

\sssec{}

By the t-exactness of the Hecke action on $\Shv(\Bun_G)_{\on{temp}}$, it suffices to show that the functor 
$$\on{coeff}^{\on{Vac}}_{\on{temp}}:\Shv(\Bun_G)_{\on{temp}} \to \Vect$$
has a cohomological amplitude bounded on the left.

\medskip 

This is in turn equivalent to the functor 
$$\on{coeff}^{\on{Vac}}:\Shv(\Bun_G) \to \Vect$$
having a cohomological amplitude bounded on the left, which is obvious. 

\qed[\thmref{t:L sends comp to even coconn}]

\section{Proof of \thmref{t:ff GLC} in characteristic \texorpdfstring{$0$}{000}} \label{s:char 0}

In this section we will show that when $\on{char}(k)=0$, the functor 
$$\Shv_{\Nilp}(\Bun_G) \overset{\BL^{\on{restr}}_G}\to \IndCoh_\Nilp(\LS^{\on{restr}}_\cG)$$
is an equivalence.

\ssec{Constructible Betti geometric Langlands} \label{ss:Betti constr}

In this subsection we take $k$ to be the field of complex numbers $\BC$, and we will work with the sheaf theory
denoted $\Shv^{\on{Betti,constr}}(-)$ of (ind-)constructible Betti sheaves with coefficients in an arbitrary field $\sfe$
of characteristic $0$. 

\sssec{}

Let $\LS^{\on{Betti,restr}}_\cG$ be the moduli space of local systems with restricted variation,
defined using the constructible sheaf theory $\Shv^{\on{Betti,constr}}(-)$. 

\medskip

The material in Sects. \ref{ss:coarse}-\ref{ss:Eis}
applies verbatim and we obtain a functor 
$$\BL_G^{\on{Betti,restr}}:\Shv^{\on{Betti,constr}}_\Nilp(\Bun_G)\to \IndCoh_{\Nilp}(\LS^{\on{Betti,restr}}_\cG).$$

\sssec{}

We claim:

\begin{thm} \label{t:Betti restr}
The functor $\BL_G^{\on{Betti,restr}}$ is an equivalence.
\end{thm}

\begin{proof}

The above functor $\BL_G^{\on{Betti,restr}}$ is the same as the functor denoted by the same symbol 
in \cite[Sect. 3.5.3]{GLC1}.

\medskip

Now the assertion follows from the validity of the full de Rham version of the geometric Langlands conjecture
(proved in \cite{GLC5}) combined with \cite[Theorem 3.5.6]{GLC1}.

\end{proof} 

\ssec{Betti vs \'etale comparison} \label{ss:over C}

In this subsection we continue to assume that $k=\BC$, and we will take $\sfe:=\ol\BQ_\ell$.
We will work with two sheaf theories:

\medskip

One is $\Shv^{\on{et}}(-)=:\Shv(-)$ of  (ind-)constructible $\ol\BQ_\ell$-adic \'etale sheaves,
considered in \secref{s:l-adic}. 

\medskip

The other is $\Shv^{\on{Betti,constr}}(-)$ considered in \secref{ss:Betti constr} above. 

\medskip

We will show that the validity of \thmref{t:Betti restr} implies the validity of \thmref{t:ff GLC}(ii) 
(for $k=\BC$). 

\sssec{} \label{sss:et to Betti}

Note that we have a fully faithful natural transformation between the sheaf theories
\begin{equation} \label{e:et to Betti}
(\on{et}\to \on{Betti}):\Shv^{\on{et}}(-)\to \Shv^{\on{Betti,constr}}(-)
\end{equation} 
that commutes with both !- and *- direct and inverse images, see 
\cite[Theorems XI.4.4, XVI.4.1 and XVII.5.3.3]{SGA4}.

\medskip

At the level of compact objets (i.e., constructible sheaves) on affine schemes, its essential image 
is characterized as follows:

\medskip

Let $\CF$ be an object of $\Shv^{\on{Betti,constr}}(Y)^c$, and let $Y_\alpha$ be a decomposition of $Y$ 
into smooth locally closed subsets, such that the restrictions (either *- or !-) of $\CF$ to the strata are
lisse. Consider the individual cohomology sheaves $H^i(\CF|_{Y_\alpha})$ as representations $V_{\alpha,i}$
of $\pi_1(Y_\alpha)$ (for some chosen base point). 

\medskip

Then $\CF$ lies in the essential image of $(\on{et}\to \on{Betti}))_Y$ if and only if each $V_{\alpha,i}$
admits a $\pi_1(Y_\alpha)$-invariant lattice with respect to $\CO_E\subset E\subset \ol\BQ_\ell$ for
a finite extension $E\supseteq \BQ_\ell$. 

\sssec{} \label{sss:et to Betti BunG}

In particular, the natural transformation \eqref{e:et to Betti} induces a fully 
faithful functor
$$(\on{et}\to \on{Betti})_{\Bun_G}:\Shv^{\on{et}}(\Bun_G)\to \Shv^{\on{Betti,constr}}(\Bun_G),$$
which restricts to a (fully faithful) functor
$$\Shv^{\on{et}}_\Nilp(\Bun_G)\to \Shv^{\on{Betti,constr}}_\Nilp(\Bun_G).$$

\sssec{} \label{sss:compactness char 0}

Recall that \thmref{t:compactness} has been proved for 
$$\Shv^{\on{Betti,constr}}_\Nilp(\Bun_G)\hookrightarrow \Shv^{\on{Betti,constr}}(\Bun_G)$$ in
\cite[Theorem 16.4.10]{AGKRRV1}. Hence, it holds for 
$$\Shv_\Nilp^{\on{et}}(\Bun_G)\hookrightarrow  \Shv^{\on{et}}(\Bun_G)$$ as well. 

\sssec{}

Let
$$\LS^{\on{et,restr}}_\cG \text{ and } \LS^{\on{Betti,restr}}_\cG$$
be the two versions of the be the moduli space of local systems with restricted variation. 

\medskip

The functor
$$(\on{et}\to \on{Betti})_X:\on{QLisse}^{\on{et}}(X)\to \on{QLisse}(X)^{\on{Betti,constr}}$$
is symmetric monoidal, and hence induces a map between the corresponding moduli spaces
\begin{equation} \label{e:et to Betti LS}
(\on{et}\to \on{Betti})_{\LS_\cG}:\LS^{\on{et,restr}}_\cG\to \LS^{\on{Betti,restr}}_\cG.
\end{equation} 

\medskip

We claim:

\begin{lem} \label{l:et to Betti LS}
The map \eqref{e:et to Betti LS} factors through an isomorphism from $\LS^{\on{et,restr}}_\cG$
to the disjoint union of some of the connected components of $\LS^{\on{Betti,restr}}_\cG$.
\end{lem} 

\begin{proof}

Follows from the description of essential image of the natural transformation 
\eqref{e:et to Betti} in \secref{sss:et to Betti}, see \cite[Sect. 9.5.8]{AGKRRV1}.

\end{proof} 

\begin{rem}
An assertion parallel to \lemref{l:et to Betti LS} holds for any full symmetric monoidal 
subcategory of $\on{Lisse}(X)^\heartsuit$.
\end{rem} 

\sssec{}

Thus, we can regard
$$\QCoh(\LS^{\on{et,restr}}_\cG)\underset{\QCoh(\LS^{\on{Betti,restr}}_\cG)}\otimes \Shv^{\on{Betti,constr}}_\Nilp(\Bun_G)$$
as a full subcategory (in fact, a direct summand) of $\Shv^{\on{Betti,constr}}_\Nilp(\Bun_G)$, and it is easy to see
that the functor 
$$(\on{et}\to \on{Betti})_{\Bun_G}:\Shv^{\on{et}}_\Nilp(\Bun_G)\to \Shv^{\on{Betti,constr}}_\Nilp(\Bun_G)$$
lands in it: this follows from the fact that this functor is compatible with the actions of
$\QCoh(\LS^{\on{et,restr}}_\cG)$, viewed as a direct factor of $\QCoh(\LS^{\on{Betti,restr}}_\cG$ on the two sides. 

\medskip

Similarly, direct image along $(\on{et}\to \on{Betti})_{\LS_\cG}$ is an equivalence
$$\IndCoh_\Nilp(\LS^{\on{et,restr}}_\cG)\overset{\sim}\to 
\QCoh(\LS^{\on{et,restr}}_\cG)\underset{\QCoh(\LS^{\on{Betti,restr}}_\cG)}\otimes \IndCoh_\Nilp(\LS^{\on{restr}}_\cG).$$

\medskip

Furthermore, it follows from \propref{p:coeff} that we have a commutative diagram
$$
\CD
\Shv^{\on{et}}_\Nilp(\Bun_G) @>{(\on{et}\to \on{Betti})_{\Bun_G}}>> 
\QCoh(\LS^{\on{et,restr}}_\cG)\underset{\QCoh(\LS^{\on{Betti,restr}}_\cG)}\otimes \Shv^{\on{Betti,constr}}_\Nilp(\Bun_G) \\
@V{\BL_G^{\on{et,restr}}}VV  @V{\sim}V{\on{Id}\otimes \BL_G^{\on{Betti,restr}}}V \\
\IndCoh_\Nilp(\LS^{\on{et,restr}}_\cG) @>{(\on{et}\to \on{Betti})_{\LS_\cG}}>{\sim}>
\QCoh(\LS^{\on{et,restr}}_\cG)\underset{\QCoh(\LS^{\on{Betti,restr}}_\cG)}\otimes \IndCoh_\Nilp(\LS^{\on{Betti,restr}}_\cG),
\endCD
$$
in which the top horizontal arrow is fully faithful.

\medskip

Hence, we obtain that the functor $\BL_G^{\on{et,restr}}$ is fully faithful. Thus, in order to prove that it is an equivalence, 
it remains to show that the essential image of $\BL_G^{\on{et,restr}}$ generates the target. 

\sssec{}

Write
$$\BL_G^{\on{et,restr}}=(\BL_G^{\on{et,restr}})_{\on{red}}\sqcup (\BL_G^{\on{et,restr}})_{\on{irred}};$$
these are unions of connected components that correspond to reducible (resp., irreducible) local systems.

\medskip

It suffices to show that the essential image of each of the corresponding functors
\begin{multline*} 
\QCoh((\BL_G^{\on{et,restr}})_{\on{red}})\underset{\QCoh(\LS^{\on{et,restr}}_\cG)}\otimes \Shv^{\on{et}}_\Nilp(\Bun_G)
\overset{\on{Id}\otimes \BL_G^{\on{et,restr}}}\longrightarrow \\
\to \QCoh((\BL_G^{\on{et,restr}})_{\on{red}})\underset{\QCoh(\LS^{\on{et,restr}}_\cG)}\otimes \IndCoh_\Nilp(\LS^{\on{et,restr}}_\cG)\simeq
\IndCoh_\Nilp((\LS^{\on{et,restr}}_\cG)_{\on{red}})
\end{multline*}
and 
\begin{multline*} 
\QCoh((\BL_G^{\on{et,restr}})_{\on{irred}})\underset{\QCoh(\LS^{\on{et,restr}}_\cG)}\otimes \Shv^{\on{et}}_\Nilp(\Bun_G)
\overset{\on{Id}\otimes \BL_G^{\on{et,restr}}}\longrightarrow \\
\to \QCoh((\BL_G^{\on{et,restr}})_{\on{irred}})\underset{\QCoh(\LS^{\on{et,restr}}_\cG)}\otimes \IndCoh_\Nilp(\LS^{\on{et,restr}}_\cG)\simeq
\IndCoh_\Nilp((\LS^{\on{et,restr}}_\cG)_{\on{irrred}})
\end{multline*}
generates its target category.

\medskip

This is in turn equivalent to showing that the essential image of the each of the functors
\begin{equation} \label{e:gen for red}
\Shv^{\on{et}}_\Nilp(\Bun_G)\overset{\BL_G^{\on{et,restr}}}\longrightarrow \IndCoh_\Nilp(\LS^{\on{et,restr}}_\cG)
\overset{\on{restriction}}\longrightarrow \IndCoh_\Nilp((\LS^{\on{et,restr}}_\cG)_{\on{red}})
\end{equation} 
and 
\begin{equation} \label{e:gen for irred}
\Shv^{\on{et}}_\Nilp(\Bun_G)\overset{\BL_G^{\on{et,restr}}}\longrightarrow \IndCoh_\Nilp(\LS^{\on{et,restr}}_\cG)
\overset{\on{restriction}}\longrightarrow \IndCoh_\Nilp((\LS^{\on{et,restr}}_\cG)_{\on{irred}})
\end{equation} 
generates the target category. 

\sssec{}

We first prove the assertion for \eqref{e:gen for red}. By induction on the semi-simple rank, we may
assume that \thmref{t:ff GLC}(ii) holds for proper Levi subgroups of $G$.

\medskip

Note that\footnote{In {\it loc. cit.} this is proved in the de Rham context, but the argument applies in any sheaf-theoretic context.} 
by \cite[Theorem 13.3.6]{AG1}, the union of the essential images of the functors 
$$\Eis^{-,\on{spec}}:\IndCoh_\Nilp(\LS^{\on{et,restr}}_\cM)\to \IndCoh_\Nilp(\LS^{\on{et,restr}}_\cG)$$
for proper Levi subgroups generates $\IndCoh_\Nilp((\LS^{\on{et,restr}}_\cG)_{\on{red}})$,
viewed as a full subcategory in $\IndCoh_\Nilp(\LS^{\on{et,restr}}_\cG)$.

\medskip 

Hence, the required generation assertion follows from \thmref{t:L and Eis}. 

\sssec{}

We now show that \eqref{e:gen for irred} generates the target category. Note that the functor
$\bu^{\on{spec}}$ (see \secref{sss:emb temp spec}) induces an \emph{equivalence}
$$\QCoh((\LS^{\on{et,restr}}_\cG)_{\on{irred}})\to \IndCoh_\Nilp((\LS^{\on{et,restr}}_\cG)_{\on{irred}}).$$

\medskip

Hence, it is enough to show that the essential image of the functor
$$
\Shv^{\on{et}}_\Nilp(\Bun_G)\overset{\BL_{G,\on{coarse}}^{\on{et,restr}}}\longrightarrow \QCoh(\LS^{\on{et,restr}}_\cG)
\overset{\on{restriction}}\longrightarrow \QCoh((\LS^{\on{et,restr}}_\cG)_{\on{irred}})
$$
generates the target category.

\medskip

We will show that the essential image of the functor $\BL_{G,\on{coarse}}^{\on{et,restr}}$ itself generates the target category. 
First, we note that this functor is fully faithful (indeed, our situation embeds fully faithfully into the Betti situation, where the functor
in question is an equivalence). Hence, it is enough to show that its \emph{left}\footnote{Note that without fully faithfulness,
generation of the target is equivalent to the fact that the \emph{right} adjoint be consevative.} adjoint, 
i.e., the functor $\BL_{G,\on{temp}}^{\on{et,restr},L}$, 
is conservative. 

\medskip

Note, however, that the diagram
$$
\CD
\Shv^{\on{et}}_\Nilp(\Bun_G) @<{\BL_{G,\on{temp}}^{\on{et,restr},L}}<<  \QCoh(\LS^{\on{et,restr}}_\cG) \\
@V{(\on{et}\to \on{Betti})_{\Bun_G}}VV @VV{(\on{et}\to \on{Betti})_{\LS_\cG}}V \\
\Shv^{\on{Betti,constr}}_\Nilp(\Bun_G) @<{\BL_{G,\on{temp}}^{\on{Betti,restr},L}}<{\sim}<  \QCoh(\LS^{\on{Betti,restr}}_\cG)
\endCD
$$
is commutative: indeed, the two circuits are compatible with the spectral action $\QCoh(\LS^{\on{et,restr}}_\cG)$
(which we think as a direct factor of $\QCoh(\LS^{\on{Betti,restr}}_\cG)$) and send $\CO_{\LS^{\on{et,restr}}_\cG}$
to (the Betti version of) $\on{Poinc}^{\on{Vac}}_{!,\Nilp}$. 

\medskip

Since the right vertical arrow in the diagram is conservative, and the bottom horizontal arrow
is an equivalence, we obtain that the top horizontal arrow is conservative, as required.

\qed[\thmref{t:ff GLC}(ii) for $k=\BC$]

\ssec{Proof of \thmref{t:ff GLC}(ii) for an arbitrary \texorpdfstring{$k$}{k} of char. \texorpdfstring{$0$}{0}}

The proof will use the Lefschetz principle. 

\sssec{}

Note that if $k_1\subset k_2$ is an extension of algebraically closed fields of characteristic $0$,
for a prestack $\CZ_1$ over $k_1$ and its base change $\CZ_2$ to $k_2$, the pullback functor
\begin{equation} \label{e:lisse base change}
\on{Lisse}(\CZ_1)\to \on{Lisse}(\CZ_2)
\end{equation} 
is an equivalence, see \cite[Proposition XIII.4.3]{SGA1}. 

\medskip

This formally implies that the pullback functor 
$$\Shv(\CZ_1)\to \Shv(\CZ_2)$$
is fully faithful. 

\medskip

Let $\CN_1\subset T^*(\CZ_1)$ be a conical Lagrangian subset. It follows from \eqref{e:lisse base change}
that the pullback functor
$$\Shv_{\CN_1}(\CZ_1)\to \Shv_{\CN_2}(\CZ_2)$$
is an equivalence. 

\sssec{}

In particular, for a curve $X_1$ defined over $k_1$, and its base change $X_2$ to $k_2$, the map 
$$\LS^{\on{restr}}_\cG(X_1)\to \LS^{\on{restr}}_\cG(X_2)$$
is an isomorphism (as prestacks over $\ol\BQ_\ell$).

\sssec{}

Given the curve $X$ over $k$, let $k'$ be a countably generated field over which it is defined;
denote the resulting curve over $k'$ by $X'$.  

\medskip

We obtain a commutative diagram
$$
\CD
\Shv_\Nilp(\Bun_G(X')) @>{\BL^{\on{restr}}_G(X')}>> \IndCoh_\Nilp(\LS^{\on{restr}}_\cG(X')) \\
@V{\sim}VV @VV{\sim}V \\
\Shv_\Nilp(\Bun_G(X)) @>{\BL^{\on{restr}}_G(X)}>> \IndCoh_\Nilp(\LS^{\on{restr}}_\cG(X)).  
\endCD
$$

Hence, the validity of \thmref{t:ff GLC}(ii) over $k$ is equivalent to its validity over $k'$.

\sssec{}

Embedding $k'$ into $\BC$, we obtain that  the validity of \thmref{t:ff GLC}(ii) over $k'$ is equivalent to its validity over $\BC$.
However, the latter has been established in \secref{ss:over C} above.

\qed[\thmref{t:ff GLC}(ii)]

\section{The specialization functor} \label{s:Sp}

In this section we introduce a procedure that will allow us to deduce information about the
Langlands functor in characteristic $p$ from its counterpart in characteristic $0$.

\medskip

This procedure essentially amounts to taking nearby cycles for a family of curves over a DVR.  

\ssec{Axiomatics for the functor}

\sssec{}

Let $\sk$ be an (algebraically closed) field of positive characteristic. Let  $\sR_0:=\on{Witt}(\sk)$
be the ring of Witt vectors of $\sk$, let $\sK_0$ denote the field of fractions of $\sR_0$ and let
$\sK$ denote the algebraic closure of $\sK_0$. Let $\sR$ denote the integral closure of $\sR_0$
in $\sK$. 

\medskip

Given a (smooth complete) curve $X_\sk$ over $\sk$, we can \emph{choose} its extension to a (smooth complete) curve 
$X_{\sR_0}$ over $\Spec(\sR_0)$. Let $X_\sK$ be the base change of $X_{\sR_0}$ to $K$. 

\sssec{Notational convention}

We will insert subscripts $\sk$, $\sR_0$ or $\sK$ into the corresponding geometric objects in order to 
specify which situation we are working with. 

\medskip

A symbol without such a subscript (e.g., just $\Bun_G$) means that the discussion applies
to any of the above situations. 

\sssec{} \label{sss:Sp LS}

Note that restriction along $X_\sk\to X_{\sR_0}$ defines an equivalence
$$\qLisse(X_{\sR_0})\simeq \qLisse(X_\sk).$$

Restriction along $X_\sK\to X_{\sR_0}$ defines a \emph{fully faithful} functor
$$\qLisse(X_\sR)\to \qLisse(X_\sK).$$

From here we obtain an embedding
$$\LS^{\on{restr}}_{\cG,\sk}\overset{\iota}\hookrightarrow \LS^{\on{restr}}_{\cG,\sK},$$
which identifies $\LS^{\on{restr}}_{\cG,\sk}$ with the union of \emph{some of} the 
connected components of $\LS^{\on{restr}}_{\cG,\sK}$. 

\sssec{} \label{sss:preoperties}

Consider the corresponding moduli stacks
$$\Bun_{G,\sk} \text{ and } \Bun_{G,\sK}.$$

\medskip

Consider the full subcategory (in fact, a direct summand)\footnote{We emphasize that there is no such object as $\Shv(\Bun_{G,\sK,\sk})$;
only $\Shv_{\Nilp}(\Bun_{G,\sK,\sk})$.}
$$\Shv_{\Nilp}(\Bun_{G,\sK,\sk}):=
\QCoh(\LS^{\on{restr}}_{\cG,\sk})\underset{\QCoh(\LS^{\on{restr}}_{\cG,\sK})}\otimes \Shv_{\Nilp}(\Bun_{G,\sK})
\subset \Shv_{\Nilp}(\Bun_{G,\sK}).$$

%\medskip
%
%The validity of \thmref{t:ff GLC} for $\sK$ implies that the functor $\BL^{\on{restr}}_{G,\sK}$ induces an equivalence
%$$\Shv_{\Nilp}(\Bun_{G,\sK,\sk})\overset{\sim}\to 
%\QCoh(\QCoh(\LS^{\on{restr}}_{\cG,\sk}))\underset{\QCoh(\LS^{\on{restr}}_{\cG,\sK})}\otimes 
%\IndCoh_\Nilp(\LS^{\on{restr}}_{\cG,\sK})=:\IndCoh_\Nilp(\LS^{\on{restr}}_{\cG,\sK,\sk});$$
%denote it by $\BL^{\on{restr}}_{G,\sK,\sk}$. 

\sssec{}

We will prove:

\begin{thm}  \label{t:construct deg}
There exists a functor
\begin{equation} \label{e:Sp Nilp}
\on{Sp}_{\sK\to \sk}:\Shv_{\Nilp}(\Bun_{G,\sK,\sk})\to \Shv_{\Nilp}(\Bun_{G,\sk})
\end{equation} 
with the following properties:

\medskip

\begin{itemize}

\item{\em(A)} The functor $\on{Sp}_{\sK\to \sk}$ is a Verdier quotient. 

\medskip

\item{\em(B)} The functor $\on{Sp}_{\sK\to \sk}$ intertwines the actions of $\QCoh(\LS^{\on{restr}}_{\cG,\sk})$. 

\medskip

\item{\em(C)} The functor $\on{Sp}_{\sK\to \sk}$ sends $\on{Poinc}^{\on{Vac}}_{!,\Nilp,\sK,\sk}$ to $\on{Poinc}^{\on{Vac}}_{!,\Nilp,\sk}$,
where $\on{Poinc}^{\on{Vac}}_{!,\Nilp,\sK,\sk}$ is the direct summand of  $\on{Poinc}^{\on{Vac}}_{!,\Nilp,\sK}$ that belongs to 
$\Shv_{\Nilp}(\Bun_{G,\sK,\sk})$. 

\medskip

\item{\em(D)} The functor $\on{Sp}_{\sK\to \sk}$ makes the diagrams
$$
\CD
\Shv_{\Nilp}(\Bun_{M,\sK,\sk}) @>{\on{Sp}_{\sK\to \sk,M}}>> \Shv_{\Nilp}(\Bun_{M,\sk}) \\
@V{\Eis^-_{!,\sK}}VV @VV{\Eis^-_{!,\sk}}V \\
\Shv_{\Nilp}(\Bun_{G,\sK,\sk}) @>{\on{Sp}_{\sK\to \sk,G}}>> \Shv_{\Nilp}(\Bun_{G,\sk}) 
\endCD
$$
commute.

\medskip

\item{\em(E)} The functor $\on{Sp}_{\sK\to \sk}$ is t-exact. 

\end{itemize} 

\end{thm} 

This theorem will be proved in the course of Sects. \ref{s:unit}-\ref{s:asymptotics}. For the rest of this section 
we will assume \thmref{t:construct deg}. We will derive some further properties of the functor $\on{Sp}_{\sK\to \sk}$ 
of \eqref{e:Sp Nilp}, as well as consequences for the category 
$\Shv_{\Nilp}(\Bun_{G,\sk})$ that follow from the above properties. 

%\medskip
%
%In the rest of this subsection we will show that the existence of such a functor formally implies 
%\thmref{t:ff GLC}.

\sssec{} \label{sss:Sp compactness}

We claim that the functor $\on{Sp}_{\sK\to \sk}$ of \eqref{e:Sp Nilp} preserves compactness.

\medskip

Indeed, the validity of \thmref{t:ff GLC} for $\sK$ implies that the category $\Shv_{\Nilp}(\Bun_{G,\sK,\sk})$ 
has compact generators of the form:

\smallskip

\noindent(i) $\CE\star \on{Poinc}^{\on{Vac}}_{!,\Nilp,\sK,\sk}$, where $\CE$ is a dualizable object in $\QCoh(\LS^{\on{restr}}_{\cG,\sk})$;

\smallskip

\noindent(ii) $\Eis^-_!(\CF_M)$ for $\CF_M\in \Shv_{\Nilp}(\Bun_{M,\sK,\sk})^c$.

\medskip 

Now, objects of the form
$$\on{Sp}_{\sK\to \sk}(\CE\star \on{Poinc}^{\on{Vac}}_{!,\Nilp,\sK,\sk})$$
are compact in $\Shv_{\Nilp}(\Bun_{G,\sk})$ by Properties (B) and (C), and objects of the form 
$$\on{Sp}_{\sK\to \sk}(\Eis^-_!(\CF_M))$$
are compact in $\Shv_{\Nilp}(\Bun_{G,\sk})$ by Property (D) and induction on semi-simple rank. 

\sssec{}

Note that combining with Property (A), we obtain:

\begin{cor} \label{c:generation in pos char}
Objects of the form 

\smallskip

\noindent{\em(i)} $\CE\star \on{Poinc}^{\on{Vac}}_{!,\Nilp,\sk}$, $\CE$ is a dualizable object in $\QCoh(\LS^{\on{restr}}_{\cG,\sk})$;

\smallskip

\noindent{\em(ii)} $\Eis^-_!(\CF_M)$, $\CF_M\in \Shv_{\Nilp}(\Bun_{M,\sk})^c$ 

\medskip

\noindent compactly generate $\Shv_{\Nilp}(\Bun_{G,\sk})$.
\end{cor}

\begin{rem}
Note that \corref{c:generation in pos char} is as an extension to positive characteristic 
of the main result of \cite{FR}.

\medskip

One can similarly show that objects of type (i) in \corref{c:generation in pos char} 
generate the \emph{tempered} subcategory of $\Shv_{\Nilp}(\Bun_{M,\sk})$. 

\end{rem} 

\begin{rem}

We conjecture that the entire $\Shv(\Bun_{G,\sk})$ is generated by objects of the form: 

\medskip

\noindent(i) $\CF\star \on{Poinc}^{\on{Vac}}_{!,\sk}$, $\CF\in \Rep(\cG)_\Ran$, 

\medskip

\noindent(ii) $\Eis^-_!(\CF_M)$, $\CF_M\in \Shv(\Bun_{M,\sk})^c$. 

\medskip

We can prove this over a ground field of characteristic $0$, but so far not over a field
of positive characteristic. 

\end{rem} 

\sssec{Proof of \thmref{t:compactness}} \label{sss:proof of compactness}

We can now deduce \thmref{t:compactness} for $\sk$.

\medskip

We need to show that compact generators of $\Shv_{\Nilp}(\Bun_{G,\sk})$ are compact as objects of
$\Shv(\Bun_{G,\sk})$.

\medskip

By \corref{c:generation in pos char}, it suffices to show to show that the objects
$$\CE\star \on{Poinc}^{\on{Vac}}_{!,\Nilp,\sk} \simeq \on{Sp}_{\sK\to \sk}(\CE\star \on{Poinc}^{\on{Vac}}_{!,\Nilp,\sK,\sk})$$
and
$$\Eis^-_!(\CF_M), \quad \CF_M\in \Shv_{\Nilp}(\Bun_{M,\sk})^c$$  
are compact in $\Shv(\Bun_{G,\sk})$.

\medskip

For objects of the second type, this is clear by induction on the semi-simple rank, since the functor $\Eis^-_!$
preserves compactness.  Thus, it remains to deal with objects of the first type.

\medskip

According to \cite[Proposition 16.4.7]{AGKRRV1}, the validity of \thmref{t:compactness} (in a given context)
is equivalent to the fact that the compact generators of $\Shv_\Nilp(\Bun_G)$ are eventually coconnective. 

\medskip

Hence, this property holds for $\Shv_\Nilp(\Bun_{G,\sK})$ by \secref{sss:compactness char 0}. In particular,
the objects $\CE\star \on{Poinc}^{\on{Vac}}_{!,\Nilp,\sK,\sk}$ are eventually coconnective. Now Property (E) 
implies that the objects 
$$\on{Sp}_{\sK\to \sk}(\CE\star \on{Poinc}^{\on{Vac}}_{!,\Nilp,\sK})$$
are also eventually coconnective. 

\qed[\thmref{t:compactness}]

\ssec{Proof of \thmref{t:ff GLC}(i)}

In this subsection we will show that the existence of the functor $\on{Sp}_{\sK\to \sk}$ with Properties (A)-(E)
specified in \thmref{t:construct deg} allows us to deduce \thmref{t:ff GLC}(i) from \thmref{t:ff GLC}(ii).

\sssec{}

First, note that the validity of \thmref{t:ff GLC}(ii) for $\sK$ implies that the functor $\BL^{\on{restr}}_{G,\sK}$ induces an equivalence
$$\Shv_{\Nilp}(\Bun_{G,\sK,\sk})\overset{\sim}\to 
\QCoh(\LS^{\on{restr}}_{\cG,\sk})\underset{\QCoh(\LS^{\on{restr}}_{\cG,\sK})}\otimes 
\IndCoh_\Nilp(\LS^{\on{restr}}_{\cG,\sK})
\simeq \IndCoh_\Nilp(\LS^{\on{restr}}_{\cG,\sk});$$
denote it by $\BL^{\on{restr}}_{G,\sK,\sk}$. 

\sssec{}

Set
\begin{equation} \label{e:Langlands quasi-inv}
'\BL^{\on{restr},L}_{G,\sk}:=\on{Sp}_{\sK\to \sk}\circ (\BL^{\on{restr}}_{G,\sK,\sk})^{-1}, \quad
\IndCoh_\Nilp(\LS^{\on{restr}}_{\cG,\sk})\to \Shv_\Nilp(\Bun_{G,\sk}).
\end{equation} 

Note that it follows from Property (C) that
\begin{equation} \label{e:temp adj}
'\BL^{\on{restr},L}_{G,\sk}\circ \bu^{\on{spec}}\simeq \BL^{\on{restr},L}_{G,\on{temp},k}.
\end{equation} 

\sssec{}

We will prove:

\begin{prop} \label{p:adj}
The functor $'\BL^{\on{restr},L}_{G,\sk}$ is the left adjoint of $\BL^{\on{restr}}_{G,\sk}$.
\end{prop}

The proof will be given in \secref{ss:adj} below. We proceed with the proof of \thmref{t:ff GLC}(i).
Assuming the proposition, we will denote
$$\BL^{\on{restr},L}_{G,\sk}:={}'\BL^{\on{restr},L}_{G,\sk}.$$

\sssec{}

As a formal corollary of \propref{p:adj} combined with Property (A), we obtain:

\begin{cor} \label{c:L ff}
The functor $\BL^{\on{restr}}_{G,\sk}$ is fully faithful.
\end{cor}

Thus, to prove \thmref{t:ff GLC}(ii) it remains to show that the essential image of $\BL^{\on{restr}}_{G,\sk}$
equals 
$$\IndCoh_\Nilp({}'\!\LS^{\on{restr}}_{\cG,\sk}),$$
where $'\!\LS^{\on{restr}}_{\cG,\sk}$ is the disjoint union of some of the connected components
of $\LS^{\on{restr}}_{\cG,\sk}$.

\sssec{} \label{sss:proof main start}

Since the functors involved are $\QCoh(\LS^{\on{restr}}_{\cG,\sk})$-linear, we can work with 
one connected component of $\LS^{\on{restr}}_{\cG,\sk}$ at a time. I.e., we need to show that for
a given connected component $\CZ_\alpha$ of $\LS^{\on{restr}}_{\cG,\sk}$, either the adjoint functors
\begin{equation} \label{e:adj alpha}
\BL^{\on{restr}}_{\cG,\sk,\alpha}:
\QCoh(\CZ_\alpha)\underset{\QCoh(\LS^{\on{restr}}_{\cG,\sk})}\otimes \Shv_\Nilp(\Bun_{G,\sk}) \leftrightarrows
\IndCoh_\Nilp(\CZ_\alpha):\BL^{\on{restr},L}_{G,k,\alpha}
\end{equation} 
are mutually inverse equivalences, or the left-hand side is zero. 

\medskip

We consider separately 
the cases when a given connected component corresponds to reducible or irreducible
local systems.

\sssec{}

We start with the irreducible case. Note that the inclusion
$$\bu^{\on{spec}}:\QCoh(\CZ_\alpha)\to \IndCoh_\Nilp(\CZ_\alpha)$$
is an equality.

\medskip

Hence, the composition
$$\BL^{\on{restr}}_{\cG,\sk,\alpha}\circ \BL^{\on{restr},L}_{G,k,\alpha},$$
being a $\QCoh(\CZ_\alpha)$-linear endofunctor of $\QCoh(\CZ_\alpha)$,
is given by tensoring by a unital algebra object\footnote{The algebra structure comes from the fact
that $\BL^{\on{restr}}_{\cG,\sk,\alpha}\circ \BL^{\on{restr},L}_{G,k,\alpha}$ is a $\QCoh(\CZ_\alpha)$-linear monad.}
$$\CA_\alpha\in  \QCoh(\CZ_\alpha).$$

We need to show that either the unit map
\begin{equation} \label{e:unit map}
\CO_{\CZ_\alpha}\to \CA_\alpha
\end{equation}
is an isomorphism, or $\CA_\alpha=0$.

\medskip

By Barr-Beck-Lurie, the right adjoint in \eqref{e:adj alpha} identifies the left-hand side 
with 
$$\CA_\alpha\mod(\QCoh(\CZ_\alpha)).$$

By \corref{c:L ff}, the forgetful functor
\begin{equation} \label{e:unit map alg}
\CA_\alpha\mod(\QCoh(\CZ_\alpha))\to \QCoh(\CZ_\alpha)
\end{equation}
is fully faithful.

\medskip

Let 
$$i_\sigma:\on{pt}\to \CZ_\alpha$$
correspond to the (unique) closed point of $\CZ_\alpha$. 

\medskip

Applying to \eqref{e:unit map alg} the operation $\Vect \underset{i_\sigma^*,\QCoh(\CZ_\alpha)}\otimes (-)$,
we obtain a fully faithful functor
$$i_\sigma^*(\CA_\alpha)\mod(\Vect)\to \Vect.$$

This easily implies that either $\sfe\to i_\sigma^*(\CA_\alpha)$ is an isomorphism, or $i_\sigma^*(\CA_\alpha)=0$. 
The latter dichotomy implies that one of the above two possibilities for $\CA_\alpha$ itself must hold. 

%Note, however, that $\CZ_\alpha$ is eventually coconnective and has a unique isomorphism class of $\sk$-points, up to
%isomorphism:
%$$i_\alpha:\on{pt}\to \CZ_\alpha.$$
%
%Hence, \eqref{e:unit map} is an isomorphism if and only if the map
%$$\sfe\to i_\alpha^*(\CA_\alpha)$$
%is an isomorphism, and $\CA_\alpha$ is zero if and only if $i_\alpha^*(\CA_\alpha)$ is zero.
%
%\medskip
%
%This allows us to replace the pair of adjoint functors \eqref{e:adj alpha} by their further tensor products
%$$\Vect\underset{i_\alpha^*,\QCoh(\CZ_\alpha)}\otimes (-),$$
%i.e.,
%\begin{equation} \label{e:adj alpha pt}
%\Vect \underset{i_\alpha^*,\QCoh(\LS^{\on{restr}}_{\cG,\sk})}\otimes \Shv_\Nilp(\Bun_{G,\sk}) \leftrightarrows \Vect.
%\end{equation} 
%
%Since the right adjoint in \eqref{e:adj alpha} was fully faithful, so is the right adjoint in \eqref{e:adj alpha pt}:
%indeed, the property of the counit of the adjunction to be an isomorphism survives tensor products. 
%
%\medskip
%
%Hence, we are reduced to showing that either the adjoint functors \eqref{e:adj alpha pt} are mutually
%inverse equivalences, or the left-hand side is zero. However, this is always the case 
%
%

\sssec{} \label{sss:red case}

We now consider the \emph{reducible} case. Let $\LS^{\on{restr}}_{\cG,\sk,\sigma}$ be a connected
component whose closed point $\sigma$ is a semi-simple local system. Then there 
exists a unique conjugacy class of Levi subgroups $M$ (see \cite[Sects. 3.6 and 3.7]{AGKRRV1}) such that 
$\sigma$ factors via an \emph{irreducible} $\cM$-local system $\sigma_\cM$.

\medskip

By induction on the semi-simple rank, we can assume that the assertion of \thmref{t:ff GLC}(ii)
is valid for $M$. Let $\LS^{\on{restr}}_{\cM,\sk,\sigma_\cM}$ be the corresponding connected
component of $\LS^{\on{restr}}_{\cM,\sk}$.

\medskip

We consider the following two cases of the behavior of the functor \eqref{e:adj alpha} 
with $G$ replaced by $M$ and $\CZ_\alpha:=\LS^{\on{restr}}_{\cM,\sk,\sigma_\cM}$:

\smallskip

\noindent(a) It is an equivalence;

\smallskip

\noindent(b) The left-hand side is zero. 

\sssec{} \label{sss:gen comp}

We claim that in case (a), the functor \eqref{e:adj alpha} for $G$ and $\CZ_\alpha:=\LS^{\on{restr}}_{\cG,\sk,\sigma}$
is an equivalence. 

\medskip

Indeed, it is enough to show that the essential image of $\BL^{\on{restr}}_{\cG,\sk,\alpha}$ generates the target category.
Note that $\IndCoh_\Nilp(\LS^{\on{restr}}_{\cG,\sk,\sigma})$ is generated by the essential
images of the functors
$$\Eis^{-,\on{spec}}:\IndCoh_\Nilp(\LS^{\on{restr}}_{\cM,\sk,\sigma})\to 
\IndCoh_\Nilp(\LS^{\on{restr}}_{\cG,\sk,\sigma})$$
for the parabolics in the given class of association.

\medskip

Now the required assertion follows from \thmref{t:L and Eis} (for $\sk$).

\sssec{} \label{sss:proof main end}

We claim that in case (b), the left adjoint in \eqref{e:adj alpha} for $G$ and $\CZ_\alpha:=\LS^{\on{restr}}_{\cG,\sk,\sigma}$,
is zero. This will show that the left-hand side is zero (by \corref{c:L ff}). 

\medskip

By construction and Property (B), we can identify the functor in question with
$$\on{Sp}_{\sK\to \sk}:
\QCoh(\LS^{\on{restr}}_{\cG,\sk,\sigma})\underset{\QCoh(\LS^{\on{restr}}_{\cG,\sk})}\otimes 
\Shv_\Nilp(\Bun_{G,\sK,\sk})\to \QCoh(\LS^{\on{restr}}_{\cG,\sk,\sigma})\underset{\QCoh(\LS^{\on{restr}}_{\cG,\sk})}\otimes 
\Shv_\Nilp(\Bun_{G,\sk}).$$

We will show that the above functor annihilates the generators. By Property (D), we have a commutative diagram
$$
\CD
\QCoh(\LS^{\on{restr}}_{\cM,\sk,\sigma_\cM})\underset{\QCoh(\LS^{\on{restr}}_{\cM,\sk})}\otimes 
\Shv_\Nilp(\Bun_{M,\sK,\sk}) @>{\Eis^-_!}>> 
\QCoh(\LS^{\on{restr}}_{\cG,\sk,\sigma})\underset{\QCoh(\LS^{\on{restr}}_{\cG,\sk})}\otimes 
\Shv_\Nilp(\Bun_{G,\sK,\sk}) \\
@V{\on{Sp}_{\sK\to \sk}}VV @VV{\on{Sp}_{\sK\to \sk}}V \\
\QCoh(\LS^{\on{restr}}_{\cM,\sk,\sigma_\cM})\underset{\QCoh(\LS^{\on{restr}}_{\cM,\sk})}\otimes 
\Shv_\Nilp(\Bun_{M,\sk}) @>>{\Eis^-_!}> 
\QCoh(\LS^{\on{restr}}_{\cG,\sk,\sigma})\underset{\QCoh(\LS^{\on{restr}}_{\cG,\sk})}\otimes 
\Shv_\Nilp(\Bun_{G,\sk}).
\endCD
$$

We claim that the top horizontal arrows in these diagrams, taken for all the parabolics in the given class of association,
generate the target. This follows from the combination of:

\medskip

\noindent(i) The fact that $\BL^{\on{restr}}_{G,\sK}$ is an equivalence;

\smallskip

\noindent(ii) \thmref{t:L and Eis} for $\sK$;

\smallskip

\noindent(iii) The corresponding fact on the spectral side.

\medskip 

Hence, it suffices to show that the anti-clockwise circuits in these diagrams vanish. However, this follows
from the fact that the lower-left corner vanishes (this is the assumption in case (b)).

\qed[\thmref{t:ff GLC}(i)]

\sssec{Proof of \thmref{t:ff GLC}(iii)}

We need to show that in the notations of \eqref{e:adj alpha}, the category 
\begin{equation} \label{e:cat conn comp}
\QCoh(\CZ_\alpha)\underset{\QCoh(\LS^{\on{restr}}_{GL_n})}\otimes \Shv_\Nilp(\Bun_{GL_n})
\end{equation} 
is non-zero for every connected component $\CZ_\alpha$ of $\LS^{\on{restr}}_{GL_n}$.

\medskip

By induction, we may assume that the functor $\BL^{\on{restr}}_{GL_{n'}}$ is an equivalence for $n'<n$.
In particular, it is an equivalence for proper Levi subgroups of $G=GL_n$. Hence, in the notations of
\secref{sss:red case} only scenario (a) occurs. In particular, the category \eqref{e:cat conn comp} is
non-zero whenever $\CZ_\alpha$ corresponds to reducible local systems.

\medskip

It remains to treat the case of $\CZ_\alpha$ corresponding to an irreducible local system $\sigma$.
It is enough to show that $\Shv_\Nilp(\Bun_{GL_n})$ contains a non-zero Hecke eigensheaf corresponding
to $\sigma$. However, this has been established in \cite{FGV,Ga1}.

\qed[\thmref{t:ff GLC}(iii)]

\ssec{Proof of \propref{p:adj}} \label{ss:adj}

\sssec{}

Note that the functor $'\BL^{\on{restr},L}_{G,\sk}$ preserves compactness (see \secref{sss:Sp compactness}).
Hence, it is enough to show that for
$$\CF\in \IndCoh_\Nilp(\LS^{\on{restr}}_{\cG,\sk})^c \text{ and } \CM\in \Shv_\Nilp(\Bun_{G,\sk})^c,$$
there is a canonical isomorphism
\begin{equation} \label{e:adj 1}
\CHom_{\Shv(\Bun_{G,\sk})}({}'\BL^{\on{restr},L}_{G,\sk}(\CF),\CM)\simeq
\CHom_{\IndCoh_\Nilp(\LS^{\on{restr}}_{\cG,\sk})}(\CF,\BL^{\on{restr}}_{G,\sk}(\CM)).
\end{equation} 

\sssec{}

Consider the tempered quotients
$$\bu^R:\Shv_\Nilp(\Bun_{G,\sk})\leftrightarrows \Shv_\Nilp(\Bun_{G,\sk})_{\on{temp}}:\bu$$
and
$$\bu^{\on{spec},R}: \IndCoh_\Nilp(\LS^{\on{restr}}_{\cG,\sk}) \leftrightarrows  \QCoh(\LS^{\on{restr}}_{\cG,\sk}):\bu^{\on{spec}}.$$

By construction, the functor $\BL^{\on{restr}}_{G,\sk}$ induces a functor
$$\BL^{\on{restr}}_{G,\on{temp},\sk}:\Shv_\Nilp(\Bun_{G,\sk})_{\on{temp}}\to \QCoh(\LS^{\on{restr}}_{\cG,\sk})$$
such that
%$$\bu^{\on{spec}}\circ \BL^{\on{restr}}_{G,\on{temp},k}\simeq \BL^{\on{restr}}_{G,k}\circ \bu
%\text{ and }
$$\BL^{\on{restr}}_{G,\on{temp},k}\circ \bu^R\simeq \BL^{\on{restr}}_{G,\on{coarse},\sk}\simeq \bu^{\on{spec},R}\circ \BL^{\on{restr}}_{G,\sk}.$$

In particular, the functor $\BL^{\on{restr},L}_{G,\on{temp},\sk}$ takes values in $\Shv_\Nilp(\Bun_{G,\sk})_{\on{temp}}$
(viewed as a subcategory) and provides a left adjoint to $\BL^{\on{restr}}_{G,\on{temp},\sk}$.

\sssec{}

By \propref{p:Sp temp} below, the functor $'\BL^{\on{restr},L}_{G,\sk}$ induces a functor
$$'\BL^{\on{restr},L}_{G,\on{temp},\sk}:\QCoh(\LS^{\on{restr}}_{\cG,\sk})\to \Shv_\Nilp(\Bun_{G,\sk})_{\on{temp}},$$
so that 
$$\bu^R\circ {}'\BL^{\on{restr},L}_{G,\sk}\simeq {}'\BL^{\on{restr},L}_{G,\on{temp},\sk}\circ \bu^{\on{spec},R}.$$

\medskip

From \eqref{e:temp adj} we obtain that
$$'\BL^{\on{restr},L}_{G,\on{temp},\sk}\simeq \BL^{\on{restr},L}_{G,\on{temp},\sk}.$$

In particular, we obtain that the functors
\begin{equation} \label{e:adj 2}
\BL^{\on{restr}}_{G,\on{temp},\sk}:
\Shv_\Nilp(\Bun_{G,\sk})_{\on{temp}} \leftrightarrows \QCoh(\LS^{\on{restr}}_{\cG,\sk}):{}'\BL^{\on{restr},L}_{G,\on{temp},\sk}
\end{equation} 
do form an adjoint pair.

\sssec{}

Consider the maps
\begin{multline} \label{e:adj 3}
\CHom_{\Shv(\Bun_{G,\sk})}({}'\BL^{\on{restr},L}_{G,\sk}(\CF),\CM) \to
\CHom_{\Shv_\Nilp(\Bun_{G,\sk})_{\on{temp}}}(\bu^R\circ {}'\BL^{\on{restr},L}_{G,\sk}(\CF),\bu^R(\CM)) = \\
=\CHom_{\Shv_\Nilp(\Bun_{G,\sk})_{\on{temp}}}({}'\BL^{\on{restr},L}_{G,\on{temp},k}\circ \bu^{\on{spec},R}(\CF),\bu^R(\CM)) 
\overset{\text{adjunction\,\eqref{e:adj 2}}}\simeq \\
\simeq \CHom_{\QCoh(\LS^{\on{restr}}_{\cG,\sk})}(\bu^{\on{spec},R}(\CF),\BL^{\on{restr}}_{G,\on{temp},\sk}\circ \bu^R(\CM))= \\
= \CHom_{\QCoh(\LS^{\on{restr}}_{\cG,\sk})}(\bu^{\on{spec},R}(\CF),\bu^{\on{spec},R}\circ \BL^{\on{restr}}_{G,\sk}(\CM))\leftarrow
\CHom_{\IndCoh_\Nilp(\LS^{\on{restr}}_{\cG,\sk})}(\CF,\BL^{\on{restr}}_{G,\sk}(\CM))
\end{multline}

\medskip

In order to establish \eqref{e:adj 1}, it suffices to show that the first and the last arrow in \eqref{e:adj 3} are isomorphisms. 

\medskip

For the last arrow, this follows from the fact that the functor $\BL^{\on{restr}}_{G,\sk}$ sends compact objects to eventually
coconnective objects (by construction), and the functor $\bu^{\on{spec},R}$
is fully faithful on the eventually coconnective subcategory.

\medskip

For the first arrow, since $\BL^{\on{restr},L}_{G,\sk}(\CF)$ is compact (see \secref{sss:Sp compactness}), the assertion
follows from the next lemma: 

\begin{lem} \label{l:proj to temp}
The restriction of the functor $\bu^R$ to $\Shv_\Nilp(\Bun_{G,\sk})^c$ is fully faithful.
\end{lem}

\begin{proof}

Repeats verbatim the proof of \cite[Proposition 5.2.3]{GLC1}.

\end{proof}

\qed[\propref{p:adj}]

\ssec{An alternative proof of \thmref{t:ff GLC}(i) and \propref{p:adj}}

This proof will use an additional property of the functor $\on{Sp}_{\sK\to \sk}$, given by Remark
\ref{r:Sp and CT}. 

\sssec{}

Let $$'\BL^{\on{restr},L}_{G,\sk}:\IndCoh_\Nilp(\LS^{\on{restr}}_{\cG,\sk})\to \Shv_\Nilp(\Bun_{G,\sk})$$
be defined as in \eqref{e:Langlands quasi-inv}. 

\medskip

We already know that this functor preserves compactness and by Property (A) of the functor \eqref{e:Sp Nilp},
it is a Verdier quotient. I.e., the functor $'\BL^{\on{restr},L}_{G,\sk}$ admits a $\QCoh(\LS^{\on{restr}}_{\cG,\sk})$-linear
fully faithful right adjoint. Denote it by $'\BL^{\on{restr}}_{G,\sk}$ (note, however, that we do know yet that 
$'\BL^{\on{restr}}_{G,\sk}$ is isomorphic to $\BL^{\on{restr}}_{G,\sk}$). 

\medskip

We wish to show that $'\BL^{\on{restr}}_{G,\sk}$ is an equivalence onto a direct summand of $\IndCoh_\Nilp(\LS^{\on{restr}}_{\sG,\sk})$
corresponding  to the union of some of the connected
components of $\LS^{\on{restr}}_{\cG,\sk}$.

\sssec{}

Inspecting the argument in Sects. \ref{sss:proof main start}-\ref{sss:proof main end}, we see that the only place where we
used that $'\BL^{\on{restr}}_{G,\sk}=\BL^{\on{restr}}_{G,\sk}$ was in case (a) in \secref{sss:gen comp}. We provide an alternative argument as follows:

\medskip

It suffices to prove that the functor $'\BL^{\on{restr},L}_{G,\sk}|_{\IndCoh_\Nilp(\LS^{\on{restr}}_{\cG,\sk,\sigma})}$
is conservative. In other words, we wish to show that in case (a)  the functor
\begin{equation} \label{e:Sp comp}
\on{Sp}_{\sK\to \sk}:
\QCoh(\LS^{\on{restr}}_{\cG,\sk,\sigma})\underset{\QCoh(\LS^{\on{restr}}_{\cG,\sk})}\otimes \Shv(\Bun_{G,\sK,\sk})\to
\QCoh(\LS^{\on{restr}}_{\cG,\sk,\sigma})\underset{\QCoh(\LS^{\on{restr}}_{\cG,\sk})}\otimes \Shv(\Bun_{G,\sk})
\end{equation} 
is conservative.

\medskip

According to Remark \ref{r:Sp and CT}, we have commutative diagrams
\begin{equation} \label{e:CT q-Lang}
\CD
\QCoh(\LS^{\on{restr}}_{\cM,\sk,\sigma})\underset{\QCoh(\LS^{\on{restr}}_{\cM,\sk})}\otimes \Shv(\Bun_{M,\sK,\sk}) @>{\on{Sp}_{\sK\to \sk}}>> 
\QCoh(\LS^{\on{restr}}_{\cM,\sk,\sigma})\underset{\QCoh(\LS^{\on{restr}}_{\cM,\sk})}\otimes \Shv(\Bun_{M,\sk})  \\
@A{\on{CT}^-_*}AA @AA{\on{CT}^-_*}A \\
\QCoh(\LS^{\on{restr}}_{\cG,\sk,\sigma})\underset{\QCoh(\LS^{\on{restr}}_{\cG,\sk})}\otimes \Shv(\Bun_{G,\sK,\sk}) @>{\on{Sp}_{\sK\to \sk}}>> 
\QCoh(\LS^{\on{restr}}_{\cG,\sk,\sigma})\underset{\QCoh(\LS^{\on{restr}}_{\cG,\sk})}\otimes \Shv(\Bun_{G,\sk}),
\endCD 
\end{equation}
where $\on{CT}^-_*$ denotes the constant term functor. 

\medskip

Now, the validity of GLC over $\sK$ implies that the direct sums of the functors $\on{CT}^-_*$ over the
parabolics in the given class of association (i.e., the left vertical arrow in \eqref{e:CT q-Lang})
is conservative. 

\medskip
 
By induction on the semi-simple rank we can assume that the top horizontal arrow in \eqref{e:CT q-Lang}
is conservative.

\medskip

Hence, the bottom the top horizontal arrow in \eqref{e:CT q-Lang}
is also conservative, as desired. 

\qed[\thmref{t:ff GLC}(i)]

\sssec{}

We now give an alternative argument for \propref{p:adj}. Namely, we wish to show that
$'\BL^{\on{restr}}_{G,\sk}=\BL^{\on{restr}}_{G,\sk}$. We will show that $'\BL^{\on{restr}}_{G,\sk}$
satisfies the two conditions of \corref{c:existence of Langlands functor}. 

\medskip

We already know that $'\BL^{\on{restr}}_{G,\sk}$ induces an equivalence 
$$\Shv_\Nilp(\Bun_{G,\sk}) \to \IndCoh_\Nilp({}'\LS^{\on{restr}}_\sk).$$
In particular, it sends compact objects to compact (and, in particular, eventually
coconnective) ones. Hence, $'\BL^{\on{restr}}_{G,\sk}$ satisfies condition (i)
of \corref{c:existence of Langlands functor}.

\medskip

Condition (ii) follows by passing to right adjoints from \eqref{e:temp adj}. 

\qed[\propref{p:adj}]

\section{Proof of \thmref{t:construct deg}} \label{s:construct funct}

In this section we begin the proof of \thmref{t:construct deg}. Namely, we will
construct the functor $\on{Sp}$ and establish Properties (B)-(E). 

\medskip

Property (A) will be dealt with in the next section.

\ssec{Construction of the specialization functor}

\sssec{}

Let $\CY_{\sR_0}$ be a scheme or algebraic stack over $\sR_0$. Our goal is
to define a functor
\begin{equation} \label{e:specialization functor}
\on{Sp}_{\sK\to \sk}:\Shv(\CY_\sK)\to \Shv(\CY_\sk).
\end{equation} 

\medskip

Denote
$$\Spec(\sk)\overset{\bi}\to \Spec(\sR)\overset{\bj}\leftarrow \Spec(\sK).$$

We will use the same symbols $\bi$ and $\bj$ for the corresponding maps
$$\CY_\sk\to \CY_{\sR}\leftarrow \CY_\sK.$$

Euphemistically, the functor $\on{Sp}_{\sK\to \sk}$ is given by
\begin{equation} \label{e:Sp euph}
\bi^*\circ \bj_*.
\end{equation} 

Note, however, that $\Spec(\sR)$ is non-Noetherian. In what follows we will rewrite the definition of 
$\on{Sp}_{\sK\to \sk}$ in terms of functors that only use algebro-geometric 
objects of finite type. 

\sssec{}

First, covering $\CY_{\sR_0}$ smoothly be (affine) schemes, it is sufficient to define $\on{Sp}_{\sK\to \sk}$
for $\CY_{\sR_0}=S_{\sR_0}$, where $S_{\sR_0}$ is an affine scheme of finite type over $\sR_0$, 
provided that it commutes with pullbacks along smooth maps $S'_{\sR_0}\to S_{\sR_0}$. 

\medskip

Second, since $\Shv(S_\sK)$ is by definition the ind-completion of $\Shv(S_\sK)^c$, it suffices to define 
$\on{Sp}_{\sK\to \sk}$ as a functor $\Shv(S_\sK)^c\to \Shv(S_\sk)^c$. 

\medskip

Third, by the definition of $\Shv(-)^c$, this category is the colimit of $\Shv(-)_E^c$, over $\BQ_\ell\subset E\subset \ol\BQ_\ell$,
where $E$ is a finite extension of $\BQ_\ell$. Hence, it suffices to define 
$\on{Sp}_{\sK\to \sk}$ as a functor $\Shv(S_\sK)_E^c\to \Shv(S_\sk)_E^c$. 

\medskip

Fourth, by the definition of $\Shv(-)_E^c$, it is obtained as the localization with respect to $\ell$ of $\Shv(-)^c_{\CO_E}$. 
Hence, it suffices to define 
$\on{Sp}_{\sK\to \sk}$ as a functor $\Shv(S_\sK)_{\CO_E}^c\to \Shv(S_\sk)_{\CO_E}^c$. 

\medskip

Finally, by the definition of $\Shv(-)_{\CO_E}^c$, it is obtained as
$$\underset{n}{\on{lim}}\, \Shv(-)_{\CO_E/\ell^n}^c,$$ 
where the limit is taken in the $\infty$-category of \emph{non-cocomplete} DG categorries.

\medskip

Hence, it suffices to define $\on{Sp}_{\sK\to \sk}$ as a functor 
\begin{equation} \label{e:Sp finite}
\Shv(S_\sK)_{\CO_E/\ell^n}^c\to \Shv(S_\sk)_{\CO_E/\ell^n}^c.
\end{equation}

\sssec{}

Note that
$$\Shv(S_\sK)_{\CO_E/\ell^n}^c=\underset{\sK'_0}{\on{colim}}\, \Shv(S_{\sK'_0})_{\CO_E/\ell^n}^c,$$
where:

\begin{itemize}

\item The colimit is taken in the $\infty$-category of \emph{non-cocomplete} DG categorries;

\item The index category is the poset of finite extensions $\sK'_0\supset \sK_0$.

\end{itemize}

We proceed to define the corresponding functors
\begin{equation} \label{e:Sp finite finite}
\on{Sp}_{\sK'_0\to \sk}:\Shv(S_{\sK'_0})_{\CO_E/\ell^n}^c\to  \Shv(S_\sk)_{\CO_E/\ell^n}^c.
\end{equation}

\sssec{}

Let $\sR'_0$ denote the integral closure of $\sR_0$ in $\sK'_0$. Consider the diagram
$$
\CD
\CY_\sk @>>> \CY_{\sR'_0} @<<< \CY_{\sK'_0} \\
@VVV @VVV @VVV \\
\Spec(\sk) @>>> \Spec(\sR'_0) @<<< \Spec(\sK'_0).
\endCD
$$

We define \eqref{e:Sp finite finite} to be the nearby cycles functor $\Psi$. 

\sssec{}

Note that, by construction, the composition
$$\Shv(\CY_{\sK_0}) \overset{\text{pullback}}\longrightarrow \Shv(\CY_\sK) \overset{\on{Sp}_{\sK\to \sk}}\longrightarrow \Shv(\CY_\sk)$$
is the usual nearby cycles functor
$$\Psi:\Shv(\CY_{\sK_0}) \to \Shv(\CY_\sk).$$

\sssec{} \label{sss:Sp is exact}

We record the following properties of the functor \eqref{e:specialization functor}:

\medskip

\begin{enumerate}

\item It is t-exact;

\medskip

\item It sends constructible\footnote{Constructible:=pullback to an affine scheme by means of a 
smooth morphism is constructible (i.e., compact). Note, however, that for stacks ``constructible"
does not imply ``compact": the obstruction is (i) non-quasi-compactness of the stack and 
(ii) non-trivial stabilizers.} objects in $\Shv(\CY_\sK)$ to constructible objects in $\Shv(\CY_\sK)$;

\medskip

\item For $\CY_{\sR_0}=\CY^1_{\sR_0}\times \CY^2_{\sR_0}$ and $\CF_i\in \Shv(\CY^i_\sK)$, we have
\begin{equation} \label{e:Sp Kunneth}
\on{Sp}_{\sK\to \sk}(\CF_1\boxtimes \CF_2)\simeq \on{Sp}_{\sK\to \sk}(\CF_1)\boxtimes \on{Sp}_{\sK\to \sk}(\CF_2).
\end{equation} 

\medskip

\item For a map $f:\CY^1_{\sR_0}\to \CY^2_{\sR_0}$ we have the natural transformations
\begin{equation} \label{e:Sp and dir im}
(f_\sk)_!\circ \on{Sp}_{\sK\to \sk}\to \on{Sp}_{\sK\to \sk}\circ (f_\sK)_!, \,\,
\on{Sp}_{\sK\to \sk}\circ (f_\sK)_*\to (f_\sk)_*\circ \on{Sp}_{\sK\to \sk},
\end{equation} 
and 
\begin{equation} \label{e:Sp and inv im}
(f_\sk)^*\circ \on{Sp}_{\sK\to \sk}\to \on{Sp}_{\sK\to \sk}\circ (f_\sK)^*,\,\, 
\on{Sp}_{\sK\to \sk}\circ (f_\sK)^!\to (f_\sk)^!\circ \on{Sp}_{\sK\to \sk},
\end{equation} 
where $f_\sk$ (resp., $f_\sK$) denotes the fiber of $f$ over $\Spec(\sk)$ (resp., $\Spec(\sK)$).

\medskip

\item The maps \eqref{e:Sp and dir im} are mutually inverse isomorphisms when $f$ is proper and
the maps \eqref{e:Sp and inv im} are mutually inverse isomorphisms (up to a shift by the relative dimension) when $f$ is smooth. 

\medskip

\item\label{i:sp 6} It commutes with Verdier duality on constructible objects
\begin{equation} \label{e:Sp and Verdier 0}
\BD_{\CY_\sk}(\on{Sp}_{\sK\to \sk}(\CF)) \simeq \on{Sp}_{\sK\to \sk}(\BD_{\CY_\sK}(\CF));
\end{equation} 
moreover, for an object $\CF\in \Shv(\CY_\sK)^{\on{constr}}$,
the following diagram commutes
\begin{equation} \label{e:Sp and Verdier}
\CD
\on{Sp}_{\sK\to \sk}(\CF)\boxtimes  \BD_{\CY_\sk}(\on{Sp}_{\sK\to \sk}(\CF)) @>{\sim}>{\on{id}\boxtimes \text{\eqref{e:Sp and Verdier 0}}}>
\on{Sp}_{\sK\to \sk}(\CF)\boxtimes \on{Sp}_{\sK\to \sk}(\BD_{\CY_\sK}(\CF)) \\
@AAA @AAA \\
(\Delta_{\CY_\sk})_!(\ul\sfe_{\CY_\sk}) @>>>
\on{Sp}_{\sK\to \sk}\circ (\Delta_{\CY_\sK})_!(\ul\sfe_{\CY_\sK}),
\endCD
\end{equation} 
where:

\smallskip

\begin{itemize}

\item The symbol $\on{Sp}_{\sK\to \sk}$ in the bottom-right corner refers to the functor \eqref{e:specialization functor}
for $\CY_{\sR_0}\underset{\Spec(\sR_0)}\times \CY_{\sR_0}$;

\item The bottom horizontal arrow comes from the identification $\ul\sfe_{\Spec(\sk)}\simeq \on{Sp}_{\sK\to \sk}(\ul\sfe_{\Spec(\sK)})$
and the natural transformation 
$$(\Delta_{\CY_\sk})_!\circ \pi_{\CY_\sk}^* \circ \on{Sp}_{\sK\to \sk}\to \on{Sp}_{\sK\to \sk} \circ (\Delta_{\CY_\sK})_! \circ \pi_{\CY_\sK}^*$$
where $\pi_{\CY_-}:\CY_- \to \Spec(-)$ is the relevant structural map.

\end{itemize}

\end{enumerate} 

\begin{rem}\label{r:sp and tw arr}

Suppose $\alpha:\CF \to \CG$
is a morphism of constructible sheaves. From \eqref{e:Sp and Verdier}
we obtain a tautological commutative diagram
\begin{equation}\label{e:Sp and Verdier, tw}
\CD
\on{Sp}_{\sK\to \sk}(\CG)\boxtimes  \BD_{\CY_\sk}(\on{Sp}_{\sK\to \sk}(\CF)) @>{\sim}>{\on{id}\boxtimes \text{\eqref{e:Sp and Verdier 0}}}>
\on{Sp}_{\sK\to \sk}(\CG)\boxtimes \on{Sp}_{\sK\to \sk}(\BD_{\CY_\sK}(\CF)) \\
@AAA @AAA \\
(\Delta_{\CY_\sk})_!(\ul\sfe_{\CY_\sk}) @>>>
\on{Sp}_{\sK\to \sk}\circ (\Delta_{\CY_\sK})_!(\ul\sfe_{\CY_\sK}).
\endCD
\end{equation}

In this formulation, there is some additional homotopy coherence to note. 
Each term of \eqref{e:Sp and Verdier, tw}
can now be considered as a functor
from $\on{TwArr}(\Shv(\CY_\sK))^{\on{constr}}$, the 
twisted arrow category of constructible sheaves, 
to $\Shv(\CY_\sk \times_{\Spec(\sk)} \CY_\sk)$. Then the 
more functorial assertion of \eqref{i:sp 6} is that there is an commutative square of
functors \eqref{e:Sp and Verdier, tw}.

\end{rem}

\sssec{}

The following observation will be used in the sequel:

\begin{lem} \label{l:Sp of ULA}
Let $f:\CY_{\sR_0}\to \CZ_{\sR_0}$ be a map of algebraic stacks over $\Spec(\sR_0)$.
Let $\CF_{\CY,\sR_0}\in \Shv(\CY_{\sR_0})$ be ULA over $\CZ_{\sR_0}$. Then for any 
$\CF_{\CZ}\in \Shv(\CZ_{\sK})$, the naturally defined map
$$\CF_{\CY,\sk}\overset{*}\otimes (f_{\sk}^*\circ \on{Sp}_{\sK\to \sk}(\CG_\CZ))\to 
\on{Sp}_{\sK\to\sk}(\CF_{\CY,\sK}\overset{*}\otimes f_{\sK}^*(\CF_{\CZ}))$$
is a isomorphism, where $\CF_{\CY,\sk}$ (resp., $\CF_{\CY,\sK}$) denotes the *-restriction of $\CF_{\CY,\sR_0}$ to
$\CY_\sk$ (resp., $\CY_\sK$)
\end{lem} 

\begin{rem} \label{r:van cycles}

Note that when $\CZ_{\sR_0}=\Spec(\sR_0)$, the ULA condition simply means that
$$\Phi(\CF_{\CY,\sR_0})=0,$$
where $\Phi$ is the vanishing cycles functor.

\medskip

In other words, this is equivalent to the map
$$\CF_{\CY,\sk}\to \on{Sp}_{\sK\to\sk}(\CF_{\CY,\sK})$$
being an isomorphism. 

\end{rem} 

\sssec{}

In what follows we will study the resulting functor 
\begin{equation} \label{e:Sp BunG}
\on{Sp}_{\sK\to \sk}:\Shv(\Bun_{G,\sK})\to \Shv(\Bun_{G,\sk}).
\end{equation}

\ssec{Specialization and Hecke functors}

\sssec{}

Consider the version of the Hecke stack over $\Spec(\sR_0)$:
$$
\CD
\Bun_{G,\sR_0} @<{\hl}<< \on{Hecke}_{X,\sR_0} @>{\hr}>> \Bun_{G,\sR_0}  \\
& & @VV{s}V \\
& & X_{\sR_0}.
\endCD
$$

For an irreducible representation
$$V^\lambda\in \Rep(\cG), \quad \lambda\in \Lambda^+,$$
let 
$$\on{Sat}^{\on{nv}}_{X,\sR_0}(V^\lambda)\in \Shv(\on{Hecke}_{X,\sR_0})$$
be the object\footnote{The cohomological shift $[-2]$ in the formula below is designed in order to offset 
$\dim(X)+\dim(\sR_0)$.}
$$\IC_{\ol{\on{Hecke}}^\lambda_{X,\sR_0}}[-2].$$

\medskip

We claim:

\begin{thm} \label{t:Sat acycl}
The object $\on{Sat}^{\on{nv}}_{X,\sR_0}(V^\lambda)$ restricts to 
$\Sat^{\on{nv}}_{X,\sK}(V^\lambda)$ and $\Sat^{\on{nv}}_{X,\sk}(V^\lambda)$, respectively, and is ULA over
$\Bun_{G,\sR_0}\underset{\Spec(\sR_0)}\times X_{\sR_0}$ with respect to either of the projections. 
\end{thm} 

We will prove this theorem in \secref{ss:Sat acycl}.

\sssec{}

For a finite set $\CI$, consider the corresponding $\CI$-legged Hecke stack:
$$
\CD
\Bun_G @<{\hl}<< \on{Hecke}_{X^\CI} @>{\hr}>> \Bun_G  \\
& & @VV{s}V \\
& & X^\CI.
\endCD
$$

As in the usual geometric Satake theory, \thmref{t:Sat acycl} allows us to construct a family of functors
\begin{equation} \label{e:Sat nv legs}
\on{Sat}^{\on{nv}}_{X^{\CI},\sR_0}:\Rep(\cG)^{\otimes \CI}\to \Shv(\on{Hecke}_{X^{\CI},\sR_0}).
\end{equation} 

Moreover, the essential image of the functor \eqref{e:Sat nv legs} lies in the full subcategory of $\Shv(\on{Hecke}_{X^{\CI},\sR_0})$
that consists of objects that are ULA over $\Bun_{G,\sR_0}$ (in fact, over $\Bun_{G,\sR_0}\underset{\Spec(\sR_0)}\times X^\CI_{\sR_0}$). 

\sssec{}

Consider the Hecke functor
$$\on{H}_V:\Shv(\Bun_G)\to \Shv(\Bun_G\times X^{\CI}), 
\quad (\hl\times s)_!\left(\hr^*(-)\overset{*}\otimes \on{Sat}^{\on{nv}}_{X^{\CI}}(V)\right),\quad V\in \Rep(\cG)^{\otimes \CI}.$$

\medskip

The functors \eqref{e:Sat nv legs} give rise to natural transformations 
\begin{equation} \label{e:Hecke action Sp}
\on{H}_{V,\sk}\circ \on{Sp}_{\sK\to \sk}\to \on{Sp}_{\sK\to \sk}\circ \on{H}_{V,\sK}
\end{equation} 
as functors
$$\Shv(\Bun_{G,\sK})\to \Shv(\Bun_{G,\sk}\times X^{\CI}_{\sk}).$$

\medskip

We claim:

\begin{prop} \label{p:Sp and Hecke}
The natural transformations \eqref{e:Hecke action Sp} are isomorphisms. 
\end{prop} 

\begin{proof}

The map
$$\hl\times s:\on{Hecke}_{X^{\CI}}\to \Bun_G\times X^\CI$$
is proper, hence
$$(\hl\times s)_!\circ \on{Sp}_{\sK\to \sk}\to \on{Sp}_{\sK\to \sk}\circ (\hl\times s)_!$$
is an isomorphism.

\medskip

Now, the natural transformation
$$(\hr^*(-)\overset{*}\otimes \on{Sat}^{\on{nv}}_{X^{\CI}}(V))\circ \on{Sp}_{\sK\to \sk}\to
\on{Sp}_{\sK\to \sk}\circ (\hr^*(-)\overset{*}\otimes \on{Sat}^{\on{nv}}_{X^{\CI}}(V))$$
is also an isomorphism by \lemref{l:Sp of ULA} and the ULA property of the objects
$\on{Sat}^{\on{nv}}_{X^{\CI},\sR_0}(V)$.

\end{proof} 

\sssec{}

The functors $\on{Sat}^{\on{nv}}_{X^{\CI},\sR_0}$, $\CI\in\on{fSet}$ make the following system of diagrams of functors 
commute:
\begin{equation} \label{e:Sat nv legs compat}
\CD
\Rep(\cG)^{\otimes \CI_1} @>{\on{Sat}^{\on{nv}}_{X^{\CI_1},\sR_0}}>> \Shv(\on{Hecke}_{X^{\CI_1},\sR_0}) \\
@V{\Rep(\cG)^\phi}VV @VV{\Delta^*_\phi}V \\
\Rep(\cG)^{\otimes \CI_2} @>{\on{Sat}^{\on{nv}}_{X^{\CI_2},\sR_0}}>> \Shv(\on{Hecke}_{X^{\CI_2},\sR_0}),
\endCD
\end{equation} 
for $\phi:\CI_1\to \CI_2$, where:

\begin{itemize}

\item The left vertical arrow is induced by the symmetric monoidal structure on $\Rep(\cG)$;

\item The map $\Delta_\phi$ is (the base change of) the diagonal map $X^{\CI_2}\to X^{\CI_1}$.

\end{itemize}

This leads to a system of commutative diagrams of functors
$$
\CD
(\on{id}_{\Bun_G}\times \Delta_\phi)^*\circ \on{H}_{V,\sk}\circ \on{Sp}_{\sK\to \sk} @>>> 
(\on{id}_{\Bun_G}\times \Delta_\phi)^*\circ \on{Sp}_{\sK\to \sk}\circ \on{H}_{V,\sK} \\
@V{\sim}VV @VV{\sim}V \\
\on{H}_{\Res_\phi(V),\sk}\circ \on{Sp}_{\sK\to \sk} @>>> \on{Sp}_{\sK\to \sk}\circ \on{H}_{\Res_\phi(V),\sK},
\endCD
$$ 
which also depends functorially on $V\in \Rep(\cG)^{\otimes \CI}$. 

\ssec{Compatibility with Beilinson's projector}

\sssec{}

Recall Beilinson's projector, denoted $\sP$, which is an idempotent acting on $\Shv(\Bun_G)$
with essential image $\Shv_\Nilp(\Bun_G)$, see \cite[Sect. 13.4]{AGKRRV1}.

\medskip

We will denote by $\sP_\sk$ (resp., $\sP_\sK$) the corresponding endofunctor of $\Shv(\Bun_{G,\sK})$
(resp., $\Shv(\Bun_{G,\sk})$). 

\medskip

Let now $\sP_{\sK,\sk}$ be a version of $\sP_\sK$, where in we take
$\Gamma_!(\LS^{\on{restr}}_{\sG,\sk},-)$ instead of $\Gamma_!(\LS^{\on{restr}}_{\sG,\sK},-)$.
This is an idempotent endomorphism of $\Shv(\Bun_{G,\sK})$, equal to the composition of
$\sP_\sK$ and the projection onto the direct summand 
$$\Shv_{\Nilp}(\Bun_{G,\sK,\sk})\subset \Shv_{\Nilp}(\Bun_{G,\sK}).$$

\sssec{}

We claim: 

\begin{prop} \label{p:Sp and P}
There exists a canonical isomorphism
$$\on{Sp}_{\sK\to \sk}\circ \sP_{\sK,\sk}\simeq \sP_\sk\circ \on{Sp}_{\sK\to \sk}.$$
\end{prop}

\begin{proof} 

By construction (see \cite[Sect. 13.1.11]{AGKRRV1}), the functor $\sP_{\sK,\sk}$ is the colimit of functors $\sF_I$, $I\in \on{fSet}$,
each of which is the composition of functors of the following form:

\medskip

\begin{itemize}

\item(a) $\on{H}_{V_I}:\Shv(\Bun_G)\to \Shv(\Bun_G\times X^I)$ for $V_I\in \Rep(\cG)^{\otimes I}$;

\medskip

\item(b) $(-)\otimes \on{Ev}(V'_I):\Shv(\Bun_G\times X^I)\to \Shv(\Bun_G\times X^I)\otimes \QCoh(\LS^{\on{restr}}_\cG)$,
where:

\begin{itemize}

\item $V'_I\in  \Rep(\cG)^{\otimes I}$;

\item $\on{Ev}(V'_I)\in \qLisse(X^I)\otimes \QCoh(\LS^{\on{restr}}_\cG)$ is the tautological object corresponding to $V'_I$.

\end{itemize}

\medskip

\item(c) $\on{Id}\otimes \Gamma_!(\LS^{\on{restr}}_{\cG,\sk},-)$.

\end{itemize} 

\medskip

The functors in (a) commute with $\on{Sp}_{\sK\to \sk}$ by \propref{p:Sp and Hecke}. The functors in (b)
commute with $\on{Sp}_{\sK\to \sk}$ by \lemref{l:Sp of ULA}. The functors in (c) commute with 
$\on{Sp}_{\sK\to \sk}$ tautologically.

\end{proof} 

\begin{rem} \label{r:Hecke param}
For future use we remark that both Proposition \ref{p:Sp and Hecke} \and \ref{p:Sp and P}
remain valid if instead of $\Bun_{G,\sR_0}$ we consider 
$$\Bun_{G,\sR_0}\underset{\Spec(\sR_0)}\times \CY_{\sR_0},$$
for some stack $\CY_{\sR_0}$ over $\Spec(\sR_0)$. The proofs remain the same.

\end{rem}

\sssec{}

From \propref{p:Sp and P} we obtain: 

\begin{cor} \label{c:Sp Nilp}
The functor \eqref{e:Sp BunG} sends
$$\Shv_{\Nilp}(\Bun_{G,\sK,\sk})\subset \Shv_{\Nilp}(\Bun_{G,\sK})$$
to
$$\Shv_\Nilp(\Bun_{G,\sk})\subset \Shv(\Bun_{G,\sk}).$$
\end{cor}

\begin{rem}

Note that \propref{p:Sp and P} implies that the functor \eqref{e:Sp BunG} sends all of
$\Shv_{\Nilp}(\Bun_{G,\sK})$ to $\Shv_\Nilp(\Bun_{G,\sk})$, while killing the direct summands
of $\Shv_{\Nilp}(\Bun_{G,\sK})$ that are supported off 
$$\LS^{\on{restr}}_{\cG,\sk}\subset \LS^{\on{restr}}_{\cG,\sK}.$$

This follows from the fact that the functor $\sP_{\sK,\sk}$ applied to $\Shv_{\Nilp}(\Bun_{G,\sK})$
acts as a projector on $\Shv_{\Nilp}(\Bun_{G,\sK,\sk})$. 

\end{rem}

\begin{rem}

One can also prove that the functor \eqref{e:Sp BunG} sends $\Shv_{\Nilp}(\Bun_{G,\sK})\to \Shv_{\Nilp}(\Bun_{G,\sk})$
as follows:

\medskip

It follows from \cite[Theorem 14.4.3]{AGKRRV1} that the subcategory
$$\Shv_\Nilp(\Bun_G)\subset \Shv(\Bun_G)$$
can be characterized as follows: it consists of those objects $\CF\in \Shv(\Bun_G)$ for which for every $V\in \Rep(\cG)$
the object
$$\on{H}_V(\CF)\in \Shv(\Bun_G\times X)$$
belongs to the full subcategory\footnote{In \cite[Theorem 14.4.3]{AGKRRV1}, this is formulated as belonging to 
$\Shv(\Bun_G)\otimes \on{qLisse}(X)\subset \Shv(\Bun_G\times X)$; however, the apparently weaker condition of belonging to
$\Shv(\Bun_G)\otimes \Shv(X)$ is equivalent to this stronger condition: objects in the essential image of $\on{H}_V$ are ULA
over $X$, and any object in $\Shv(\Bun_G)\otimes \Shv(X)$ that is ULA over $X$ lies in $\Shv(\Bun_G)\otimes \on{qLisse}(X)$.}
$$\Shv(\Bun_G)\otimes \Shv(X)\subset \Shv(\Bun_G\times X).$$

Now, this property is preserved by the functor $\on{Sp}_{\sK\to \sk}$ by \eqref{e:Sp Kunneth}. 

\end{rem} 

\sssec{}

Thus, thanks to \corref{c:Sp Nilp} we obtain the functor \eqref{e:Sp Nilp}. 

\medskip

Note that the functor \eqref{e:Sp Nilp} satisfies Property (E) from \secref{sss:preoperties} holds. In fact, 
the functor $\on{Sp}_{\sK\to \sk}$ is t-exact by \secref{sss:Sp is exact}.

\sssec{}

We are now ready to establish Property (B) (from \secref{sss:preoperties}).
Indeed, it follows from \propref{p:Sp and Hecke}, since the spectral action
is determined by the Hecke functors (see \cite[Corollary 12.8.4(b)]{AGKRRV1}).

\ssec{Properties (C) and (D) of the specialization functor}

\sssec{}

Note that the construction in \cite[Sect. 3.3]{GLC1} (see \secref{sss:recall Poinc} below for a review)
makes sense over $\sR_0$, and produces an object
$$\on{Poinc}^{\on{Vac}}_{!,\sR_0}\in \Shv(\Bun_{G,\sR_0}).$$

In \secref{ss:Poinc acycl} we will prove:

\begin{thm} \label{t:Poinc acycl}
The object $\on{Poinc}^{\on{Vac}}_{!,\sR_0}$ is ULA over $\Spec(\sR_0)$.
\end{thm}

\sssec{}

By combining \thmref{t:Poinc acycl} with \propref{p:Sp and P} and \lemref{l:Sp of ULA},
we obtain Property (C). 

\sssec{}

Consider the stacks
$$\Bun_{P^-,\sR_0}\overset{j}\hookrightarrow \wt\Bun_{P^-,\sR_0}$$
and the projection
$$\wt\sfq^-:\wt\Bun_{P^-,\sR_0}\to \Bun_{M,\sR_0}.$$

\medskip

In \secref{ss:IC acycl} we will prove:

\begin{thm} \label{t:IC acycl}
The object 
$$j_!(\ul\sfe_{\Bun_{P^-,\sR_0}})\in \Shv(\wt\Bun_{P^-,\sR_0})$$
is ULA with respect to the projection $\wt\sfq^-$.
\end{thm}

\sssec{}

By combining \thmref{t:IC acycl} and \lemref{l:Sp of ULA} we deduce Property (D)
of \eqref{e:Sp Nilp}. 

\ssec{Specialization and temperedness}

In this subsection and formulate and prove a property of the functor $\on{Sp}_{\sK\to \sk}$
that describes its interaction with the \emph{temperization} functor.

\sssec{}

We claim:

\begin{prop} \label{p:Sp temp}
The functor 
 $$\on{Sp}_{\sK\to \sk}:\Shv(\Bun_{G,\sK}) \to \Shv(\Bun_{G,\sk})$$
 induces a functor
 $$\Shv(\Bun_{G,\sK})_{\on{temp}} \to \Shv(\Bun_{G,\sk})_{\on{temp}}$$
 so that the diagram
 $$
\CD
\Shv(\Bun_{G,\sK}) @>{\on{Sp}_{\sK\to \sk}}>> \Shv(\Bun_{G,\sk}) \\
@V{\bu^R}VV @VV{\bu^R}V \\
\Shv(\Bun_{G,\sK})_{\on{temp}} @>>>  \Shv(\Bun_{G,\sk})_{\on{temp}}
\endCD
$$
commutes.
\end{prop} 

We will give two proofs of this proposition.  

\sssec{First proof}

Let $x\in X$ be a chosen point (which we use to define $\Shv_\Nilp(\Bun_{G})_{\on{temp}}$) and let 
$\Sph(G)_x$ is the spherical Hecke category at $x$ (see \cite[Sect. 1.5]{GLC2}).

\medskip 

Recall (see \cite[Sect. 18.4]{AG1}) that the colocalization 
$$\bu^R:\Shv(\Bun_{G})\leftrightarrows \Shv_\Nilp(\Bun_{G})_{\on{temp}}:\bu$$
can be characterized in terms of the action of the monoidal category 
$$\IndCoh_\Nilp(\on{pt}\underset{\cg}\times \on{pt}/\on{Ad}(\cG))$$
on $\Shv(\Bun_{G})$, obtained as a composition of the 
\emph{derived Satake} equivalence
$$\Sph^{\on{spec}}(\cG)_x:=\IndCoh_\Nilp(\on{pt}\underset{\cg}\times \on{pt}/\on{Ad}(\cG)) \overset{\on{Sat}_G}\simeq \Sph(G)_x$$
and the natural action of $\Sph(G)_x$ on $\Shv(\Bun_{G})$ by Hecke functors. 

\medskip

Thus, in order to prove \propref{p:Sp temp}, it suffices to show that the functor $\on{Sp}_{\sK\to \sk}$
intertwines the actions of $\Sph^{\on{spec}}(\cG)_x$ on $\Shv(\Bun_{G,\sK})$ and on $\Shv(\Bun_{G,\sk})$,
where we let $x$ be a section $\Spec(\sR_0)\to X_{\sR_0}$. 

\medskip

By construction, the functor
$$\on{Sp}_{\sK\to \sk}:\Sph(G)_{x,\sK}\to \Sph(G)_{x,\sk}$$
intertwines the action of $\Sph(G)_{x,\sK}$ on $\Shv(\Bun_{G,\sK})$ with the action of 
$\Sph(G)_{x,\sk}$ on $\Shv(\Bun_{G,\sk})$.
 
\medskip 

Hence, it remains to show that the endomorphism $\phi$ of $\Sph^{\on{spec}}(\cG)_x$ that makes the 
following diagram commute
$$
\CD
\Sph^{\on{spec}}(\cG)_x @>{\phi}>> \Sph^{\on{spec}}(\cG)_x \\
@V{\on{Sat}_G}V{\sim}V @V{\sim}V{\on{Sat}_G}V \\
\Sph(G)_{x,\sK} @>{\on{Sp}_{\sK\to \sk}}>> \Sph(G)_{x,\sk}
\endCD
$$
is actually an automorphism. 

\medskip

We note that $\phi$ has the following properties:

\begin{itemize}

\item It is monoidal;

\item It makes the diagram 
$$
\CD
\Rep(\cG)  @>{\on{id}}>> \Rep(\cG)  \\
@VVV @VVV \\
\Sph^{\on{spec}}(\cG)_x @>{\phi}>> \Sph^{\on{spec}}(\cG)_x 
\endCD
$$
commute (this follows from \thmref{t:Sat acycl});

\item It induces the identity map on endomorphisms of the unit object
$$\on{Sym}(\cg[-2])^\cG\simeq \End_{\Sph^{\on{spec}}(\cG)_x}(\one_{\Sph^{\on{spec}}(\cG)_x})\overset{\phi}\to 
\End_{\Sph^{\on{spec}}(\cG)_x}(\one_{\Sph^{\on{spec}}(\cG)_x})\simeq \on{Sym}(\cg[-2])^\cG.$$

\end{itemize}

The third points follow from the fact that the composite
\begin{multline*} 
\on{Sym}(\ft^*[-2])^W
\simeq \on{C}^\cdot(\on{pt}/G_\sK)\simeq \End_{\Sph(G)_{x,\sK}}(\one_{\Sph(G)_{x,\sK}})\overset{\on{Sp}_{\sK\to \sk}}\longrightarrow  \\
\to \End_{\Sph(G)_{x,\sK}}(\one_{\Sph(G)_{x,\sk}})
\simeq \on{C}^\cdot(\on{pt}/G_\sk)\simeq \on{Sym}(\ft^*[-2])^W
\end{multline*}
is the identity map, while the identification $\on{Sym}(\cg[-2])^\cG\simeq \End_{\Sph^{\on{spec}}(\cG)_x}(\one_{\Sph^{\on{spec}}(\cG)_x})$ equals 
\begin{multline*} 
\on{Sym}(\cg[-2])^\cG\simeq \on{Sym}(\check\ft[-2])^W\simeq \\
\simeq \on{Sym}(\ft^*[-2])^W\simeq \End_{\Sph(G)_x}(\one_{\Sph(G)_x})\overset{\on{Sat}_G}\simeq \End_{\Sph^{\on{spec}}(\cG)_x}(\one_{\Sph^{\on{spec}}(\cG)_x}).
\end{multline*} 

We now claim:

\begin{lem} \label{l:Satake rigid}
Any endomorphism of $\Sph^{\on{spec}}(\cG)_x$ that has the above three properties is an automorphism.
\end{lem}

The proof of the lemma is given in \secref{sss:Satake rigid} below.

\qed[First proof of \propref{p:Sp temp}]

\sssec{Second proof}

We have to show that the functor $\on{Sp}_{\sK\to \sk}$ sends 
$$\on{ker}(\bu^R):\Shv(\Bun_{G,\sK}) \overset{\bu^R}\twoheadrightarrow \Shv(\Bun_{G,\sK})_{\on{temp}}$$
to
$$\on{ker}(\bu^R):\Shv(\Bun_{G,\sk}) \overset{\bu^R}\twoheadrightarrow \Shv(\Bun_{G,\sk})_{\on{temp}}.$$

\medskip

Recall the characterization of the kernel of 
$$\Shv(\Bun_G) \overset{\bu^R}\to \Shv(\Bun_G)_{\on{temp}}$$
in \cite[Sect. 4.3]{FR}.

\medskip

The required assertion follows from the fact that the functors involved in this characterization commute with
$\on{Sp}_{\sK\to \sk}$. 

\qed[Second proof of \propref{p:Sp temp}]

\sssec{Proof of \lemref{l:Satake rigid}} \label{sss:Satake rigid}

We rewrite
$$\Sph^{\on{spec}}(\cG)_x\simeq \Sym(\cg[-2])\mod^{\cG}.$$

Using the first two properties, we can de-equivariantize both sides, i.e.,
apply $\Vect\underset{\Rep(\cG)}\otimes -$, and obtain a functor
$$\on{Sym}(\cg[-2])\mod \overset{\wt\phi}\to \on{Sym}(\cg[-2])\mod$$
with the following properties:

\begin{itemize}

\item It is $\cG$-equivariant;

\item It sends $\on{Sym}(\cg[-2])\in \on{Sym}(\cg[-2])\mod$ to itself;

\item It makes the diagram
$$
\CD
\on{Sym}(\cg[-2])^\cG @>{\on{id}}>> \on{Sym}(\cg[-2])^\cG  \\
@VVV @VVV \\
\on{Sym}(\cg[-2]) & & \on{Sym}(\cg[-2]) \\
@V{\sim}VV @V{\sim}VV \\
\End_{\on{Sym}(\cg[-2])\mod}(\on{Sym}(\cg[-2])) @>{\wt\phi}>> \End_{\on{Sym}(\cg[-2])\mod}(\on{Sym}(\cg[-2])) 
\endCD
$$
commute.

\end{itemize}

It suffices to show that the bottom horizontal arrow in this diagram is an isomorphism. This arrow is a map
of algebras, hence is determined by a map of vector spaces. 
$$\cg[-2]\to \on{Sym}(\cg[-2]).$$

The grading forces the latter map to come from a map of vector spaces $\cg\to \cg$. By $\cG$-equivariance,
the latter map acts as a scalar on each simple factor. 

\medskip

However, the commutation of the above diagram forces this scalar to be $1$.

\qed[\lemref{l:Satake rigid}]

\section{Proof of Property (A)} \label{s:unit}

The goal of this section is to establish Property (A) of the functor \eqref{e:Sp Nilp}. 

\ssec{The key input}

\sssec{}

Consider the object
\begin{equation} \label{e:diagonal}
(\Delta_{\Bun_{G,\sR_0}})_!(\ul\sfe_{\Bun_{G,\sR_0}})\in \Shv(\Bun_{G,\sR_0}\underset{\Spec(\sR_0)}\times \Bun_{G,\sR_0}).
\end{equation}

\sssec{}

The key input in the proof of Property (A) is provided by the following result:

\begin{thm} \label{t:diag acycl}
The object \eqref{e:diagonal} is ULA over $\Spec(\sR_0)$.
\end{thm} 

The theorem will be proved in \secref{ss:diag acycl}. We now proceed with the proof of Property (A). 

\sssec{}

Combining \thmref{t:diag acycl} with \lemref{l:Sp of ULA}, we obtain:

\begin{cor} \label{c:Sp of diag}
The natural map 
$$(\Delta_{\Bun_{G,\sk}})_!(\ul\sfe_{\Bun_{G,\sk}})\simeq 
(\Delta_{\Bun_{G,\sk}})_!\circ \on{Sp}_{\sK\to \sk}(\ul\sfe_{\Bun_{G,\sK}})\to
\on{Sp}_{\sK\to \sk}\circ (\Delta_{\Bun_{G,\sK}})_!(\ul\sfe_{\Bun_{G,\sK}})$$
is an isomorphism.
\end{cor}

Using the commutation of specialization with Verdier duality, from \corref{c:Sp of diag} we obtain:

\begin{cor} \label{c:Sp diag omega} 
The canonical map
$$\on{Sp}_{\sK\to \sk}\circ (\Delta_{\Bun_{G,\sK}})_*(\omega_{\Bun_{G,\sK}})\to
(\Delta_{\Bun_{G,\sk}})_*\circ \on{Sp}_{\sK\to \sk}(\omega_{\Bun_{G,\sK}})\simeq 
(\Delta_{\Bun_{G,\sk}})_*(\omega_{\Bun_{G,\sk}})$$
is an isomorphism. 
\end{cor} 

\ssec{Harder-Narasimhan strata}

\sssec{}

Recall the notion of \emph{contractive} locally closed substack, see \cite[Sect. 5.2.1]{DG1}. 

\medskip

For $\theta\in \Lambda^+_\BQ$, let 
$$\Bun_G^{(\leq \theta)} \subset \Bun_G$$
be the open union of Harder-Narasimhan strata, as defined in \cite[Sect. 7.3.4]{DG1}.

\medskip

The main technical result of the paper \cite{DG1}, proved in Sect. 9.3 of {\it loc. cit.}, says that for all $\theta$ large enough 
and $\theta'\geq \theta$, the (locally) closed substack
$$\Bun_G^{(\leq \theta')}-\Bun_G^{(\leq \theta)}\subset \Bun_G^{(\leq \theta')}$$
is contractive.

\medskip

The proof of this result applies equally well in the relative situation over $\Spec(\sR_0)$. 

\sssec{}

Let $\theta\leq \theta'$ be as above and consider the corresponding open embedding
$$\jmath_{\theta,\theta'}:\Bun_G^{(\leq \theta)}\hookrightarrow \Bun_G^{(\leq \theta')}.$$

We are going to prove:

\begin{prop} \label{p:Sp trunc}
The natural transformations
$$(\jmath_{\theta,\theta'})_!\circ \on{Sp}_{\sK\to \sk}\to \on{Sp}_{\sK\to \sk}\circ (\jmath_{\theta,\theta'})_!$$
and
$$\on{Sp}_{\sK\to \sk}\circ (\jmath_{\theta,\theta'})_*\to (\jmath_{\theta,\theta'})_*\circ \on{Sp}_{\sK\to \sk}$$
are isomorphisms.
\end{prop}

\begin{proof}

We will prove the first isomorphism, as the second one is similar. 

\medskip

It is enough to prove the assertion after applying pullback with respect to a smooth surjective map.
Hence, by the definition of contractiveness, it is enough to prove the more general \propref{p:Sp contr} 
below.

\end{proof}

\sssec{} \label{sss:contr}

Let $\CY_{\sR_0}$ and $\CZ_{\sR_0}$ be as in \cite[Sect. 5.1.1]{DG1}. Denote
$$\CU_{\sR_0}:=\CY_{\sR_0}-\CZ_{\sR_0}\overset{\jmath}\hookrightarrow \CY_{\sR_0}.$$

\begin{prop} \label{p:Sp contr}
The natural transformation
$$\jmath_!\circ \on{Sp}_{\sK\to \sk}\to \on{Sp}_{\sK\to \sk}\circ \jmath_!$$
is an isomorphism.
\end{prop}

\begin{proof}

Let $\imath$ denote the closed embedding $\CZ_{\sR_0}\to \CY_{\sR_0}$. It is enough to show that
the natural transformation
\begin{equation} \label{e:Sp contr}
\imath^*\circ \on{Sp}_{\sK\to \sk}\to \on{Sp}_{\sK\to \sk}\circ \imath^*,
\end{equation} 
as functors $\Shv(\CY_\sK)\to \Shv(\CZ_\sk)$, 
is an isomorphism.

\medskip

As in \cite[Sects. 5.1.3-5.1.5]{DG1}, we can assume that\footnote{In the formulas below, $\on{pt}:=\Spec(\sR_0)$, and similarly
for $\BA^1$ and $\BG_m$.} $\CZ=\on{pt}/\BG_m\times Z$ and 
$\CY=\BA^n/\BG_m\times Z$ for some base $Z$, where $\BG_m$ acts on $\BA^n$ via the
$m$th power of the standard character, where $m>0$. Further, applying blow-up 
(see \cite[Sect. 5.1.6]{DG1}), we can assume that $n=1$. 

\medskip

It is enough to prove that the map \eqref{e:Sp contr} is an isomorphism on the generators.
We take the generators to be of the form
$$(f_m)_*(\CF_{\BA^1})\boxtimes \CF_Z,$$
where 

\begin{itemize}

\item $f_m$ is the ``raising to the power $m$" map $\BA^1\to \BA^1$;

\item $\CF_Z\in \Shv(Z)$,

\item $\CF_{\BA^1}\in \Shv(\BA^1)^{\BG_m}$ is either $\delta_{0,\BA^1}$ or $\sfe_{\BA^1}$.

\end{itemize}

In the above cases, the fact that \eqref{e:Sp contr} is an isomorphism is evident. 

\end{proof} 

\ssec{Specialization for the ``co"-category}

\sssec{} \label{sss:Sp on co}

Recall the category $\Shv(\Bun_G)_{\on{co}}$, see \cite[Sects. 2.5 and C.2]{AGKRRV2}. By definition,
$$\Shv(\Bun_G)_{\on{co}}\simeq \underset{\theta\in \Lambda^+_\BQ}{\on{colim}}\, \Shv(\Bun_G^{(\leq \theta)}),$$
where the colimit is taken with respect to the functors $(\jmath_{\theta,\theta'})_*$.

\medskip

From \propref{p:Sp trunc} we obtain that there exists a well-defined functor
\begin{equation} \label{e:Sp for co}
\on{Sp}^{\on{co}}_{\sK\to \sk}:\Shv(\Bun_{G,\sK})_{\on{co}}\to \Shv(\Bun_{G,\sk})_{\on{co}}
\end{equation} 
that makes the following diagrams commute
\begin{equation} \label{e:Sp and j co}
\CD
\Shv(\Bun_{G,\sK})_{\on{co}} @>{\on{Sp}^{\on{co}}_{\sK\to \sk}}>> \Shv(\Bun_{G,\sk})_{\on{co}} \\
@A{(\jmath_\theta)_{\on{co,*}}}AA @AA{(\jmath_\theta)_{\on{co,*}}}A \\
\Shv(\Bun_{G,\sK}^{(\leq \theta)}) @>{\on{Sp}_{\sK\to \sk}}>> \Shv(\Bun_{G,\sk}^{(\leq \theta)}),
\endCD
\end{equation}
where:

\begin{itemize}

\item $\jmath_\theta$ denotes the open embedding $\Bun_G^{(\leq \theta)}\hookrightarrow \Bun_G$;

\item $(\jmath_\theta)_{\on{co,*}}$ denotes the corresponding tautological functor $\Shv(\Bun_G^{(\leq \theta)}) \to \Shv(\Bun_G)_{\on{co}}$
(not to be confused with $(\jmath_\theta)_*:\Shv(\Bun_G^{(\leq \theta)}) \to \Shv(\Bun_G)$.

\end{itemize}

\sssec{}

Moreover, the following diagram commutes tautologically
\begin{equation} \label{e:nv Ps-Id}
\CD
\Shv(\Bun_{G,\sK})_{\on{co}} @>{\on{Sp}^{\on{co}}_{\sK\to \sk}}>> \Shv(\Bun_{G,\sk})_{\on{co}} \\
@VV{\on{Ps-Id}^{\on{nv}}}V  @VV{\on{Ps-Id}^{\on{nv}}}V \\
\Shv(\Bun_{G,\sK})  @>{\on{Sp}_{\sK\to \sk}}>> \Shv(\Bun_{G,\sk}),
\endCD
\end{equation}
where 
$$\on{Ps-Id}^{\on{nv}}:\Shv(\Bun_G)_{\on{co}}\to \Shv(\Bun_G)$$
is as in \cite[Sect. C.2.3]{AGKRRV2}.

\sssec{}

Recall that for $\CF\in \Shv(\Bun_G)^c$, its Verdier dual is well-defined as an object
$$\BD_{\Bun_G}(\CF)\in  \Shv(\Bun_G)_{\on{co}}.$$

\medskip

The commutation of specialization with Verdier duality implies that for $\CF\in \Shv(\Bun_{G,\sK})^c$
we have a canonical isomorphism 
\begin{equation} \label{e:Sp verdier co}
\on{Sp}^{\on{co}}_{\sK\to \sk}(\BD_{\Bun_{G,\sK}}(\CF))\simeq \BD_{\Bun_{G,\sk}}(\on{Sp}_{\sK\to \sk}(\CF))
\end{equation} 
as objects in $\Shv(\Bun_{G,\sk})_{\on{co}}$.

\sssec{}

We now consider the ``mixed" category
$$\Shv(\Bun_G\times \Bun_G)_{\on{co}_2},$$
defined in \cite[Sect. C.4.2]{AGKRRV2}.

\medskip

By definition,
$$\Shv(\Bun_G\times \Bun_G)_{\on{co}_2}=\underset{\theta\in \Lambda^+_\BQ}{\on{colim}}\, \Shv(\Bun_G\times \Bun_{G,\sK}^{(\leq \theta)}),$$
where the colimit is taken with respect to the functors
$(\on{id}\times \jmath_{\theta,\theta'})_*$.

\medskip

According to \cite[Sect. C.4.3]{AGKRRV2}, the functor
\begin{equation} \label{e:mxd as limit}
\Shv(\Bun_G\times \Bun_G)_{\on{co}_2}\to
\underset{\CU}{\on{lim}}\, \Shv(\CU\times \Bun_G)_{\on{co}}
\end{equation}
is an equivalence, where:

\begin{itemize}

\item $\CU$ runs over the poset of quasi-compact open substacks of $\Bun_G$;

\item For $\CU\subset \CU'$, the functor $\Shv(\CU'\times \Bun_G)_{\on{co}}\to \Shv(\CU\times \Bun_G)_{\on{co}}$
is given by restriction.

\end{itemize}

\medskip

The discussion in \secref{sss:Sp on co} applies equally to the mixed situation.

\medskip

In addition, we have a tautologically defined functor
$$\on{Ps-Id}^{\on{nv}}:\Shv(\Bun_G\times \Bun_G)_{\on{co}_2}\to \Shv(\Bun_G\times \Bun_G)$$
and the corresponding counterpart of diagram \eqref{e:nv Ps-Id} commutes. 

\sssec{}

We will consider the corresponding functor
$$\on{Sp}_{\sK\to \sk}^{\on{co}_2}:
\Shv(\Bun_{G,\sK}\times \Bun_{G,\sK})_{\on{co}_2}\to \Shv(\Bun_{G,\sk}\times \Bun_{G,\sk})_{\on{co}_2}.$$

\sssec{}

Recall now the object
$$(\Delta_{\Bun_G})^{\on{fine}}_*(\omega_{\Bun_G})\in \Shv(\Bun_G\times \Bun_G)_{\on{co}_2},$$
defined in \cite[Sect. C.4.6]{AGKRRV2}.

\medskip 

In terms of the equivalence \eqref{e:mxd as limit}, the restriction of $(\Delta_{\Bun_G})^{\on{fine}}_*(\omega_{\Bun_G})$
to a given $\CU\times \Bun_G$ is
$$(\on{id}\times \jmath)_{\on{co},*}\circ (\Delta_\CU)_*(\omega_\CU), \quad \jmath:\CU\hookrightarrow \Bun_G.$$

We have
$$\on{Ps-Id}^{\on{nv}}((\Delta_{\Bun_G})^{\on{fine}}_*(\omega_{\Bun_G}))\simeq (\Delta_{\Bun_G})_*(\omega_{\Bun_G}),$$
as objects of $\Shv(\Bun_G\times \Bun_G)$. 

\sssec{}

We are going to prove:

\begin{prop} \label{p:Sp omega fine}
There exist a unique isomorphism
\begin{equation} \label{e:Sp omega fine}
\on{Sp}^{\on{co}_2}_{\sK\to \sk}\left((\Delta_{\Bun_{G,\sK}})^{\on{fine}}_*(\omega_{\Bun_{G,\sK}})\right)\simeq 
(\Delta_{\Bun_{G,\sk}})^{\on{fine}}_*(\omega_{\Bun_{G,\sk}})
\end{equation}
as objects of $\Shv(\Bun_{G,\sk}\times \Bun_{G,\sk})_{\on{co}_2}$ that makes the diagram
$$
\CD
\on{Ps-Id}^{\on{nv}}\circ \on{Sp}^{\on{co}_2}_{\sK\to \sk}\left((\Delta_{\Bun_{G,\sK}})^{\on{fine}}_*(\omega_{\Bun_{G,\sK}})\right)
@>{\sim}>{\text{\eqref{e:Sp omega fine}}}>
\on{Ps-Id}^{\on{nv}}\circ (\Delta_{\Bun_{G,\sk}})^{\on{fine}}_*(\omega_{\Bun_{G,\sk}}) \\
@V{\sim}VV @VV{\sim}V \\
\on{Sp}_{\sK\to \sk}\circ (\Delta_{\Bun_{G,\sK}})_*(\omega_{\Bun_{G,\sK}}) @>{\sim}>{\text{\corref{c:Sp diag omega}}}> 
(\Delta_{\Bun_{G,\sk}})_*(\omega_{\Bun_{G,\sk}})
\endCD
$$
commute.
\end{prop}

\begin{proof}

Using the equivalence \eqref{e:mxd as limit}, we have to show that there exists a unique family
of isomorphisms
\begin{equation} \label{e:Sp omega fine U}
\on{Sp}^{\on{co}_2}_{\sK\to \sk}\left((\Delta_{\Bun_{G,\sK}})^{\on{fine}}_*(\omega_{\Bun_{G,\sK}})\right)|_{\CU_\sk\times \Bun_{G,\sk}}\simeq
(\Delta_{\Bun_{G,\sk}})^{\on{fine}}_*(\omega_{\Bun_{G,\sk}})|_{\CU_\sk\times \Bun_{G,\sk}}
\end{equation} 
as objects of $\Shv(\CU_\sk\times \Bun_{G,\sk})_{\on{co}}$, compatible via $\on{Ps-Id}^{\on{nv}}$ with the isomorphism
$$\on{Sp}_{\sK\to \sk}\circ (\Delta_{\Bun_{G,\sK}})_*(\omega_{\Bun_{G,\sK}})|_{\CU_\sk\times \Bun_{G,\sk}}  \simeq
(\Delta_{\Bun_{G,\sk}})_*(\omega_{\Bun_{G,\sk}})|_{\CU_\sk\times \Bun_{G,\sk}},$$
induced by the isomorphism of \corref{c:Sp diag omega}, where $\CU$ runs over the poset of quasi-compact open substacks of $\Bun_G$. 

\medskip

We can replace the poset of all $\CU$ by a cofinal one that consists of the substacks 
$\Bun_G^{(\leq \theta)}$ for $\theta$ sufficiently large. 

\medskip

The left-hand side in \eqref{e:Sp omega fine U} is by definition
\begin{equation} \label{e:Sp omega fine 1}
\on{Sp}^{\on{co}_2}_{\sK\to \sk}\circ (\on{id}\times \jmath_\theta)_{\on{co},*}\circ 
(\Delta_{\Bun_{G,\sK}^{(\leq \theta)}})_*(\omega_{\Bun_{G,\sK}^{(\leq \theta)}}).
\end{equation} 

Using \eqref{e:Sp and j co}, we write it as 
\begin{equation} \label{e:Sp omega fine 2}
(\on{id}\times \jmath_\theta)_{\on{co},*}\left(\on{Sp}_{\sK\to \sk} \circ (\Delta_{\Bun_{G,\sK}^{(\leq \theta)}})_*(\omega_{\Bun_{G,\sK}^{(\leq \theta)}})\right).
\end{equation} 

The right-hand in \eqref{e:Sp omega fine U} is
\begin{equation} \label{e:Sp omega fine 3}
(\on{id}\times \jmath_\theta)_{\on{co},*}\left((\Delta_{\Bun_{G,\sk}^{(\leq \theta)}})_*(\omega_{\Bun_{G,\sk}^{(\leq \theta)}})\right).
\end{equation}

By \corref{c:Sp diag omega}, the expressions in \eqref{e:Sp omega fine 2} and \eqref{e:Sp omega fine 3} become isomorphic after applying the functor
$$\on{Ps-Id}^{\on{nv}}:\Shv(\Bun_{G,\sk}^{(\leq \theta)}\times \Bun_{G,\sk})_{\on{co}}\to 
\Shv(\Bun_{G,\sk}^{(\leq \theta)}\times \Bun_{G,\sk}).$$

Now, the required assertion follows from the fact that the functor $\on{Ps-Id}^{\on{nv}}$ is fully faithful on the essential image of
$$(\on{id}\times \jmath_\theta)_{\on{co},*}:\Shv(\Bun_{G,\sk}^{(\leq \theta)}\times \Bun_{G,\sk}^{(\leq \theta)})\to
\Shv(\Bun_{G,\sk}^{(\leq \theta)}\times \Bun_{G,\sk})_{\on{co}}.$$

\end{proof}

\ssec{Method of proof}

In this subsection we launch the proof of Property (A) proper. 

\sssec{} \label{sss:op functor}

The proof is based on the following principle:

\medskip

Let $F:\bC_1\to \bC_2$ be a functor between compactly generated
categories. Assume that $F$ preserves compactness, so it admits a continuous right adjoint, denoted $F^R$.
Denote 
$$F^{\on{op}}:=(F^R)^\vee, \quad \bC_1^\vee\to \bC_2^\vee.$$

Explicitly, identifying
$$\bC^\vee_i:=\on{Ind}((\bC_i^c)^{\on{op}}),$$
the functor $F^{\on{op}}$ is obtained by ind-extending the same-named functor on compact objects.

\sssec{}

Consider the tautological map
\begin{equation} \label{e:map of units}
(F\otimes F^{\on{op}})(\bu_{\bC_1})\to \bu_{\bC_2},
\end{equation} 
where $\bu_{\bC_i}\in \bC_i\otimes \bC^\vee_i$ is the unit of the self-duality. 

\medskip

The map \eqref{e:map of units} is characterized as follows. For $\bc_1\in \bC_1^c$ with formal
dual $\bc_1^\vee\in \bC_1^\vee$, the diagram
\begin{equation} \label{e:char map of units}
\CD
(F\otimes F^{\on{op}})(\bu_{\bC_1}) @>{\text{\eqref{e:map of units}}}>> \bu_{\bC_2} \\
@AAA @AAA \\
(F\otimes F^{\on{op}})(\bc_1\otimes \bc_1^\vee) @>{\sim}>> F(\bc_1)\otimes F(\bc_1)^\vee,
\endCD
\end{equation} 
commutes, where the vertical arrows are the canonical maps
\begin{equation} \label{e:map to unit}
\on{can}_\bc:\bc\otimes \bc^\vee\to u_{\bC}, \quad \bc\in \bC^c
\end{equation} 
for a compactly generated category $\bC$.

\begin{rem} 

Properly speaking, the object $u_{\bC}\in \bC\otimes \bC^\vee$ can be characterized as the colimit
over the the category of twisted arrows in $\bC^c$ that sends
$$(\phi:\bc'' \to \bc')\in \on{TwArr}(\bC^c)\, \mapsto (\bc'\otimes (\bc'')^\vee)\in \bC\otimes \bC^\vee.$$

Indeed, for $\phi:\bc'' \to \bc'$ as above, the corresponding map $\bc'\otimes (\bc'')^\vee\to u_\bC$
is either of the circuits in the following commutative diagram
$$
\CD
\bc'\otimes (\bc'')^\vee @>{\on{id}\otimes \phi^\vee}>> \bc'\otimes (\bc')^\vee \\
@V{\phi\otimes \on{id}}VV  @VV{\on{can}_{\bc'}}V \\
\bc''\otimes (\bc'')^\vee @>{\on{can}_{\bc''}}>>  u_{\bC}.
\endCD
$$

In what follows, we have chosen to simplify the exposition by not explicitly
mentioning the twisted arrows functoriality and working only with a single object at a
time. To make the discussion complete, one replaces the eventual reference to 
\eqref{e:Sp and Verdier} 
with a reference to \remref{r:sp and tw arr}.

\end{rem}

\sssec{}

We have:

\begin{lem} \label{l:criter for loc}
The functor $F$ is a Verdier quotient if and only if the map \eqref{e:map of units} is an isomorphism.
\end{lem}

\begin{proof}

We need to show that the counit of the adjunction
$$F\circ F^R\to \on{Id}$$
is an isomorphism. I.e., we have to show that the map
\begin{equation} \label{e:counit adj}
((F\circ F^R)\otimes \on{Id})(\bu_{\bC_2})\to \bu_{\bC_2}
\end{equation} 
is an isomorphism. 

\medskip

Note that $F^R$ identifies with the functor dual to $F^{\on{op}}$. Hence, we can rewrite
\begin{multline*} 
((F\circ F^R)\otimes \on{Id})(\bu_{\bC_2})=
(F\otimes \on{Id})\circ (F^R\otimes \on{Id})\circ(\bu_{\bC_2})\simeq 
(F\otimes \on{Id})\circ ((F^{\on{op}})^\vee\otimes \on{Id})\circ(\bu_{\bC_2})\simeq \\
\simeq (F\otimes \on{Id})\circ (\on{Id}\otimes F^{\on{op}})(\bu_{\bC_1})=
(F\otimes F^{\on{op}})(\bu_{\bC_1}).
\end{multline*}

Under this identification, the map \eqref{e:counit adj} becomes the map \eqref{e:map of units}.

\end{proof}

\sssec{}

We will show that the situation described above takes place for
$$\bC_1:=\Shv_\Nilp(\Bun_{G,\sK}),\,\, \bC_1:=\Shv_\Nilp(\Bun_{G,\sk}),\,\, F=\on{Sp}_{\sK\to \sk}.$$

\sssec{}

First, recall that the functor \eqref{e:Sp Nilp} preserves compactness, see \secref{sss:Sp compactness}
(this relies on Properties (B), (C) and (D) , which are proved independently of Property (A)). 

\medskip

Next, recall that (when working over a field), according to \cite[Sects. 2.5.8, Corollary 2.6.5 and Proposition 2.7.6]{AGKRRV2}, the dual of the category 
$\Shv_\Nilp(\Bun_G)$ identifies with
$$\Shv_\Nilp(\Bun_G)_{\on{co}}\subset \Shv(\Bun_G)_{\on{co}}$$
in such a way that the unit of the duality is the object
\begin{equation} \label{e:unit Nilp}
\sP((\Delta_{\Bun_G})^{\on{fine}}_*(\omega_{\Bun_G})),
\end{equation}
where: 

\begin{itemize}

\item $\sP$ denotes here Beilinson's projector, applied along the first factor;

\medskip

\item The object in \eqref{e:unit Nilp}, which a priori lies in $\Shv(\Bun_G\times \Bun_G)_{\on{co}_2}$ 
belongs in fact to
$$\Shv_\Nilp(\Bun_G)\otimes \Shv_\Nilp(\Bun_G)_{\on{co}}\subset \Shv(\Bun_G\times \Bun_G)_{\on{co}_2}.$$ 

\end{itemize}

\sssec{} \label{sss:Verdier on Nilp}

Furthermore (still working over a field), the corresponding equivalence
$$(\Shv_\Nilp(\Bun_G)^c)^{\on{op}}\simeq (\Shv_\Nilp(\Bun_G)_{\on{co}})^c$$
is induced by the Verdier duality equivalence
$$(\Shv(\Bun_G)^c)^{\on{op}}\simeq (\Shv(\Bun_G)_{\on{co}})^c,$$
where we are using the fact that the embedidng
$$\Shv_\Nilp(\Bun_G)\hookrightarrow \Shv(\Bun_G)^c$$
preserves compactness, by \thmref{t:compactness}.

\sssec{}

A variant of the above discussion applies to $\Shv_\Nilp(\Bun_{G,\sK,\sk})$, where we only need to replace
$\sP$ by $\sP_{\sK,\sk}$. 

\sssec{}

We claim that the functor $\on{Sp}^{\on{co}}_{\sK\to \sk}$ of \eqref{e:Sp for co} sends
$$\Shv_\Nilp(\Bun_{G,\sK,\sk})_{\on{co}}\to \Shv_\Nilp(\Bun_{G,\sk})_{\on{co}}$$
and identifies with $(\on{Sp}_{\sK\to \sk})^{\on{op}}$ (in the notations of \secref{sss:op functor}).

\medskip

Indeed, this follows by combining \eqref{e:Sp verdier co} and \secref{sss:Verdier on Nilp}. 

\sssec{}

Thus, by \lemref{l:criter for loc}
 it remains to show that the map \eqref{e:map of units}, which in our case is the map
\begin{equation} \label{e:map of units to verify}
(\on{Sp}_{\sK\to \sk}\otimes \on{Sp}^{\on{co}}_{\sK\to \sk})
\left(\sP_{\sK,\sk}((\Delta_{\Bun_{G,\sK}})^{\on{fine}}_*(\omega_{\Bun_{G,\sK}}))\right)\to
\sP_{\sk}((\Delta_{\Bun_{G,\sk}})^{\on{fine}}_*(\omega_{\Bun_{G,\sk}})),
\end{equation}
is an isomorphism in $\Shv_\Nilp(\Bun_{G,\sk})\otimes \Shv_\Nilp(\Bun_{G,\sk})_{\on{co}}$. 

\ssec{Verification}

\sssec{}

Applying the (fully faithful) functor
\begin{multline*} 
\Shv_\Nilp(\Bun_{G,\sk})\otimes \Shv_\Nilp(\Bun_{G,\sk})_{\on{co}}\to \\
\to \Shv(\Bun_{G,\sk})\otimes \Shv(\Bun_{G,\sk})_{\on{co}}\to \Shv(\Bun_{G,\sk}\times \Bun_{G,\sk})_{\on{co}_2},
\end{multline*}
it suffices to show that \eqref{e:map of units to verify} is an isomorphism of objects in 
$\Shv(\Bun_{G,\sk}\times \Bun_{G,\sk})_{\on{co}_2}$.

\sssec{}

By \propref{p:Sp and P}, we identify the left-hand side in \eqref{e:map of units to verify} with
$$\sP_{\sk}\circ \on{Sp}^{\on{co}_2}_{\sK\to \sk}\left((\Delta_{\Bun_{G,\sK}})^{\on{fine}}_*(\omega_{\Bun_{G,\sK}})\right),$$
and using \propref{p:Sp omega fine}, we identify it further with 
$$\sP_{\sk}\left((\Delta_{\Bun_{G,\sk}})^{\on{fine}}_*(\omega_{\Bun_{G,\sk}})\right),$$
which is the right-hand side in \eqref{e:map of units to verify}.

\medskip

It remains to show that the map in \eqref{e:map of units to verify} corresponds under the above
identifications to the identity map. 

\sssec{}

By \eqref{e:char map of units}, we have to show that for $\CF\in \Shv_\Nilp(\Bun_{G,\sK,\sk})^c$, the isomorphism 
\begin{multline} \label{e:map of units to verify 1} 
\on{Sp}^{\on{co}_2}_{\sK\to \sk}
\left(\sP_{\sK,\sk}((\Delta_{\Bun_{G,\sK}})^{\on{fine}}_*(\omega_{\Bun_{G,\sK}}))\right) \overset{\text{ \propref{p:Sp and P}}}\simeq \\
\simeq \sP_{\sk}\circ \on{Sp}^{\on{co}_2}_{\sK\to \sk}\left((\Delta_{\Bun_{G,\sK}})^{\on{fine}}_*(\omega_{\Bun_{G,\sK}})\right)
\overset{\text{\propref{p:Sp omega fine}}}\simeq
\sP_{\sk}\left((\Delta_{\Bun_{G,\sk}})^{\on{fine}}_*(\omega_{\Bun_{G,\sk}})\right)
\end{multline}
makes the following diagram commute
\begin{equation} \label{e:map of units to verify 2} 
\CD
\on{Sp}^{\on{co}_2}_{\sK\to \sk}
\left(\sP_{\sK,\sk}((\Delta_{\Bun_{G,\sK}})^{\on{fine}}_*(\omega_{\Bun_{G,\sK}}))\right) @>{\text{\eqref{e:map of units to verify 1}}}>> 
\sP_{\sk}\left((\Delta_{\Bun_{G,\sk}})^{\on{fine}}_*(\omega_{\Bun_{G,\sk}})\right) \\
@AAA @AAA \\
\on{Sp}_{\sK\to \sk}(\CF)\boxtimes \on{Sp}^{\on{co}}_{\sK\to \sk}(\BD_{\Bun_{G,\sK}}(\CF)) @>{\sim}>>  
\on{Sp}_{\sK\to \sk}(\CF)\boxtimes \BD_{\Bun_{G,\sk}}(\on{Sp}_{\sK\to \sk}(\CF)),
\endCD
\end{equation}
where:

\begin{itemize}

\item The left vertical arrow is 
\begin{multline} \label{e:map of units to verify 3} 
\on{Sp}_{\sK\to \sk}(\CF)\boxtimes \on{Sp}^{\on{co}}_{\sK\to \sk}(\BD_{\Bun_{G,\sK}}(\CF)) 
\overset{\text{\eqref{e:Sp Kunneth}}}\simeq \on{Sp}_{\sK\to \sk}(\CF\boxtimes \BD_{\Bun_{G,\sK}}(\CF)) \to \\
\to \on{Sp}_{\sK\to \sk}
\left(\sP^{\on{co}_2}_{\sK,\sk}((\Delta_{\Bun_{G,\sK}})^{\on{fine}}_*(\omega_{\Bun_{G,\sK}}))\right),
\end{multline}
where the last arrow is obtained by applying $\sP_{\sK,\sk}$ to the canonical map \eqref{e:map to unit}, which in our case is 
\begin{equation} \label{e:map of units to verify 4} 
\CF\boxtimes \BD_{\Bun_{G,\sK}}(\CF) \to (\Delta_{\Bun_{G,\sK}})^{\on{fine}}_*(\omega_{\Bun_{G,\sK}});
\end{equation} 

\item The right vertical arrow is the map \eqref{e:map to unit}.

\end{itemize}

\sssec{}

Recall now that due to the validity of \thmref{t:compactness}, the functor $\sP$ is the right adjoint of the embedding
$\on{emb.Nilp}$, see \cite[Proposition 17.2.3]{AGKRRV1}. 

\medskip

From here, we formally obtain that the functor $\sP_{\sK,\sk}$ provides a right adjoint to the embedding
$$\Shv_\Nilp(\Bun_{G,\sK,\sk})\hookrightarrow \Shv(\Bun_{G,\sK}).$$

We obtain that the commutation of diagram \eqref{e:map of units to verify 2}
is equivalent to the commutation of the following diagram:
\begin{equation} \label{e:map of units to verify 5} 
\CD
\on{Sp}^{\on{co}_2}_{\sK\to \sk}
\left((\Delta_{\Bun_{G,\sK}})^{\on{fine}}_*(\omega_{\Bun_{G,\sK}})\right) @>{\text{\propref{p:Sp omega fine}}}>> 
(\Delta_{\Bun_{G,\sk}})^{\on{fine}}_*(\omega_{\Bun_{G,\sk}}) \\
@AAA @AAA \\
\on{Sp}_{\sK\to \sk}(\CF)\boxtimes \on{Sp}^{\on{co}}_{\sK\to \sk}(\BD_{\Bun_{G,\sK}}(\CF)) @>{\sim}>>  
\on{Sp}_{\sK\to \sk}(\CF)\boxtimes \BD_{\Bun_{G,\sk}}(\on{Sp}_{\sK\to \sk}(\CF)),
\endCD
\end{equation}
where:

\begin{itemize}

\item The left vertical arrow is 
\begin{multline} \label{e:map of units to verify 6} 
\on{Sp}_{\sK\to \sk}(\CF)\boxtimes \on{Sp}^{\on{co}}_{\sK\to \sk}(\BD_{\Bun_{G,\sK}}(\CF)) 
\overset{\text{\eqref{e:Sp Kunneth}}}\simeq \on{Sp}^{\on{co}_2}_{\sK\to \sk}(\CF\boxtimes \BD_{\Bun_{G,\sK}}(\CF)) \to \\
\to \on{Sp}^{\on{co}_2}_{\sK\to \sk}
(\Delta_{\Bun_{G,\sK}})^{\on{fine}}_*(\omega_{\Bun_{G,\sK}}),
\end{multline}
where the last arrow is obtained by applying $\sP_{\sK,\sk}$ to the canonical map \eqref{e:map to dualizing fine} below;

\medskip

\item The right vertical arrow is the map \eqref{e:map to dualizing fine}.

\end{itemize}

\sssec{}

Let $\CY$ be an algebraic stack. For an object $\CF\in \Shv(\CY)^{\on{constr}}$ there exists a canonical map
$$(\Delta_\CY)_!(\ul\sfe_\CY)\to \CF\boxtimes \BD_\CY(\CF).$$

Passing to Verdier duals, we obtain a map
\begin{equation} \label{e:map to dualizing}
\CF\boxtimes \BD_\CY(\CF)\to (\Delta_\CY)_*(\omega_\CY).
\end{equation} 

\sssec{}

We take $\CY=\Bun_G$ and $\CF\in \Shv(\Bun_G)^c$. In this case, the map \eqref{e:map to dualizing}
lifts canonically to a map
\begin{equation} \label{e:map to dualizing fine}
\CF\boxtimes \BD_{\Bun_G}(\CF)\to (\Delta_{\Bun_G})^{\on{fine}}_*(\omega_{\Bun_G})
\end{equation}
in $\Shv(\Bun_G\times \Bun_G)_{\on{co}_2}$. 

\medskip

Indeed, both sides in \eqref{e:map to dualizing fine} belong to the essential image of
$$(\on{id}\times \jmath)_{\on{co},*}:\Shv(\Bun_G\times \CU)\to \Shv(\Bun_G\times \Bun_G)_{\on{co}_2}$$
for a quasi-compact $\CU\overset{\jmath}\hookrightarrow \Bun_G$, cf. proof of \propref{p:Sp omega fine}. 

\sssec{}

Now, it follows by Verdier duality from \eqref{e:Sp and Verdier} that the diagram
\begin{equation} \label{e:map of units to verify 7} 
\CD
\on{Sp}^{\on{co}_2}_{\sK\to \sk}
\left((\Delta_{\Bun_{G,\sK}})_*(\omega_{\Bun_{G,\sK}})\right) @>{\text{\propref{p:Sp omega fine}}}>> 
(\Delta_{\Bun_{G,\sk}})_*(\omega_{\Bun_{G,\sk}}) \\
@AAA @AAA \\
\on{Sp}_{\sK\to \sk}(\CF)\boxtimes \on{Sp}^{\on{co}}_{\sK\to \sk}(\BD_{\Bun_{G,\sK}}(\CF)) @>{\sim}>>  
\on{Sp}_{\sK\to \sk}(\CF)\boxtimes \BD_{\Bun_{G,\sk}}(\on{Sp}_{\sK\to \sk}(\CF))
\endCD
\end{equation}
obtained from \eqref{e:map of units to verify 5} by applying
$$\on{Ps-Id}^{\on{nv}}:\Shv(\Bun_{G,\sk}\times \Bun_{G,\sk})_{\on{co}_2}\to \Shv(\Bun_{G,\sk}\times \Bun_{G,\sk})$$
does commute. 

\medskip

This formally implies the commutation of \eqref{e:map of units to verify 5} by the same principle as in the proof
of \propref{p:Sp omega fine}.

\section{Proofs of the local acyclicity theorems} \label{s:ULA}

The goal of this section is to prove the ULA theorems stated in the previous section. 

\ssec{Proof of \thmref{t:Sat acycl}} \label{ss:Sat acycl}

\sssec{}

Let $\Delta^\lambda_{X,\sR_0},\nabla^\lambda_{X,\sR_0}\in \Shv(\on{Hecke}_{X,\sR_0})$ be the standard and costandard objects corresponding to $\lambda$,
respectively. I.e., they are, respectively, the !- and *- extensions of the (cohomologically shifted) 
constant sheaf on $\on{Hecke}^\lambda_{X,\sR_0}$.

\medskip

It is enough to show that $\Delta^\lambda_{X,\sR_0}$ and $\nabla^\lambda_{X,\sR_0}$
are ULA over $\Bun_{G,\sR_0}\underset{\Spec(\sR_0)}\times X_{\sR_0}$.
Indeed, this would imply that $\IC^\lambda_{X,\sR_0}:=\IC_{\ol{\on{Hecke}}^\lambda_{X,\sR_0}}$ is also ULA  and has the specified
restrictions over $\sk$ (this follows from the fact that the ULA condition is inherited by the passage to subquotients of perverse cohomologies,
see \cite[Corollary 1.12]{HS}). 

\medskip

We will prove the assertion for $\Delta^\lambda_{X,\sR_0}$; the assertion for $\nabla^\lambda_{X,\sR_0}$ will follow by duality. 

\sssec{}

Let $\Bun^{\on{Fl}}_{G,\sR_0}$ be the moduli stack whose $S$-points are:

\smallskip

\begin{itemize}

\item An $S$-point $x$ of $X_{\sR_0}$;

\medskip

\item A $G$-bundle $\CP$ on $X_{\sR_0}\underset{\Spec(\sR_0)}\times S$;

\item A reduction to $B$ of the restriction of $\CP$ along 
$S\overset{(x,\on{id})}\to X_{\sR_0}\underset{\Spec(\sR_0)}\times S$. 

\end{itemize}

\medskip

Consider the fiber product
$$'\!\on{Hecke}_{X,\sR_0}:=\Bun^{\on{Fl}}_{G,\sR_0}\underset{\Bun_{G,\sR_0}}\times \on{Hecke}_{X,\sR_0},$$
equipped with the maps
$$\Bun^{\on{Fl}}_{G,\sR_0} \overset{'\!\hl}\leftarrow {}'\!\on{Hecke}_{X,\sR_0}\overset{'\!\hr}\to
\Bun_{G,\sR_0}.$$

It is naturally stratified by locally closed subsets 
$$'\!\on{Hecke}_{X,\sR_0}^{\lambda'}, \quad \lambda'\in \Lambda\simeq W\backslash W^{\on{aff,ext}},$$
where $W^{\on{aff,ext}}$ is the extended affine Weyl group. Denote by $'\Delta^{\lambda'}_{X,\sR_0}$ the corresponding
standard object.

\medskip

The pullback of $\Delta^\lambda_{X,\sR_0}$ along the (smooth) projection
$$'\!\on{Hecke}_{X,\sR_0}\to \on{Hecke}_{X,\sR_0}$$
admits a filtration with subquotients $'\Delta^{\lambda'}_{X,\sR_0}$ for
$\lambda'\in \Lambda$ projecting to $\lambda\in \Lambda^+\simeq \Lambda/W$.

\medskip

Hence, order to prove that $\Delta^\lambda_{X,\sR_0}$ is ULA over $\Bun^{\on{Fl}}_{G,\sR_0}$ (with respect to
both $\hl$ and $\hr$), it suffices to show that the objects $'\Delta^{\lambda'}_{X,\sR_0}$ are ULA over
$\Bun^{\on{Fl}}_{G,\sR_0}$ with respect to $'\!\hl$ and over $\Bun_{G,\sR_0}$ with respect to $'\!\hr$.

\sssec{}

Let $\on{Hecke}^{\on{Fl}}_{X,\sR_0}$ be the moduli stack\footnote{Here ``Fl" stands for ``affine flags", as opposed to the usual
Hecke stack, which is modeled on the affine Grassmannian.} whose $S$-points are:

\smallskip

\begin{itemize}

\item An $S$-point $x$ of $X_{\sR_0}$;

\smallskip

\item A pair of $G$-bundles $\CP$ and $\CP'$ on $X_{\sR_0}\underset{\Spec(\sR_0)}\times S$;

\item Reductions to $B$ of the restrictions of $\CP$ and $\CP'$ along 
$S\overset{(x,\on{id})}\to X_{\sR_0}\underset{\Spec(\sR_0)}\times S$;

\item An isomorphism $\CP|_{X_{\sR_0}\underset{\Spec(\sR_0)}\times S-S}\simeq \CP'|_{X_{\sR_0}\underset{\Spec(\sR_0)}\times S-S}$.

\end{itemize}

\medskip

Denote by $\hl^{\on{Fl}}$ and $\hr^{\on{Fl}}$ the natural projections
$$\Bun^{\on{Fl}}_{G,\sR_0} \overset{\hl^{\on{Fl}}}\leftarrow \on{Hecke}^{\on{Fl}}_{X,\sR_0}\overset{\hr^{\on{Fl}}}\to
\Bun^{\on{Fl}}_{G,\sR_0}.$$

\medskip

The prestack $\on{Hecke}^{\on{Fl}}_{X,\sR_0}$ is an ind-algebraic stack. It is naturally stratified by locally closed
substacks $\on{Hecke}^{\on{Fl},\wt{w}}_{X,\sR_0}$, where $\wt{w}$ runs over $W^{\on{aff,ext}}$.

\medskip

For a given $\wt{w}\in \wt{W}$, let 
$$\Delta^{\wt{w}}_{X,\sR_0}\in \Shv(\on{Hecke}^{\on{Fl},\wt{w}}_{X,\sR_0})$$ denote the corresponding standard object
(i.e., the !-extension of the constant sheaf on $\on{Hecke}^{\on{Fl},\wt{w}}_{X,\sR_0}$).

\sssec{}

We have a Cartesian diagram
$$
\CD
\on{Hecke}^{\on{Fl}}_{X,\sR_0} @>>> '\!\on{Hecke}_{X,\sR_0})\\
@V{\hr^{\on{Fl}}}VV @VV{'\!\hr}V \\
\Bun^{\on{Fl}}_{G,\sR_0} @>>> \Bun_{G,\sR_0}
\endCD
$$
with the horizontal maps being smooth and proper. 

\medskip

The object $'\Delta^{\lambda'}_{X,\sR_0}\in \Shv({}'\!\on{Hecke}_{X,\sR_0}))$ is isomorphic (up to a cohomological shift)
to the direct image of $\Delta^{\wt{w}}_{X,\sR_0}\in \Shv(\on{Hecke}^{\on{Fl}}_{X,\sR_0})$ for any $\wt{w}\in W^{\on{aff,ext}}$
that projects to $\lambda'$ under $W^{\on{aff,ext}}\to W\backslash W^{\on{aff,ext}}\simeq \Lambda$. 

\medskip

Hence, it is enough to show that the objects $\Delta^{\wt{w}}_{X,\sR_0}$ are ULA over $\Bun^{\on{Fl}}_{G,\sR_0}\underset{\Spec(\sR_0)}\times X_{\sR_0}$
for both $\hl^{\on{Fl}}$ and $\hr^{\on{Fl}}$. 

\sssec{}

We argue by induction on the length $\ell(\wt{w})$. 

\medskip

When $\ell(\wt{w})=0$, the map
$$\on{Hecke}^{\on{Fl},\wt{w}}_{X,\sR_0} \overset{\hl\times s\text{ or} \hr\times s}\longrightarrow  
\Bun^{\on{Fl}}_{G,\sR_0}\underset{\Spec(\sR_0)}\times X_{\sR_0}$$
is an isomorphism, and there is nothing to prove. 

\medskip

When $\ell(\wt{w})=1$, the map 
$$\on{Hecke}^{\on{Fl},\wt{w}}_{X,\sR_0}
\overset{\hl\times s\text{ or} \hr\times s}\longrightarrow  \Bun^{\on{Fl}}_{G,\sR_0}\underset{\Spec(\sR_0)}\times X_{\sR_0}$$
is smooth fibration with fibers isomorphic to $\BA^1$. Similarly, the closure  
$$\ol{\on{Hecke}}{}^{\on{Fl},\wt{w}}_{X,\sR_0}\supset \on{Hecke}^{\on{Fl},\wt{w}}_{X,\sR_0}$$ 
is a smooth fibration 
$$\ol{\on{Hecke}}{}^{\on{Fl},\wt{w}}_{X,\sR_0}
\overset{\hl\times s\text{ or} \hr\times s}\longrightarrow  \Bun^{\on{Fl}}_{G,\sR_0}\underset{\Spec(\sR_0)}\times X_{\sR_0}$$
with fibers isomorphic to $\BP^1$. 

\medskip

Hence, the object $\Delta^{\wt{w}}_{X,\sR_0}$ is a cone of a map between constant sheaves on schemes that are smooth
over $\Bun^{\on{Fl}}_{G,\sR_0}\underset{\Spec(\sR_0)}\times X_{\sR_0}$. Hence, it is ULA. 
 
\sssec{}
 
For $\wt{w}$ of length $\geq 2$, choose a decomposition
$$\wt{w}=\wt{w}_1\cdot \wt{w}_2,\quad \ell(\wt{w})=\ell(\wt{w}_1)+\ell(\wt{w}_2).$$

Consider the convolution diagram
$$
\xy
(0,0)*+{\Bun^{\on{Fl}}_{G,\sR_0}\underset{\Spec(\sR_0)}\times X_{\sR_0}}="X";
(20,20)*+{\on{Hecke}^{\on{Fl}}_{X,\sR_0}}="Y";
(40,0)*+{\Bun^{\on{Fl}}_{G,\sR_0}\underset{\Spec(\sR_0)}\times X_{\sR_0}}="Z";
(60,20)*+{\on{Hecke}^{\on{Fl}}_{X,\sR_0}}="W";
(80,0)*+{\Bun^{\on{Fl}}_{G,\sR_0}\underset{\Spec(\sR_0)}\times X_{\sR_0}.}="U";
(40,40)*+{\on{Hecke}^{\on{Fl}}_{X,\sR_0}\underset{\hr,\Bun^{\on{Fl}}_{G,\sR_0}\underset{\Spec(\sR_0)}\times X_{\sR_0},\hl}\times \on{Hecke}^{\on{Fl}}_{X,\sR_0}}="V";
{\ar@{->}_{\hl\times s} "Y";"X"};
{\ar@{->}^{\hr\times s} "Y";"Z"};
{\ar@{->}_{\hl\times s} "W";"Z"};
{\ar@{->}^{\hr\times s} "W";"U"};
{\ar@{->}_{'\!\hl} "V";"Y"};
{\ar@{->}^{'\!\hr} "V";"W"};
\endxy
$$

By the induction hypothesis, the object
\begin{equation} 
\Delta^{\wt{w_1}}_{X,\sR_0}\wt\boxtimes \Delta^{\wt{w_2}}_{X,\sR_0}:=
({}'\!\hl)^*(\Delta^{\wt{w_1}}_{X,\sR_0})\overset{*}\otimes 
({}'\!\hr)^*(\Delta^{\wt{w_2}}_{X,\sR_0})
\in \Shv(\on{Hecke}^{\on{Fl}}_{X,\sR_0}\underset{\hr,\Bun^{\on{Fl}}_{G,\sR_0},\hl}\times \on{Hecke}^{\on{Fl}}_{X,\sR_0})
\end{equation} 
is ULA with respect to both
$$(\hl\times s)\circ {}'\!\hl \text{ and } (\hr\times s)\circ {}'\!\hr.$$

\sssec{}

Consider the convolution map
$$\on{conv}:\on{Hecke}^{\on{Fl}}_{X,\sR_0}
\underset{\hr,\Bun^{\on{Fl}}_{G,\sR_0}\underset{\Spec(\sR_0)}\times X_{\sR_0},\hl}\times \on{Hecke}^{\on{Fl}}_{X,\sR_0}\to
\on{Hecke}^{\on{Fl}}_{X,\sR_0}\underset{\Spec(\sR_0)}\times X_{\sR_0}$$
so that
$$(\hl\times s)\circ {}'\!\hl =(\hl\times s)\circ \on{conv} \text{ and } (\hr\times s)\circ {}'\!\hr =(\hr\times s)\circ \on{conv}.$$

We obtain that  
$$\Delta^{\wt{w_1}}_{X,\sR_0}\wt\boxtimes \Delta^{\wt{w_2}}_{X,\sR_0}$$
is ULA with respect to both
$$(\hl\times s)\circ \on{conv} \text{ and } (\hr\times s)\circ \on{conv}.$$ 
 
Since the map $\on{conv}$ is proper, we obtain that
$$\Delta^{\wt{w_1}}_{X,\sR_0}\star \Delta^{\wt{w_2}}_{X,\sR_0}:=
\on{conv}_!(\Delta^{\wt{w_1}}_{X,\sR_0}\wt\boxtimes \Delta^{\wt{w_2}}_{X,\sR_0})$$
is ULA with respect to both $\hl\times s$ and $\hr\times s$. 

\sssec{}

Now, since the map $\on{conv}$ induces an isomorphism
$$\on{Hecke}^{\on{Fl},\wt{w}_1}_{X,\sR_0}\underset{\hr,\Bun^{\on{Fl}}_{G,\sR_0},\hl}\times \on{Hecke}^{\on{Fl},\wt{w}_2}_{X,\sR_0}\to
\on{Hecke}^{\on{Fl},\wt{w}}_{X,\sR_0},$$
we obtain that 
$$\Delta^{\wt{w_1}}_{X,\sR_0}\star \Delta^{\wt{w_2}}_{X,\sR_0}\simeq 
\Delta^{\wt{w}}_{X,\sR_0}.$$

Hence, $\Delta^{\wt{w}}_{X,\sR_0}$ is also ULA as required. 

\qed[\thmref{t:Sat acycl}]

\ssec{The key mechanism}

We now proceed with the proofs of Theorems \ref{t:IC acycl} and \ref{t:diag acycl}. The proof will
be based on the \emph{contraction principle}, embodied by \propref{p:preserve ULA} below. 

\sssec{}

We place ourselves again in the situation of
\cite{DG2}, over an arbitrary Noetherian base $S$ (i.e., this is a generalization of the context of \secref{sss:contr},
where instead of $\Spec(\sR_0)$ we have a more general $S$).  

\medskip 

We claim:

\begin{prop} \label{p:preserve ULA}
Let $\CF\in \Shv(\CU)$ be ULA over $S$. Then so
is $j_!(\CF)\in \Shv(\CY)$.
\end{prop}

\begin{proof}

Repeats the proof of \propref{p:Sp contr}.

\end{proof} 

\begin{rem}

The assertion of \propref{p:preserve ULA} replicates that of \cite[Theorem 6.1.3]{HHS}\footnote{We are grateful to L.~Hamann
for pointing this out to us.}. We refer the reader to {\it loc. cit.} where the proof is written out in detail.

\end{rem}

\begin{rem}

As a side remark, we observe that  \propref{p:preserve ULA} allows us to give an alternative proof of \thmref{t:Sat acycl}:

\medskip

We can show that the objects $\Delta^\lambda_{X,\sR_0}$ are ULA by reducing 
to a contractive situation as in the original Kazhdan-Lusztig paper \cite{KL},
by intersecting with the opposite Schubert strata.  

\end{rem} 

\ssec{Proof of \thmref{t:IC acycl}} \label{ss:IC acycl}

The proof below is inspired by the computation of IC stalks in \cite{BFGM}, and follows the same
line of thought as that in \cite[Theorem 6.2.1]{HHS}.

\sssec{}

Let $M$ denote the Levi quotient of $P^-$. Let $J\subset I$ be the subset of the Dynkin diagram of $G$ 
that corresponds to the roots inside $M$.

\medskip

Denote by $\Lambda^{\on{pos}}_{G,P}$ the monoid equal to 
the non-negative integral span of $\alpha_i$, $i\in (I-J)$.  
We endow it with the standard order relation.

\medskip

Recall (see \cite[Sect. 1.3.3]{BG1}) that $\BunPtm$ carries a stratification indexed by elements of $\Lambda^{\on{pos}}_{G,P}$,
$$\BunPtm=\underset{\theta\in \Lambda^{\on{pos}}_{G,P}}\cup\, (\BunPtm)_\theta,$$
with the open stratum $(\BunPtm)_0$ being $\Bun_{P^-}$. 

\medskip

For a given $\theta$, let 
$$(\BunPtm)_{< \theta}\subset \BunPtm \text{ and } 
%\overset{j_{<\theta,\leq\theta}}\hookrightarrow 
(\BunPtm)_{\leq \theta}\subset \BunPtm$$
be the open substacks
$$\underset{\theta'< \theta}\cup\, (\BunPtm)_{\theta'} \text{ and } \underset{\theta'\leq \theta}\cup\, (\BunPtm)_{\theta'},$$
respectively.

\medskip

We will prove \thmref{t:IC acycl} by induction on $\theta$. The base of the induction is when $\theta=0$, in which case 
$(\BunPtm)_{\leq \theta}=\Bun_{P^-}$, which is smooth over $\Bun_{M,\sR_0}$, and the ULA property\footnote{For the duration
of this proof, ``ULA" means ``ULA over $\Bun_{M,\sR_0}$".} is obvious. 

\medskip

Thus, we will assume that the ULA statement holds for 
$$j_!(\ul\sfe_{\Bun_{P^-,\sR_0}})|_{(\wt\Bun_{P^-,\sR_0})_{<\theta}}$$
and we will deduce its validity for 
$$j_!(\ul\sfe_{\Bun_{P^-,\sR_0}})|_{(\wt\Bun_{P^-,\sR_0})_{\leq \theta}}.$$

%\medskip

%Thus, we have to show
%$$j_!(\ul\sfe_{\Bun_{P^-,\sR_0}})|_{(\wt\Bun_{P^-,\sR_0})_{<\theta}} \text{ is ULA } \Rightarrow
%j_{<\theta,\leq\theta}\left(j_!(\ul\sfe_{\Bun_{P^-,\sR_0}})|_{(\wt\Bun_{P^-,\sR_0})_{<\theta}}\right)  \text{ is ULA }.$$ 

\sssec{}

Consider the parabolic Zastava space $\on{Zast}$, see \secref{sss:Zast} below. Pulling back the above stratification along
the projection
$$\on{Zast}\to \BunPtm,$$
we obtain a stratification on $\on{Zast}$. We denote by
$$\on{Zast}_\theta,\,\, \on{Zast}_{\leq \theta},\,\, \on{Zast}_{<\theta}$$
the corresponding substacks.

\medskip

Denote
$$\overset{\circ}{\on{Zast}}:=\on{Zast}_0=\on{Zast}\underset{\BunPtm}\times \Bun_{P^-}.$$

Denote by $j_{\on{Zast}}$ the open embedding $\overset{\circ}{\on{Zast}}\hookrightarrow \on{Zast}$.

\medskip

Recall also that $\on{Zast}$ splits as a disjoint union
$$\on{Zast}:=\underset{\theta\in \Lambda^{\on{pos}}_{G,P}}\sqcup\, \on{Zast}^\theta.$$

%and each $\on{Zast}^\theta$ is equipped with a map
%$$\fr:\on{Zast}^\theta\to X^{(\theta)},$$
%where 
%$$X^{(\theta)}:=\underset{i}\Pi\, X^{(n_i)} \text{ for } \theta=\underset{i\in I-J}\Sigma\, n_i\cdot \alpha_i.$$

%Furthermore, $\on{Zast}$ possesses the following factorization property: for $\theta=\theta_1+\theta_2$
%we have
%$$\on{Zast}^\theta\underset{X^{(\theta)}}\times (X^{(\theta_1)}\times X^{(\theta_2)})_{\on{disj}}\simeq
%(\on{Zast}^{\theta_1}\times \on{Zast}^{\theta_2})\underset{X^{(\theta_1)}\times X^{(\theta_2)}}\times 
%(X^{(\theta_1)}\times X^{(\theta_2)})_{\on{disj}}.$$

%\sssec{}

\sssec{}

Arguing as in \cite[Sect. 5]{BFGM} or \cite[Sect. 3.9]{Ga4}, we obtain:

\begin{lem}  \label{l:interplay} \hfill

\smallskip

\noindent{\em(a)}
The ULA property of $j_!(\ul\sfe_{\Bun_{P^-,\sR_0}})|_{(\wt\Bun_{P^-,\sR_0})_{<\theta}}$
implies the ULA property of
$$(j_{\on{Zast}})_!(\ul\sfe_{\overset{\circ}{\on{Zast}}_{\sR_0}})|_{(\on{Zast}_{\sR_0})_{<\theta}}.$$

\medskip

\noindent{\em(b)} The following statements are equivalent:

\medskip

\noindent{\em(i)} $j_!(\ul\sfe_{\Bun_{P^-,\sR_0}})|_{(\wt\Bun_{P^-,\sR_0})_{\leq \theta}}$ is ULA;

\smallskip

\noindent{\em(ii)} $(j_{\on{Zast}})_!(\ul\sfe_{\overset{\circ}{\on{Zast}}_{\sR_0}})|_{(\on{Zast}_{\sR_0})_{\leq \theta}}$ is ULA;

\smallskip

\noindent{\em(iii)} $(j_{\on{Zast}})_!(\ul\sfe_{\overset{\circ}{\on{Zast}}_{\sR_0}})|_{(\on{Zast}^{\theta}_{\sR_0})_{\leq \theta}}$ is ULA.

\end{lem}

\sssec{}

Let us denote by $j_{<\theta,\leq\theta}$ the open embedding
$$\on{Zast}^\theta_{<\theta}\hookrightarrow \on{Zast}^\theta_{\leq \theta}.$$

Note that the complement of this embedding is the closed stratum $\on{Zast}^\theta_\theta\subset \on{Zast}^\theta_{\leq \theta}$. 

\medskip

Thus, we have to show
\begin{equation} \label{e:Zast implication}
(j_{\on{Zast}})_!(\ul\sfe_{\overset{\circ}{\on{Zast}}_{\sR_0}})|_{(\on{Zast}^\theta_{\sR_0})_{<\theta}}\text{ is ULA } \Rightarrow 
(j_{<\theta,\leq\theta})_!\left((j_{\on{Zast}})_!(\ul\sfe_{\overset{\circ}{\on{Zast}}_{\sR_0}})|_{(\on{Zast}^\theta_{\sR_0})_{<\theta}}\right) 
\text{ is ULA }.
\end{equation} 

\sssec{}

Recall now (see \cite[Sect. 5.1]{BFGM}) that $\on{Zast}^\theta=\on{Zast}^\theta_{\leq \theta}$ carries an action of $\BG_m$
that contracts it onto the locus $\on{Zast}^\theta_{\theta}$. 

\medskip

Since the object $(j_{\on{Zast}})_!(\ul\sfe_{\overset{\circ}{\on{Zast}}_{\sR_0}})|_{(\on{Zast}^\theta_{\sR_0})_{<\theta}}$
is $\BG_m$-equivariant, the implication \eqref{e:Zast implication} follows from \propref{p:preserve ULA}.

\qed[\thmref{t:IC acycl}]

%$$\on{Zast}^\theta:=\underset{\theta=\theta^2-\theta^1}\sqcup\, 
%(\Bun^{\theta_1}_P\underset{\Bun_G}\times \BunPtm^{\theta_2})^{\on{tr}},$$
%where:
%
%\begin{itemize}
%
%\item $\theta^i\in \pi_{1,\on{alg}}(M)\simeq \pi_0(\Bun_M)\simeq \pi_0(\Bun_P)\simeq \pi_0(\BunPtm)$;
%
%\item We regard $\Lambda_{G,P}$ as a subset of $\pi_{1,\on{alg}}(M)$.
%
%\end{itemize}

\ssec{Proof of \thmref{t:diag acycl}} \label{ss:diag acycl}

It is easy to reduce the assertion to the case when $G$ is semi-simple, which we will now assume.\footnote{Otherwise
replace $R_{Z_G}$ below by the $!$-direct image of $\sfe$ along
the map $\on{pt} \to \on{pt}/Z_G$.}

\sssec{}

Recall that according to \cite{Sch}, the diagonal map
$$\Delta:\Bun_G\to \Bun_G\times \Bun_G$$
can be factored as
$$\Bun_G=\Bun_G\times \on{pt}\to \Bun_G\times \on{pt}/Z_G\overset{j}\hookrightarrow \ol{\Bun}_G\overset{\ol\Delta}\to \Bun_G\times \Bun_G,$$
where the map $j$ is an open embedding and $\ol\Delta$ is proper. 

\sssec{Example} 

For $G=SL_2$, the stack $\ol{\Bun}_G$ classifies triples 
$$(\CE_1,\CE_2,\alpha)/\BG_m,$$
where $\CE_1$ and $\CE_2$ rank-2 bundles with trivialized determinants, 
and $\alpha$ is a \emph{non-zero} map $\CE_1\to \CE_2$.
The action of $\BG_m$ is given by scaling $\alpha$. 

\medskip

The open locus $\Bun_G\times \on{pt}/Z_G$ corresponds to the condition that $\alpha$ be an isomorphism. 

\sssec{}

Thus, it suffices to show that the object
\begin{equation} \label{e:on Vinb}
j_!(\ul\sfe_{\Bun_{G,\sR_0}}\boxtimes R_{Z_G})\in \Shv(\ol{\Bun}_{G,\sR_0})
\end{equation} 
is ULA over $\Spec(\sR_0)$, where 
$$R_{Z_G}\in \Rep(Z_G)\to \Shv(\Spec(\sR_0)/Z_G)$$
denotes the regular representation. 

\sssec{}

Recall (see \cite[Sect. 2.2.7]{Sch}) that stack $\ol{\Bun}_G$ admits a stratification indexed by the poset of standard parabolics in $G$:
$$\ol{\Bun}_G=\underset{P}\cup\, \ol{\Bun}_{G,P}$$
with the open stratum $\ol{\Bun}_{G,G}$ being the image of the embedding $j$.

\medskip

Denote by $\ol{\Bun}_{G,\geq P}$ (resp., $\ol{\Bun}_{G,>P}$) the open substack equal to the union of the strata
corresponding to the parabolics $P'$ with $P\subset P'$ (resp., $P\subsetneq P'$). 

\medskip

We argue by induction and assume that the object \eqref{e:on Vinb} is ULA when restricted to the open substack
$\ol{\Bun}_{G,>P,\sR_0}$.  We now perform the induction step and prove that the ULA property holds over 
$\ol{\Bun}_{G,\geq P,\sR_0}$.

\sssec{}

Consider the stratum $\ol{\Bun}_{G,P}$. According to \cite[Sect. 3]{Sch} that $\ol{\Bun}_{G,P}$ admits a further stratification
indexed by elements of $\Lambda^{\on{pos}}_{G,P}$:
$$\ol{\Bun}_{G,P}=\underset{\theta\in \Lambda^{\on{pos}}_{G,P}}\cup\, (\ol{\Bun}_{G,P})_\theta.$$

For a given $\theta$, let
$$(\ol{\Bun}_{G,P})_{<\theta} \subset (\ol{\Bun}_{G,P})_{\leq \theta}$$
be the corresponding open subsets.

\medskip

In particular, we have the open locus
$$(\ol{\Bun}_{G,P})_0\subset  \ol{\Bun}_{G,P}.$$

%\medskip

%
%Note that $(\ol{\Bun}_{G,P})_0$ is contained in the \emph{defect-free} locus 
%$$_0(\ol{\Bun}_G):={}_0\on{VinBun}_G/T,$$
%see \cite[Sect. 2.4]{Sch}.
%
%We will prove:
%

\sssec{Example} In the example of $G=SL_2$, the Borel stratum corresponds to the condition that the map $\alpha$ 
has generic rank $1$ and hence factors as
$$\CE_1\overset{\beta_1}\twoheadrightarrow \CL_1\overset{\gamma}\to \CL_2\overset{\beta_2}\hookrightarrow \CE_2,$$
where $\CL_i$ are line bundles and $\beta_i$ are bundle maps. 

\medskip

The stratification by $\Lambda^{\on{pos}}_{G,P}=\BZ^{\geq 0}$ is given by the total degree of zeroes of the map $\gamma$. 

\sssec{}

Denote by $\ol{\Bun}_{G,\geq P,<\theta}$ and $\ol{\Bun}_{G,\geq P,\leq\theta}$ the open substacks
$$\ol{\Bun}_{G,> P} \cup (\ol{\Bun}_{G,P})_{<\theta} \text{ and } \ol{\Bun}_{G,>P} \cup (\ol{\Bun}_{G,P})_{\leq \theta},$$
respectively. 
%Denote by $j_{\geq P,<\theta,\leq \theta}$ the open embedding
%$$\ol{\Bun}_{G,\geq P,<\theta}\hookrightarrow \ol{\Bun}_{G,\geq P,\leq \theta}.$$

\medskip

By induction, we can assume that \eqref{e:on Vinb} is ULA\footnote{For the duration of this proof, ``ULA" means ``ULA over $\Spec(\sR_0)$.} 
when restricted to $\ol{\Bun}_{G,\geq P,<\theta,\sR_0}$. We will
now perform the induction step and prove that the ULA property holds over $\ol{\Bun}_{G,\geq P,\leq \theta,\sR_0}$.

\begin{rem}

The rest of the argument, which is explained below, uses the same principle as the proof of \thmref{t:IC acycl}: we will 
replace $\ol{\Bun}_{G,\geq P,\sR_0}$ by its local model and reduce the assertion to a contractive situation 
(i.e., one covered by \propref{p:preserve ULA}).

\end{rem}

\ssec{Proof of \thmref{t:diag acycl}, continuation}

\sssec{}

Let $Y^P$ denote the open substack
$$\left(\Bun_{P}\underset{\Bun_G}\times \ol{\Bun}_G  \underset{\Bun_G}\times \Bun_{P^-}\right)^{\on{tr}}\subset 
\Bun_{P}\underset{\Bun_G}\times \ol{\Bun}_G  \underset{\Bun_G}\times \Bun_{P^-},$$
where the superscript ``tr" refers to the generic transversality condition, see \cite[Sect. 6.1.6]{Sch}. 

\sssec{Example} For $G=SL_2$, the fiber product  
$$\Bun_{P}\underset{\Bun_G}\times \ol{\Bun}_G  \underset{\Bun_G}\times \Bun_{P^-}$$
classifies the data of
$$\CL'_1\overset{\delta_1}\to \CE_1 \overset{\alpha}\to \CE_2\overset{\delta_2}\to \CL'_2,$$
where $\delta_i$ are bundle maps.

\medskip

The generic transversality condition is that the composite map $\CL'_1\to\CL'_2$ be non-zero.

\sssec{}

The stratifications on $\ol{\Bun}_G$ induce the corresponding stratifications on $Y^P$, to be denoted
$Y^P_{P'}$, $Y^P_{P',\theta}$, etc. We note that $Y^P_{P'}=\emptyset$ unless $P\subset P'$. 

\medskip

Note that the open stratum $Y^P_{G}$ identifies with
$$\overset{\circ}{\on{Zast}}\times \on{pt}/Z_G.$$ 

Denote by $j_Y$ the open embedding
$$\overset{\circ}{\on{Zast}}\times \on{pt}/Z_G\hookrightarrow Y^P.$$

\medskip

In addition, $Y^P$ splits as a disjoint union 
\begin{equation} \label{e:split Y}
Y^P:=\underset{\theta\in \Lambda^{\on{pos}}_{G,P}}\sqcup\, Y^{P,\theta},
\end{equation}
and $Y^{P,\theta}_{P,\theta'}$ is empty unless $\theta'\leq \theta$.

\sssec{Example} In the example of $G=SL_2$, the decomposition \eqref{e:split Y} is
according to the total degree of zeroes of the map $\CL'_1\to\CL'_2$. 

\sssec{}

The following is parallel to \lemref{l:interplay}:

\begin{lem} \hfill

\smallskip

\noindent{\em(a)} The ULA property of 
$$j_!(\ul\sfe_{\Bun_{G,\sR_0}}\boxtimes R_{Z_G})|_{\ol{\Bun}_{G,\geq P,<\theta,\sR_0}}$$
implies the ULA property of 
$$(j_Y)_!(\ul\sfe_{\overset{\circ}{\on{Zast}_{\sR_0}}}\boxtimes R_{Z_G})|_{Y^P_{\geq P,<\theta,\sR_0}}.$$

\smallskip

\noindent{\em(b)} The following conditions are equivalent:

\smallskip

\noindent{\em(i)} $j_!(\ul\sfe_{\Bun_{G,\sR_0}}\boxtimes R_{Z_G})|_{\ol{\Bun}_{G,\geq P,\leq \theta,\sR_0}}$
is ULA;

\smallskip

\noindent{\em(ii)} $(j_Y)_!(\ul\sfe_{\overset{\circ}{\on{Zast}_{\sR_0}}}\boxtimes R_{Z_G})|_{Y^P_{\geq P,\leq \theta,\sR_0}}$
is ULA;

\smallskip

\noindent{\em(iii)} $(j_Y)_!(\ul\sfe_{\overset{\circ}{\on{Zast}_{\sR_0}}}\boxtimes R_{Z_G})|_{Y^{P,\theta}_{\geq P,\leq \theta,\sR_0}}$
is ULA.

\end{lem} 

\sssec{}

Let $j_{\geq P,<\theta,\leq \theta}$ the open embedding
$$Y^{P,\theta}_{\geq P,<\theta}\hookrightarrow Y^{P,\theta}_{\geq P,\leq \theta}.$$

Note that the complement of this embedding is the closed stratum
$$Y^{P,\theta}_{P,\theta}\subset Y^{P,\theta}_{\geq P,\leq \theta}.$$

\sssec{Example}

In the example of $G=SL_2$, the closed substack $Y^{P,\theta}_{P,\theta}$ is the locus
where the maps
$$\CL'_1\to \CL_1 \text{ and } \CL_2\to \CL'_2$$
are isomorphisms. 

\medskip

So $\CE_i=\CL_i\oplus \CL_i^{\otimes -1}$, and the map $\alpha$ is
$$\CE_1\simeq \CL_1\oplus \CL_1^{\otimes -1}\twoheadrightarrow \CL_1\to \CL_2\hookrightarrow
\CL_2\oplus \CL_2^{\otimes -1}.$$

Thus, we obtain that $Y^{P,\theta}_{P,\theta}$ is isomorphic to a version of the Hecke stack for $\BG_m$:
it classifies pairs $(\CL_1,D)$, where $\CL_1$ is a line bundle and $D$ is an effective divisor on $X$
of degree $\theta$. 

\sssec{}

Thus, we have to show:

$$(j_Y)_!(\ul\sfe_{\overset{\circ}{\on{Zast}_{\sR_0}}}\boxtimes R_{Z_G})|_{Y^{P,\theta}_{\geq P,<\theta,\sR_0}} \text{ is ULA }
\Rightarrow 
(j_{\geq P,<\theta,\leq \theta})_!\left((j_Y)_!(\ul\sfe_{\overset{\circ}{\on{Zast}_{\sR_0}}}\boxtimes R_{Z_G})|_{Y^{P,\theta}_{\geq P,<\theta,\sR_0}}\right)
\text{ is ULA.}$$

\sssec{}

Recall now (see \cite[Sect. 6.5.5]{Sch}) that $Y^{P,\theta}_{\geq P,\leq \theta}$ carries an action of $\BG_m$, which contracts its 
into $Y^{P,\theta}_{P,\theta}$. 

\sssec{Example} In the example of $G=SL_2$, the above action is the following one. A scalar $c\in \BG_m$ acts on 
the triple 
$$\CL'_1\overset{\delta_1}\to \CE_1 \overset{\alpha}\to \CE_2\overset{\delta_2}\to \CL'_2$$
by 
$$\delta_1\mapsto c^{-1}\cdot \delta_1,\,\, \delta_2\mapsto c^{-1}\cdot \delta_2,\,\, \alpha\mapsto c^2\cdot \alpha.$$

\sssec{}

Since the object 
$$(j_Y)_!(\ul\sfe_{\overset{\circ}{\on{Zast}_{\sR_0}}}\boxtimes R_{Z_G})|_{Y^P_{\geq P,<\theta,\sR_0}} \in
\Shv(Y^P_{\geq P,<\theta,\sR_0})$$
is \emph{quasi-equivariant} with respect to this $\BG_m$-action (i.e., equivariant after passing to a finite self-isogeny
of $\BG_m$), the required assertion follows from \propref{p:preserve ULA}. 

\qed[\thmref{t:diag acycl}]

\section{Proof of \thmref{t:Poinc acycl}} \label{s:proof Poinc}

We will give two proofs. The first proof is shorter, but it requires more stringent assumptions on the characteristic of the ground field 
$\sk$ (see Assumption (*) in \secref{sss:Ass *} and Remark \ref{r:IMP}). 

\ssec{First proof of \thmref{t:Poinc acycl}} \label{ss:Poinc acycl}

\sssec{} \label{sss:recall Poinc}

Let 
$$\on{exp}^{\boxtimes I}/T\in \Shv(\BG_a^I/T)$$
be as in \cite[Sect. 3.3]{GLC1}, where $I$ is the Dynkin diagram of $G$ (see also
\secref{sss:Kir on AI} below). 

\medskip

Recall that we have a canonically defined $T$-equivariant map 
$$\chi^I:\Bun_{N,\rho(\omega_X)}\to \BG_a^I,$$ 
and consider the corresponding map
$$\chi^I/T:\Bun_{N,\rho(\omega_X)}/T\to \BG_a^I/T.$$

\medskip

Denote
$$\on{exp}_{\chi^I}/T:=(\chi^I/T)^*(\on{exp}^{\boxtimes I}/T)\in \Shv(\Bun_{N,\rho(\omega_X)}/T).$$

\medskip

Recall the map 
$$\sfp:\Bun_{N,\rho(\omega_X)}\to \Bun_G$$ and note that it naturally factors via a map
$$\sfp/T:\Bun_{N,\rho(\omega_X)}/T\to \Bun_G.$$

Recall that the object $\on{Poinc}^{\on{Vac}}_!$ is defined as
$$(\sfp/T)_!(\on{exp}_{\chi^I}/T)\in \Shv(\Bun_G).$$

\sssec{}

We consider the above objects over $\Spec(\sR_0)$. The object $\on{exp}^{\boxtimes I}_{\sR_0}/T$ is ULA over $\Spec(\sR_0)$
(e.g., by \propref{p:preserve ULA}). Since the map $\chi^I/T$ is smooth, we obtain that
$$\on{exp}_{\chi^I,\sR_0}/T\in \Shv(\Bun_{N,\rho(\omega_X),\sR_0}/T)$$
is also ULA. 

\medskip

For the rest of the first proof, we will consider separately the cases of $g\geq 2$, $g=1$ and $g=0$.

\sssec{} \label{sss:Ass *}

We will make the following assumption on the pair $(G,\on{char}(p))$:

\medskip

\noindent(*) {\it For every standard Levi $M$ and a central character $\check\mu$ of $Z_M$, the 
direct summand $(\fn_P)_{\check\mu}$, viewed as a representation of $M$, has \emph{strangeness} $0$
(see \cite[Definition 10.3.4]{DG1} for what this means).} 

\begin{rem}
As was explained in {\it loc. cit.}, the above assumption is automatic when one works over a ground
field of characteristic $0$.
\end{rem}

\begin{rem} \label{r:IMP}

According to \cite[Theorem 3.1]{IMP}, the following condition guarantees that Assumption (*) is satisfied:

\medskip

We need that for all $M$ and all roots $\check\alpha$ appearing in $\fn_P$,
$$2\cdot \langle \rho_M,\check\alpha\rangle \cdot |Z^0_M\cap [M,M]|<p.$$

In the above formula, the factor $|Z^0_M\cap [M,M]|$ appears since
in \cite{IMP} only representations with trivial determinant are allowed.

\end{rem}

\sssec{}

We proceed with the first proof of \thmref{t:Poinc acycl}. We start with the case $g\geq 2$.

\medskip

Note that the map $\sfp/T$ factors as
$$\Bun_{N,\rho(\omega_X)}/T\to \Bun_G^{(\leq (2g-2)\cdot \rho)} \overset{\jmath_{(2g-2)\cdot \rho}}\hookrightarrow \Bun_G,$$
where the first arrow is proper (see \cite[Theorem 7.4.3(a)]{DG1}).

\medskip

Hence, it suffices to show that the functor 
$$(\jmath_{(2g-2)\cdot \rho})_!:\Shv(\Bun_{G,\sR_0}^{(\leq (2g-2)\cdot \rho)}) \to \Shv(\Bun_{G,\sR_0})$$
preserves the property of being ULA over $\Spec(\sR_0)$. 

\medskip

We now use Assumption (*). It implies that the complement of 
the embedding
$$\Bun_G^{(\leq (2g-2)\cdot \rho)}\overset{\jmath_{(2g-2)\cdot \rho}}\hookrightarrow \Bun_G$$
is contractive, see \cite[Proposition 10.1.3]{DG1}. 

\medskip

Hence, the required preservation of the ULA property follows from \propref{p:preserve ULA}. 

\sssec{}

We now consider the case $g=1$. In this case, the map $\sfp/T$ factors as
\begin{equation} \label{e:factor map from BunNb}
\Bun_{N,\rho(\omega_X)}/T
\to \Bun_G^{(\leq 0)} \overset{\jmath_0}\hookrightarrow \Bun_G.
\end{equation}

\medskip

Note that $\Bun_G^{(\leq 0)}$ is the semi-stable locus $\Bun^{ss}_G\subset \Bun_G$. We factor the first arrow in 
\eqref{e:factor map from BunNb} as
$$\Bun_{N,\rho(\omega_X)}/T\overset{\omega\,\text{is trivial}}\simeq \Bun_N/T\hookrightarrow
\Bun^0_B\to \Bun^{ss}_G,$$
where the second arrow is a closed embedding, and $\Bun_B^0$ is the preimage of the neutral connected
component of $\Bun_T$ under $\sfq:\Bun_B\to \Bun_T$. 

\medskip

Note that the map 
$$\Bun^0_B\to \Bun^{ss}_G$$
is proper (indeed, $\Bun_B^0=\Bun_B^0\underset{\Bun_G}\times  \Bun^{ss}_G\to \BunBb^0\underset{\Bun_G}\times  \Bun^{ss}_G$ is an equality). 

\medskip

Hence, the first arrow in \eqref{e:factor map from BunNb} is proper. Therefore, it remains to show
that the functor 
$$(\jmath_0)_!:\Shv(\Bun^{(\leq 0)}_{G,\sR_0})\to \Shv(\Bun_{G,\sR_0})$$
preserves the property of being ULA over $\Spec(\sR_0)$.

\medskip

However, in genus 1, the complement of $\Bun_G^{(\leq 0)}$ in $\Bun_G$ is contractive
(by Assumption (*) and \cite[Proposition 10.1.3]{DG1}). Hence, the assertion follows from \propref{p:preserve ULA}. 

\ssec{Proof of \propref{t:Poinc acycl} in genus \texorpdfstring{$0$}{00}}

It remains to consider the case of $g=0$. Here the argument will be of different nature
and in fact can be considered as a simplified version of the second proof relying on special features of the genus 0 situation. 

\sssec{}

We will show directly that 
$$\Phi(\on{Poinc}^{\on{Vac}}_{!,\sR_0})=0$$
(which is what we need for Property (C)). 

\medskip

However, by Remark \ref{r:van cycles}, this is equivalent to the fact that $\on{Poinc}^{\on{Vac}}_{!,\sR_0}$is
ULA over  $\Spec(\sR_0)$. 

%
%the map
%\begin{multline} \label{e:Sp Poinc 1}
%\on{Poinc}^{\on{Vac}}_{!,\sk}\simeq (\sfp/T)_!(\on{exp}_{\chi^I}/T,\sk)=
%(\sfp/T)_!\circ \on{Sp}_{\sK\to \sk}(\on{exp}_{\chi^I}/T,\sK)\to \\
%\to \on{Sp}_{\sK\to \sk}\circ (\sfp/T)_!(\on{exp}_{\chi^I}/T,\sK)=\on{Sp}_{\sK\to \sk}(\on{Poinc}^{\on{Vac}}_{!,\sk})
%\end{multline} 
%is an isomorphism (which is what we need for Property C). 
%
%\medskip
%
%However, it is easy to see that the fact that \eqref{e:Sp Poinc 1} is equivalent to the fact that 
%$$\on{Poinc}^{\on{Vac}}_{!,\sR_0}=(\sfp/T)_!(\on{exp}_{\chi^I}/T,\sR_0)$$
%is ULA over $\Spec(\sR_0)$, see Remark \ref{r:van cycles}. 

\sssec{}

Note that the direct sum of the constant terms functors 
$$\on{CT}^\lambda_*:\Shv(\Bun_G)\to \Shv(\Bun^\lambda_T), \quad \lambda\in \Lambda^+$$
is conservative when $g=0$.

\medskip

Hence, it suffices to show that
$$\on{CT}^\lambda_*\circ \Phi(\on{Poinc}^{\on{Vac}}_{!,\sR_0})=0, \quad \lambda\in \Lambda^+.$$

\sssec{}

Recall also that according to \cite{DG2}, we have a canonical isomorphism
\begin{equation} \label{e:weird adj}
\on{CT}^\lambda_*\simeq \on{CT}^{-,\lambda}_!.
\end{equation} 
%(up to a cohomological shift, which has to do with the normalization of these functors, see \cite[Equation (8.3)]{GLC3}).

\medskip
 
We claim now that the canonical maps
$$\on{CT}^{-,\lambda}_!\circ \Phi \to \Phi \circ \on{CT}^{-,\lambda}_!$$
are isomorphisms when $\lambda\in \Lambda^+$.

\medskip

Indeed, the functor $\on{CT}^{-,\lambda}_!$ is given by
$$(\sfq^-)_!\circ (\sfp^-)^*$$
for the diagram
$$\Bun_G \overset{\sfp^-}\leftarrow \Bun^\lambda_{B^-}\overset{\sfq^-}\to \Bun_T^\lambda.$$

Now, for $\lambda\in \Lambda^+$, the map $\sfp^-$ is smooth, and hence the functor $(\sfp^-)^*$
commutes with vanishing cycles. The functor $(\sfq^-)_!$ commutes with vanishing cycles thanks to the
contraction principle (see \cite[Sect. 4.1]{DG2}) and isomorphism \eqref{e:Sp contr}. 

\sssec{}

Thus, it suffices to show that 
$$\Phi\circ \on{CT}^{-,\lambda}_!(\on{Poinc}^{\on{Vac}}_{!,\sR_0})=0.$$

\medskip

However, this is the assertion of \eqref{e:CT Vac Sp} below. 

% the (baby case of the) calculation in \thmref{t:Whit of Eis} shows that 
%$$\on{CT}^-_!\circ \on{Poinc}^{\on{Vac}}_!\simeq \on{transl}_{\rho(\omega_X)}\left(\on{Fact}(\Omega^{\on{loc}})\star \delta_{1,\Bun_T}\right),$$
%up to an appropriate cohomological shift.
%
%\medskip
%
%Unwinding the identifications, we obtain that the map \eqref{e:Sp Poinc 2} identifies with the map
%$$\on{transl}_{\rho(\omega_{X_\sk})}\left(\on{Fact}(\Omega^{\on{loc}})\star \delta_{1,\Bun_{T,\sk}}\right)\to
%\on{Sp}_{\sK\to \sk}\left(\on{transl}_{\rho(\omega_{X_\sK})}\left(\on{Fact}(\Omega^{\on{loc}})\star \delta_{1,\Bun_{T,\sK}}\right)\right),$$
%which is easily seen to be an isomorphism. 

\qed[\thmref{t:Poinc acycl}]

\ssec{Second proof of \thmref{t:Poinc acycl}} \label{ss:2nd proof}

It is easy to reduce the assertion to the case 
when $G$ is semi-simple, so we will make this assumption. 

\sssec{}

Recall (see Remark \ref{r:van cycles}) that we have to show that
\begin{equation} \label{e:van van cycles}
\Phi(\on{Poinc}^{\on{Vac}}_{!,\sR_0})=0.
\end{equation} 

\sssec{}

Recall that we have the semi-orthogonal decomposition of the category $\Shv(\Bun_G)$
$$\Shv(\Bun_G)_{\Eis}\overset{\be_{\on{Eis}}}\hookrightarrow  \Shv(\Bun_G)\overset{\be_{\on{cusp}}}\hookleftarrow \Shv(\Bun_G)_{\on{cusp}}.$$

Thus, every object $\CF\in \Shv(\Bun_G)$ fits into a (canonically defined) fiber sequence
$$\be_{\on{Eis}}\circ \be_{\on{Eis}}^R(\CF)\to \CF\to \be_{\on{cusp}}\circ  \be^L_{\on{cusp}}(\CF).$$

Furthermore, 
$$\be_{\on{Eis}}^R(\CF)=0 \text{ if and only if } \on{CT}_*(\CF)=0 \text{ for all proper parabolics } P\subset G.$$

\sssec{}

We will show that
\begin{equation} \label{e:cuspidal part}
\be^L_{\on{cusp}}(\Phi(\on{Poinc}^{\on{Vac}}_{!,\sR_0}))=0 
\end{equation} 
and 
\begin{equation} \label{e:Eis part}
\on{CT}_*(\Phi(\on{Poinc}^{\on{Vac}}_{!,\sR_0}))=0.
\end{equation}

\medskip

By the above, this will imply \eqref{e:van van cycles}.

\ssec{The Eisenstein part}

We first tackle \eqref{e:Eis part}. 

\sssec{}

First, we claim:

\begin{prop} \label{p:nearby and CT}
The natural transformation
$$\Psi\circ \on{CT}_*\to \on{CT}_*\circ \Psi$$
is an isomorphism.
\end{prop}

\begin{proof}

Using Beilinson's definition of nearby cycles, the functor $\Psi$ is the colimit
of functors
$$\bi_0^*\circ (\bj_0)_*((-)\otimes \CE),$$
where:
\begin{itemize}

\item The maps $\bi_0$ and $\bj_0$ are $\Spec(\sk)\to \Spec(\sR_0)$ and $\Spec(\sK_0)\to \Spec(\sR_0)$
and base changes thereof;

\item $\CE$ is a local system on $\Spec(\sK_0)$ or a pullback thereof.

\end{itemize}

We claim that the operation $\on{CT}_*$ commutes with all the three functors involved:

\medskip
 
\noindent(i) The commutation with $(-)\otimes \CE$ is obvious;

\medskip

\noindent(ii) The commutation with the functor $(\bj_0)_*$ is also obvious, since $\on{CT}_*$
involves !-pullbacks and *-pushforwards. 

\medskip

\noindent(iii) In order to establish the commutation with $\bi_0^*$, we apply the
isomorphism \eqref{e:weird adj}, and rewrite $\on{CT}_*$ as $\on{CT}^-_!$. 
Now again the commutation becomes obvious, since $\on{CT}^-_!$ involves *-pullback
and !-pushforward.

\end{proof}

\begin{rem} \label{r:Sp and CT}
Note that a variation of this argument shows that the natural transformation
$$\on{Sp}\circ \on{CT}_*\to \on{CT}_*\circ \on{Sp}$$
is an isomorphism.
\end{rem}

\begin{cor} \label{c:van and CT}
The natural transformation
$$\Phi\circ \on{CT}_*\to \on{CT}_*\circ \Phi$$
is an isomorphism.
\end{cor}

\begin{proof}

Follows from the fiber sequence
$$\Phi\to \Psi \to \bi^!,$$
where the functor $\bi^!$ obviously commutes with $\on{CT}_*$.

\end{proof} 

\sssec{}

By \corref{c:van and CT}, it suffices to show that
$$\Phi\circ \on{CT}_*(\on{Poinc}^{\on{Vac}}_{!,\sR_0}))=0.$$

\medskip

We use again the isomorphism \eqref{e:weird adj}, and we rewrite 
$$\on{CT}_*(\on{Poinc}^{\on{Vac}}_{!,\sR_0})\simeq
\on{CT}^-_!(\on{Poinc}^{\on{Vac}}_{!,\sR_0}).$$
 
Thus, we have to prove:

\begin{equation} \label{e:CT Vac Sp} 
\Phi\circ \on{CT}^-_!(\on{Poinc}^{\on{Vac}}_{!,\sR_0})=0.
\end{equation}  

\medskip

We claim:

\begin{thm} \label{t:CT of Poinc}
There is a canonical isomorphism
$$(\on{transl}_{\rho_P(\omega_X)})^*\circ \on{CT}^-_!(\on{Poinc}^{\on{Vac}}_{G,!})[d]\simeq
\on{Fact}(\Omega^{\on{loc}})\star \on{Poinc}^{\on{Vac}}_{M,!},$$
where:
\begin{itemize}

\item $\on{transl}_{\rho_P(\omega_X)}$ is the automorphism of $\Bun_M$ given by translation by
$\rho_P(\omega_X)\in \Bun_{Z_M}$; 

\smallskip

\item The integer $d$ is as in \cite[Corollary 10.1.8]{GLC3};

\smallskip

\item The functor $(\on{Fact}(\Omega^{\on{loc}})\star (-))$ is as in \thmref{t:Whit of Eis}.

\end{itemize} 

\end{thm} 

The assertion of \thmref{t:CT of Poinc} is obtained from that of \thmref{t:Whit of Eis} proven below
by a duality manipulation (see \cite[Corollary 10.1.8]{GLC3}). 

\begin{rem}

Alternatively, one can prove \thmref{t:CT of Poinc} by rerunning the argument of \thmref{t:Whit of Eis}.
However, the proof is simpler as here one works with the open Zastava space $\overset{\circ}{\on{Zast}}$
and one does not need any local acyclicity assertions.

\end{rem} 

\sssec{}

Thus, we have to show that
$$\Phi(\on{Fact}(\Omega^{\on{loc}})\star \on{Poinc}^{\on{Vac}}_{M,!,\sR_0})=0.$$

By \propref{p:Sp and Hecke}, 
$$\Phi(\on{Fact}(\Omega^{\on{loc}})\star \on{Poinc}^{\on{Vac}}_{M,!,\sR_0})\simeq
\on{Fact}(\Omega^{\on{loc}})\star \Phi(\on{Poinc}^{\on{Vac}}_{M,!,\sR_0}).$$

Now, by induction on the semi-simple rank, we can assume that $\Phi(\on{Poinc}^{\on{Vac}}_{M,!,\sR_0})=0$,
and the assertion follows. 

\ssec{The cuspidal part}

\sssec{}

The statement that we want to prove is that the map
\begin{equation} \label{e:on cusp}
\be^L_{\on{cusp}}(\on{Poinc}^{\on{Vac}}_{!,\sk})\simeq
\be^L_{\on{cusp}}\circ \bi_0^*(\on{Poinc}^{\on{Vac}}_{!,\sR_0})
\to \be^L_{\on{cusp}}\circ \Psi(\on{Poinc}^{\on{Vac}}_{!,\sK_0})
\end{equation} 
is an isomorphism.

\medskip

Thanks to \thmref{t:IC acycl}, the functors $\bi_0^*$, $\bj_{0,*}$, and
$\Psi$ are defined on the cuspidal 
category, viewed as a \emph{quotient} of $\Shv(\Bun_G)$, i.e., in a way that commutes
with the projection
$$\be_{\on{cusp}}^L:\Shv(\Bun_G)\to \Shv(\Bun_G)_{\on{cusp}}.$$

%
%in a way that commutes with the funct
%
%
%
%
%First, we note that thanks to \thmref{t:IC acycl}, the functor $\Phi$ is well-defined as a functor
%$$\Phi_{\on{cusp}}:\Shv(\Bun_{G,\sR_0}))_{\on{cusp}}\to \Shv(\Bun_{G,\sk}))_{\on{cusp}}$$
%so that the diagram
%$$
%\CD
%\Shv(\Bun_{G,\sR_0})) @>{\Phi}>>  \Shv(\Bun_{G,\sk}))  \\
%@V{\be^L_{\on{cusp}}}VV @VV{\be^L_{\on{cusp}}}V \\
%\Shv(\Bun_{G,\sR_0}))_{\on{cusp}} @>{\Phi_{\on{cusp}}}>>  \Shv(\Bun_{G,\sk}))_{\on{cusp}} 
%\endCD
%$$
%commutes.

\sssec{}

Recall that $\on{Poinc}^{\on{Vac}}_!\in \Shv(\Bun_G)$ was defined as
$$(\sfp/T)_!(\on{exp}_{\chi^I}/T):=(\chi^I/T)^*(\on{exp}^{\boxtimes I}/T),$$
where $\on{exp}^{\boxtimes I}/T$ is as in \cite[Sect. 3.3]{GLC1}. 

\medskip

We will now change the notations
$$\on{exp}^{\boxtimes I}/T\rightsquigarrow \on{exp}_!^{\boxtimes I}/T \text{ and }
\on{exp}_{\chi^I}/T \rightsquigarrow \on{exp}_{\chi^I,!}/T$$ 
and we note that there exists another object
$$\on{exp}_*^{\boxtimes I}/T\in \Shv(\BG_a^I/T),$$
see \secref{sss:Kir on AI}. 

\medskip

In addition, there is a canonically defined map
\begin{equation} \label{e:Poinc ! to * pre}
\on{exp}_!^{\boxtimes I}/T\to \on{exp}_*^{\boxtimes I}/T,
\end{equation}
see \eqref{e:! to * exp}. 

\sssec{}

Denote 
$$\on{exp}_{\chi^I,*}/T:=(\chi^I/T)^*(\on{exp}_*^{\boxtimes I}/T).$$

Let
$$\on{Poinc}^{\on{Vac}}_*\in \Shv(\Bun_G)$$
be the object equal to 
$$(\sfp/T)_*(\on{exp}_{\chi^I,*}/T).$$

\medskip

The map \eqref{e:Poinc ! to * pre} gives rise to a map 
\begin{equation} \label{e:Poinc ! to *}
\on{Poinc}^{\on{Vac}}_!\to \on{Poinc}^{\on{Vac}}_*.
\end{equation}

We will prove:

\begin{thm} \label{t:Poinc ! to *}
The cone of \eqref{e:Poinc ! to *} belongs to $\Shv(\Bun_G)_\Eis$.
\end{thm} 

Actually, the proof shows more, namely that for any $\CG \in \Shv(\Spec(\sR_0))$, the morphism
\begin{equation} \label{e:Poinc ! to * bis}
\sfp_{\sR_0,!}(\pi_{\Bun_{N,\rho(\omega),\sR_0}}^*(\CG)
\overset{*}{\otimes} \on{exp}_{\chi^I,!}/T)
\to 
\sfp_{\sR_0,*}(\pi_{\Bun_{N,\rho(\omega),\sR_0}}^*(\CG)
\overset{*}{\otimes} \on{exp}_{\chi^I,*}/T)
\end{equation} 
has cone lying in $\Shv(\Bun_{G,\sR_0})_{\Eis}$;
here $\pi_{\Bun_{N,\rho(\omega),\sR_0}}$ is the projection
from $\Bun_{N,\rho(\omega),\sR_0}$ to $\Spec(\sR_0)$.

\medskip 

The proof of \thmref{t:Poinc ! to *} will be given in \secref{s:asymptotics}. We now proceed with the 
proof of the fact that \eqref{e:on cusp} is an isomorphism. 

\sssec{}

Take $\CE$ to be a lisse sheaf on $\Spec(\sK_0)$.
We have a commutative diagram
\begin{equation}\label{e:Poinc cusp}
\begin{tikzcd}
\on{Poinc}^{\on{Vac}}_{!,\sR_0} \overset{*}{\otimes} 
\bj_{0,*}(\CE) 
\arrow[rr]
\arrow[d]
&& 
\bj_{0,*}(\on{Poinc}^{\on{Vac}}_{!,\sK_0} 
\overset{*}{\otimes} \CE)
\arrow[d] \\
\on{Poinc}^{\on{Vac}}_{*,\sR_0} \overset{*}{\otimes} 
\bj_{0,*}(\CE)
\arrow[rr]
&&
\bj_{0,*}(\on{Poinc}^{\on{Vac}}_{*,\sK_0} 
\overset{*}{\otimes} \CE).
\end{tikzcd}
\end{equation}

\noindent The vertical arrows have Eisenstein cones
by \thmref{t:Poinc ! to *}. Moreover, the composition
from upper left to bottom right has Eisenstein cone by the property of \eqref{e:Poinc ! to * bis}
stated above. 
Therefore, each arrow in the
above square becomes an isomorphism after applying
$\be^L_{\on{cusp}}$.

\medskip 

Applying $\bi_0^*$, we obtain a similar diagram
\begin{equation}\label{e:Poinc cusp 0}
\begin{tikzcd}
\bi_0^*(\on{Poinc}^{\on{Vac}}_{!,\sR_0}) {\overset{*}\otimes} 
\bi_0^*\circ \bj_{0,*}(\CE) 
\arrow[rr]
\arrow[d]
&& 
\bi_0^*\circ \bj_{0,*}(\on{Poinc}^{\on{Vac}}_{!,\sK_0} 
\overset{*}{\otimes} \CE)
\arrow[d] \\
\bi_0^*(\on{Poinc}^{\on{Vac}}_{*,\sR_0}) {\overset{*}\otimes} 
\bi_0^*\circ \bj_{0,*}(\CE)
\arrow[rr]
&&
\bi_0^*\circ \bj_{0,*}(\on{Poinc}^{\on{Vac}}_{*,\sK_0} 
\overset{*}{\otimes} \CE)
\end{tikzcd}
\end{equation}

\noindent in which again all arrows have Eisenstein
cones. By Beilinson's construction of
nearby cycles, we can pass to a colimit of such $\CE$'s 
so that the top horizontal arrow in the above diagram
becomes the map
\[
\bi_0^*(\on{Poinc}^{\on{Vac}}_{!,\sR_0}) =
\bi_0^*(\on{Poinc}^{\on{Vac}}_{!,\sR_0}) {\otimes} 
\Psi(\sfe_{\Spec(K_0)})
\to 
\Psi(\on{Poinc}^{\on{Vac}}_{!,\sK_0}),
\]

\noindent which we deduce has Eisenstein cone.

\section{Comparison of !- vs *- Poincar\'e objects} \label{s:asymptotics} %\thmref{t:Poinc ! to *}} \label{s:asymptotics}

The goal of this section is to prove \thmref{t:Poinc ! to *}. We continue to assume that $G$ is semi-simple. 

\ssec{The case when there exists the exponential sheaf}

Note that we only used \thmref{t:Poinc ! to *} over $\sK$ and $\sR$. However, we will first give a proof
over a field of positive characteristic (or for D-modules), since it conveys the intuitive picture.

\begin{rem} 
In the course of the proof, we will see that the cone of \eqref{e:Poinc ! to *} admits 
a canonical filtration, whose associated graded can be described explicitly, see Remark
\ref{r:asympt}.

\end{rem}.

\sssec{}

We start by rewriting the objects 
\begin{equation} \label{e:exp T}
\on{exp}_!^{\boxtimes I}/T \text{ and } \on{exp}_*^{\boxtimes I}/T
\end{equation} 
in terms of the exponential sheaf\footnote{The definition of these objects in \secref{sss:Kir on AI} uses 
the Kirillov model and hence avoids the exponential sheaf.}. 

\medskip

Namely, we start with 
$$\on{exp}^{\boxtimes I}\in \Shv(\BG_a^I)$$
(here $I$ is the set of vertices of the Dynkin diagram) 
and the objects \eqref{e:exp T} are its !- and *- direct images, respectively,
along the map
$$\BG_a^I\to \BG_a^I/T.$$

%\medskip
%
%Consider $(\BA^1)^I$, thought of as the \emph{dual} of $\BG_a^I$, and consider the open embedding
%$$\jmath:(\BA^1-0)^I\hookrightarrow (\BA^1)^I.$$
%as well as
%$$\jmath/T:(\BA^1-0)^I/T\simeq (\BA^1)^I/T.$$
%
%Note that we can identify 
%$$(\BA^1-0)^I/T\simeq \on{pt}/Z_G.$$

\sssec{}

Consider the stack
$$\ol{\Bun}_{N,\rho(\omega)}\times (\BA^1)^I.$$

Let $f$ denote the projection 
$$\ol{\Bun}_{N,\rho(\omega)}\times (\BA^1)^I\to \ol{\Bun}_{N,\rho(\omega)}.$$

We consider the canonical $T$-action on $\ol{\Bun}_{N,\rho(\omega)}\times (\BA^1)^I$, where the action
on the second factor is via
$$T\to T_{\on{adj}}\overset{\on{simple\, roots}}\simeq \BG_m^I.$$

Let $f/T$ denote the projection 
$$(\ol{\Bun}_{N,\rho(\omega)}\times (\BA^1)^I)/T\to \ol{\Bun}_{N,\rho(\omega)}/T.$$

\sssec{}

Let $\bj$ denote the open embedding
$$\ol{\Bun}_{N,\rho(\omega)}\times (\BA^1-0)^I\hookrightarrow \ol{\Bun}_{N,\rho(\omega)}\times (\BA^1)^I.$$
and $\bj/T$ the embedding 
$$\ol{\Bun}_{N,\rho(\omega)}/Z_G\simeq (\ol{\Bun}_{N,\rho(\omega)}\times (\BA^1{}-0)^I)/T\hookrightarrow 
(\ol{\Bun}_{N,\rho(\omega)}\times (\BA^1)^I)/T.$$

\medskip

We can rewrite
\begin{equation} 
\on{Poinc}^{\on{Vac}}_?\simeq (\ol\sfp/T)_?\circ (f/T)_?\circ (\bj/T)_?\circ (\pi_{Z_G})_?\circ j_?\circ 
(\chi^I)^*(\on{exp}^{\on{\boxtimes I}}),
\end{equation} 
where:

\begin{itemize}

\item $?$ is either $!$ or $*$;

\smallskip

\item $j$ denote the embedding $\Bun_{N,\rho(\omega)}\hookrightarrow \ol{\Bun}_{N,\rho(\omega)}$; 

\smallskip

\item $\pi_{Z_G}$ denotes the projection $\ol{\Bun}_{N,\rho(\omega)}\to \ol{\Bun}_{N,\rho(\omega)}/Z_G$.

\end{itemize} 

\medskip

Note that:

\smallskip

\begin{itemize}

\item The map $\ol\sfp/T$ is proper, so $(\ol\sfp/T)_!\to (\ol\sfp/T)_*$ is an isomorphism;

\medskip

\item The map $\pi_{Z_G}$ is finite\footnote{Recall that $G$ was assumed semi-simple.}, 
so $(\pi_{Z_G})_!\to (\pi_{Z_G})_*$ is an isomorphism;

\medskip

\item The extension of $(\chi^I)^*(\on{exp}^{\on{\boxtimes I}})$ along $j$ is clean, so the map
$$j_!\circ (\chi^I)^*(\on{exp}^{\on{\boxtimes I}})\to j_*\circ (\chi^I)^*(\on{exp}^{\on{\boxtimes I}})$$
is an isomorphism. 

\end{itemize} 

\sssec{}  \label{sss:ext Whit}

We can talk about the full category
$$\Whit^{\on{ext}}(\ol{\Bun}_{N,\rho(\omega)}\times (\BA^1)^I)\subset  \Shv(\ol{\Bun}_{N,\rho(\omega)}\times (\BA^1)^I),$$
where the ``extended Whittaker condition" depends in  the point in $(\BA^1)^I$ (we think of this $(\BA^1)^I$ as the 
variety of characters of $\BG_a^I$), see \secref{sss:Four}.

\begin{rem}

The notation $\Whit^{\on{ext}}$ (and the idea thereof) is borrowed from \cite[Sect. 8]{Ga2}, where the
\emph{extended Whittaker category} is studied. 

\end{rem} 

\sssec{}

We can also consider the corresponding equivariant version
$$\Whit^{\on{ext}}(\ol{\Bun}_{N,\rho(\omega)}\times (\BA^1)^I)^{T}\subset  
\Shv(\ol{\Bun}_{N,\rho(\omega)}\times (\BA^1)^I)^{T}.$$

We will prove:

\begin{prop} \label{p:rest is Eis}
For $\CF\in \Whit^{\on{ext}}(\ol{\Bun}_{N,\rho(\omega)}\times (\BA^1)^I)^{T}$, supported
off $\ol{\Bun}_{N,\rho(\omega)}\times (\BA^1{}-0)^I$, the object
$$(\ol\sfp/T)_!\circ (f/T)_!(\CF)\in \Shv(\Bun_G)$$
is Eisenstein.
\end{prop} 

\begin{prop} \label{p:Four-proper}
For $\CF\in \Whit^{\on{ext}}(\ol{\Bun}_{N,\rho(\omega)}\times (\BA^1)^I)$, the map
$$f_!(\CF)\to f_*(\CF)$$ is an isomorphism.
\end{prop}

It is clear that the combination of these two propositions implies that \eqref{e:Poinc ! to *} is Eisenstein.

\ssec{Proof of \propref{p:rest is Eis}}

\sssec{}

For a subset $J\subset I$, let 
$$\bi_J:(\ol{\Bun}_{N,\rho(\omega)}\times (\BA^1-0)^J)\subset (\ol{\Bun}_{N,\rho(\omega)}\times (\BA^1)^I)$$
be the embedding of the corresponding stratum, so that $\bi_I=\bj$. 

\medskip

Denote by
$$\Whit^{\on{part}}(\ol{\Bun}_{N,\rho(\omega)}\times (\BA^1-0)^J)\subset  \Shv(\ol{\Bun}_{N,\rho(\omega)}\times (\BA^1-0)^J)$$
and 
$$\Whit^{\on{part}}(\ol{\Bun}_{N,\rho(\omega)}\times (\BA^1-0)^J)^{T}\subset  
\Shv(\ol{\Bun}_{N,\rho(\omega)}\times (\BA^1-0)^J)^{T}.$$
the corresponding subcategories, obtained by imposing the Whittaker-type equivariance condition.  

\begin{rem}

The notation $\Whit^{\on{part}}$ (and the idea thereof) is borrowed from \cite[Sect. 7]{Ga2}.

\end{rem}

\sssec{}

We will show that the functor
$$(\ol\sfp/T)_!\circ (f/T)_!\circ (\bi_J/T)_!:
\Whit^{\on{part}}(\ol{\Bun}_{N,\rho(\omega)}\times (\BA^1-0)^J)^{T}\to \Bun_G$$
factors through the subcategory generated by
$$\Eis_!:\Shv(\Bun_M)\to \Shv(\Bun_G),$$
where $P$ is the standard parabolic corresponding to $J$.

\sssec{} 

We stratify $\ol{\Bun}_{N,\rho(\omega)}$ by $(\ol{\Bun}_{N,\rho(\omega)})_\lambda$, $\lambda\in \Lambda^{\on{pos}}$,
$$(\ol{\Bun}_{N,\rho(\omega)})_\lambda\simeq \Bun_B\underset{\Bun_T}\times X^{(\lambda)},$$
where $X^{(\lambda)}\to \Bun_T$ is the Abel-Jacobi map, shifted by $\rho(\omega_X)$. 
Let $i_\lambda$ denote the corresponding locally closed embedding. 

\medskip

Consider the corresponding subcategories
$$\Whit^{\on{part}}((\ol{\Bun}_{N,\rho(\omega)})_\lambda\times (\BA^1-0)^J)\subset 
\Shv((\ol{\Bun}_{N,\rho(\omega)})_\lambda\times (\BA^1-0)^J)$$
and 
$$\Whit^{\on{part}}(\ol{\Bun}_{N,\rho(\omega)})_\lambda\times (\BA^1-0)^J)^{T}\subset  
\Shv((\ol{\Bun}_{N,\rho(\omega)})_\lambda\times (\BA^1-0)^J)^{T}.$$

\medskip

We will show that the functor 
$$(\ol\sfp/T)_!\circ (f/T)_!\circ (\bi_J/T)_!\circ (i_\lambda/T)_!:
\Whit^{\on{part}}((\ol{\Bun}_{N,\rho(\omega)})_\lambda\times (\BA^1-0)^J)^{T}\to \Bun_G$$
factors as
$$\Whit^{\on{part}}((\ol{\Bun}_{N,\rho(\omega)})_\lambda\times (\BA^1-0)^J)^{T}\to
\Bun_M \overset{\Eis_!}\to \Bun_G.$$

\sssec{}

Note that the map 
$$\left((\ol{\Bun}_{N,\rho(\omega)})_\lambda\times (\BA^1-0)^J\right)/T\overset{(\ol\sfp/T)\circ (f/T)\circ (\bi_J/T) \circ (i_\lambda/T)}\longrightarrow
 \Bun_G$$
factors as
$$\left((\ol{\Bun}_{N,\rho(\omega)})_\lambda\times (\BA^1-0)^J\right)/T\overset{f_P}\longrightarrow 
\Bun_P\overset{\sfp_P}\longrightarrow \Bun_G.$$

\medskip

Moreover, we have a Cartesian diagram

\smallskip

$$
\CD
\left((\ol{\Bun}_{N,\rho(\omega)})_\lambda \times (\BA^1-0)^J\right)/T \\
@V{\simeq }VV \\
\left((\Bun_B\underset{\Bun_T}\times X^{(\lambda)})\times (\BA^1-0)^J\right)/T @>{'\!\sfq_P}>> 
\left((\Bun_{B(M)}\underset{\Bun_T}\times X^{(\lambda)}) \times (\BA^1-0)^J\right)/T \\
@V{f_P}VV @VV{'\!f_P}V \\
\Bun_P @>{\sfq_P}>> \Bun_M \\
@V{\sfp_P}VV \\
\Bun_G
\endCD
$$
and every object from $\Whit^{\on{part}}\left((\ol{\Bun}_{N,\rho(\omega)})_\lambda\times (\BA^1-0)^J\right)^{T}$ 
is isomorphic to the *-pullback by means of $'\!\sfq_P$ of an object in
$$\Shv\left(\Bigl((\Bun_{B(M)}\underset{\Bun_T}\times X^{(\lambda)}) \times (\BA^1-0)^J\Bigr)/T\right).$$

\sssec{}

Now, 
\begin{multline}  \label{e:asymp}
(\ol\sfp/T)_!\circ (f/T)_!\circ (\bi_J/T)_!\circ (i_\lambda/T)_!\circ ({}'\!\sfq_P)^*\simeq \\
\simeq (\sfp_P)_!\circ (f_P)_!\circ ({}'\!\sfq_P)^*\simeq (\sfp_P)_!\circ (\sfp_P)^*\circ ({}'\!f_P)_!\simeq
\Eis_!\circ ({}'\!f_P)_!
\end{multline} 
(up to a cohomological shift\footnote{Which is involved in the definition of $\Eis_!$.}), as required.

\qed[\propref{p:rest is Eis}]
 
\begin{rem} \label{r:Vinb}

We have obtained that the cone of \eqref{e:Poinc ! to *} admits a canonical filtration indexed by $\emptyset \neq J\subset I$,
where the subquotient corresponding to a given $J$ in turn has a filtration indexed by $\lambda\in \Lambda^{\on{pos}}$,
and its subquotient corresponding to a given $\lambda$ is given by
\begin{equation} 
\Eis_!\Bigl(({}'\!f_P)_! \circ ({}'\!\sfq_P)_* \circ (i_\lambda/T)^*\circ (\bi_J/T)^* \circ (\bj/T)_*\circ (\pi_{Z_G})_!\circ 
j_!\circ (\chi^I)^*(\on{exp}^{\on{\boxtimes I}})\Bigr),
\end{equation} 
(up to a cohomological shift).

\medskip

One can describe the cone of \eqref{e:Poinc ! to *} more conceptually, by combining the results of \cite{Chen} and \cite{Lin}. Namely,
it has a canonical filtration indexed by the poset of proper parabolics with the associated graded corresponding to a given $P$
being 
$$\Eis_P^{\on{enh}}\circ \on{CT}_P^{\on{enh}}(\on{Poinc}^{\on{Vac}}_!),$$
where we refer the reader to \cite{Chen} for the ``enhanced" notation.

\end{rem}

\begin{rem} \label{r:asympt}

The object 
$$({}'\!f_P)_! \circ ({}'\!\sfq_P)_* \circ (i_\lambda/T)^*\circ (\bi_J/T)^* \circ (\bj/T)_*\circ (\pi_{Z_G})_!\circ 
j_!\circ (\chi^I)^*(\on{exp}^{\on{\boxtimes I}})\in \Shv(\Bun_M)$$
is closely related to (and can be algorithmically expressed via) the object
$$({}'\!f_P)_! \circ ({}'\!\sfq_P)_* \circ (i_\lambda/T)^!\circ (\bi_J/T)^! \circ (\bj/T)_*\circ (\pi_{Z_G})_!\circ 
j_!\circ (\chi^I)^*(\on{exp}^{\on{\boxtimes I}})\in \Shv(\Bun_M).$$

The computation of the latter objects is the main goal of the paper \cite{AG2}. Namely, it says that this object 
identifies with 
$$\on{CT}_*(\on{Poinc}^{\on{Vac}}_!),$$
which in turn be calculated using \thmref{t:CT of Poinc}. 

\end{rem} 

\ssec{Proof of \propref{p:Four-proper}}

\sssec{} \label{sss:Four}

The proof will fit into the following general paradigm:

\medskip

Let $\CY$ be a stack acted on by a vector group $V$. Consider the corresponding category
$$\Whit(V^*\times \CY)\subset \Shv(V^*\times \CY).$$

Namely, this is the full subcategory consisting of objects $\CF\in \Whit(V^*\times \CY)$, equipped with an 
isomorphism
$$(\on{id}\times \on{act})^*(\CF)\simeq (p_{1,2}^*\circ \on{ev}^*(\on{exp}))\overset{*}\otimes p_{1,3}^*(\CF)$$
in $\Shv(V^*\times V\times \CY)$ that restricts to the identity map\footnote{Since $V$ is unipotent as a group,
the full equivariance structure is actually a condition, which is equivalent to the simply-minded one above.}
on $V^*\times \{0\}\times \CY$, where 

\begin{itemize}

\item $\on{act}$ denotes the action map $V\times \CY\to \CY$;

\item $\on{ev}$ denotes the evaluation map $V^*\times V\to \BG_a$.

\end{itemize} 

\medskip

Note that the functor
$$\Shv(\CY)\overset{\on{act}^*}\to \Shv(V\times \CY) \overset{\on{Four}_Y}\longrightarrow \Shv(V^*\times \CY)$$
gives rise to an equivalence 
$$\Shv(\CY)\overset{\sim}\to \Whit(V^*\times \CY).$$

\sssec{}

Let $f$ denote the map $V^*\times \CY\to \CY$. We claim:

\begin{prop} \label{p:Four-proper abs}
The natural transformation $f_!\to f_*$ becomes an isomorphism when evaluated an objects of  $\Whit(V^*\times \CY)$.
\end{prop}

\begin{proof} 

It enough to prove the isomorphism after the (smooth) pullback by means of the map 
$$V\times \CY\overset{\on{act}}\to \CY.$$

This reduces us to the case when $\CY$ has the form $V\times \CZ$ with $V$ acting on the 
first factor. 

\medskip

In this case, the equivalence
$$\Shv(V\times \CZ) \to \Whit(V^*\times V\times \CZ)$$
is given by 
$$\CF\mapsto p_{1,3}^*(\on{Four}_\CZ(\CF))\overset{*}\otimes p^*_{1,2}(\on{mult}^*(\on{exp})).$$

The operations
$$\CG \in \Shv(V^*\times \CZ) \rightsquigarrow 
f_!\left(p_{1,3}^*(\CG)\overset{*}\otimes p^*_{1,2}(\on{mult}^*(\on{exp})\right) \text{ and }
f_*\left(p_{1,3}^*(\CG)\overset{*}\otimes p^*_{1,2}(\on{mult}^*(\on{exp})\right)$$
are the !- and *- versions of the functor 
$$\on{Four}_\CZ:\Shv(V^*\times \CZ)\to \Shv(V\times \CZ).$$

Now, it is well-know that the natural transformation
$$\on{Four}_{!,\CZ}\to \on{Four}_{*,\CZ}$$
is an isomorphism. 

\end{proof}

\sssec{} \label{sss:adapt 1}

We apply the above paradigm as follows. Cover $\ol{\Bun}_{N,\rho(\omega)}$ by open substacks
$$\ol{\Bun}_{N,\rho(\omega),\text{good\,at\,x}}, \quad x\in X,$$
where we require that the generalized $B$-reduction be \emph{non-degenerate} at $x$.  It is enough to show
that the map $f_!(\CF)\to f_*(\CF)$ restricts to an isomorphism over every 
$$\ol{\Bun}_{N,\rho(\omega),\text{good\,at\,x}}.$$

\sssec{}

Consider the corresponding stack 
$$\ol{\Bun}^{\on{level}_x}_{N,\rho(\omega),\text{good\,at\,x}},$$
see \cite[Sect. 4.4.2]{Ga3}. 

\medskip

The stack $\ol{\Bun}^{\on{level}_x}_{N,\rho(\omega),\text{good\,at\,x}}$ is acted on by the group ind-scheme
$\fL(N_{\rho(\omega)})_x$. Let 
$$\chi^I_x:\fL(N_{\rho(\omega)})_x\to \BG^I_a$$
denote the canonical character. 

\sssec{}  \label{sss:adapt 3}

Let $N'\subset \fL(N_{\rho(\omega)})_x$ be a sufficiently large subgroup. Let 
$$\overset{\circ}{N}{}':=\on{ker}(\chi^I_x|_{N'}).$$

\medskip

Set 
$$\CY:=\ol{\Bun}^{\on{level}_x}_{N,\rho(\omega),\text{good\,at\,x}}/\overset{\circ}{N}{}'.$$

This is a stack locally of finite type, which carries an action of $V:=\BG_a^I$. Finally, we note 
that the category $\Whit(V^*\times \CY)$ identifies with the category 
$$\Whit^{\on{ext}}(\ol{\Bun}_{N,\rho(\omega),\text{good\,at\,x}}\times (\BA^1)^I),$$
where we think of $(\BA^1)^I$ as $V^*$. 

\qed[\propref{p:Four-proper}]

\ssec{The Kirillov model} \label{ss:Kir}

As a preparation to the proof of \thmref{t:Poinc ! to *} in the setting where the exponential
sheaf does not exist, we discuss the fomalism of \emph{Kirillov} (as opposed to Whittaker)
models.

\sssec{} \label{sss:Kir1}

Let $\CY$ be a stack equipped with an action of $\BG_a^I$. For every subset $J\subset I$, denote by
$$\Shv(\CY)_{J\on{-cl}}\overset{(\ol\bi_J)_!=(\ol\bi_J)_*}\hookrightarrow \Shv(\CY)$$
the embedding of the full subcategory $\Shv(\CY)^{\BG_a^J}$. 

\medskip

The above functor admits a left and right adjoints, denoted $(\ol\bi_J)^*$ and $(\ol\bi_J)^!$,
explicitly given by $\on{Av}_!^{\BG_a^J}$ and $\on{Av}_*^{\BG_a^J}$, respectively.

\sssec{}

Let $\Shv(\CY)_J$ be the quotient of $\Shv(\CY)_{J\on{-cl}}$ obtained by modding out
with respect to all $\Shv(\CY)_{J'\on{-cl}}$ with $J'\supset J$. Denote by $(\bj_J)^*=(\bj_J)^!$
the projection
$$\Shv(\CY)_{J\on{-cl}}\twoheadrightarrow \Shv(\CY)_J.$$

This projection admits left and right adjoints, denoted $(\bj_J)_!$ and $(\bj_J)_*$, respectively.

\sssec{} \label{sss:Kir3}

Denote
$$(\bi_J)_!:=(\ol\bi_J)_!\circ (\bj_J)_! \text{ and } (\bi_J)_*:=(\ol\bi_J)_*\circ (\bj_J)_*, \quad 
\Shv(\CY)_J\to \Shv(\CY).$$

These functors admit right and left adjoints, given by
$$(\bi_J)^!:=(\bj_J)^! \circ (\ol\bi_J)^! \text{ and } (\bi_J)^*:=(\bj_J)^* \circ (\ol\bi_J)^*,$$
respectively.

\medskip

Thus, we obtain a stratification of $\Shv(\CY)$ with subquotients $\Shv(\CY)_J$.  

\sssec{}

In particular, we have the ``open stratum" corresponding to $J=\emptyset$. 

\medskip

Note that the essential image of $(\bj_\emptyset)_!$ (resp., $(\bj_\emptyset)_*$) consists
of objects for which the !- (resp., *-) averaging for any coordinate copy of $\BG_a\subset \BG_a^I$
vanishes. 

\medskip

On general categorical grounds, we have a natural transformation
\begin{equation} \label{e:! to* Kir}
(\bj_\emptyset)_!\to (\bj_\emptyset)_*.
\end{equation} 

\sssec{} 

Let now $T$ be a torus equipped with a surjection onto $\BG_m^I$. Denote
$$Z_G:=\on{ker}(T\to \BG_m^I).$$

Assume that 
the action of $\BG_a^I$ on $\CY$ can be extended to an action of $T\ltimes \BG_a^I$.

\medskip

Then the discussion in Sects. \ref{sss:Kir1}-\ref{sss:Kir3} is applicable to 
$\Shv(\CY/T)$. In particular, we obtain the sub/quotient categories
$$\Shv(\CY/T)_{J\on{-cl}} \text{ and } \Shv(\CY/T)_J,$$
and the corresponding functors 
$$(\bi_J/T)_!, (\bi_J/T)_*, (\bi_J/T)^!, (\bi_J/T)^*, \text{etc.}$$

\sssec{}

For the next few subsection we will assume that we are in the situation in which the exponential sheaf exists,
and we will explain the connection with the Whittaker picture.

\medskip 

Consider the full subcategories
$$\Whit^{\on{ext}}(\CY\times (\BA^1)^I) \subset \Shv(\CY\times (\BA^1)^I)$$
and 
$$\Whit^{\on{ext}}(\CY\times (\BA^1)^I)^T \subset \Shv(\CY\times (\BA^1)^I)^T,$$
cf. \secref{sss:ext Whit}.

\medskip

Consider the corresponding functors
$$\bj_!: \Whit^{\on{ext}}(\CY\times (\BA^1-0)^I) \to \Whit^{\on{ext}}(\CY\times (\BA^1)^I)$$
and 
$$(\bj/T)_!: \Whit^{\on{ext}}(\CY\times (\BA^1-0)^I)^T \to \Whit^{\on{ext}}(\CY\times (\BA^1)^I)^T.$$

\medskip

Recall now that according to \secref{sss:Four}, we have an equivalence 
$$\Whit^{\on{ext}}(\CY\times (\BA^1)^I)\simeq \Shv(\CY),$$
explicitly given by !- (equivalently, *) direct image along the projection $\CY\times (\BA^1)^I\to \CY$.

\medskip

This induces an equivalence
\begin{equation} \label{e:Kir ext}
\Whit^{\on{ext}}(\CY\times (\BA^1)^I)^T\simeq \Shv(\CY/T).
\end{equation}

The following is an elementary verification:

\begin{lem}
The equivalence \eqref{e:Kir ext} fits onto the commutative diagram
$$
\CD
\Whit^{\on{ext}}(\CY\times (\BA^1-0)^I)^T  @>{(\bj/T)_!}>> \Whit^{\on{ext}}(\CY\times (\BA^1)^I)^T  \\
@V{\sim}VV @VV{\sim}V \\
\Shv(\CY/T)_\emptyset @>{(\bj_\emptyset/T)_!}>> \Shv(\CY/T).
\endCD
$$
\end{lem}

\sssec{}

Note that using the element $(1,...,1)\in (\BA^1)^I$, we can identify
$$\Shv(\CY\times (\BA^1-0)^I)^T)\simeq \Shv(\CY/Z_G)$$

Under this identification, the subcategory
$$\Whit^{\on{ext}}(\CY\times (\BA^1-0)^I)^T\subset \Shv(\CY\times (\BA^1-0)^I)^T)$$
corresponds to 
$$\Whit(\CY/Z_G)\subset \Shv(\CY/Z_G).$$

\medskip

We claim:

\begin{lem}
The equivalence
$$\Whit(\CY/Z_G)\simeq \Whit^{\on{ext}}(\CY\times (\BA^1-0)^I)^T \simeq \Shv(\CY/T)_\emptyset$$
is given by the composition
$$\Whit(\CY/Z_G)\hookrightarrow \Shv(\CY/Z_G)\overset{!-pushforward}\longrightarrow \Shv(\CY/T),$$
whose image lands in $\Shv(\CY/T)_\emptyset \subset \Shv(\CY/T)$.
\end{lem}

\sssec{} 

We now return to the situation when the exponential sheaf does not necessarily exist. We let
$\CY=\BG_a^I$, equipped with a natural action of $T\ltimes \BG_a^I$.

\medskip

Note that due to monodromicity, the essential image of the embedding 
$$\Shv(\BG_a^I/T)_\emptyset \overset{(\bj_\emptyset)_!}\hookrightarrow \Shv(\BG_a^I/T)$$
consists of objects whose !-restrictions to
$$\BG_a^J\subset \BG_a^I, \quad J\neq I$$
are $0$. I.e., these are objects that are *-extended from $(\BG_a-0)^I\overset{\jmath}\hookrightarrow \BG_a^I$. 

\medskip

Similarly, the essential image of the embedding 
$$\Shv(\BG_a^I/T)_\emptyset \overset{(\bj_\emptyset)_*}\hookrightarrow \Shv(\BG_a^I/T)$$
consists of objects whose *-restrictions to
$$\BG_a^J\subset \BG_a^I, \quad J\neq I$$
are $0$. I.e., these are objects that are !-extended from $(\BG_a-0)^I\overset{\jmath}\hookrightarrow \BG_a^I$. 

\medskip

We obtain that the category $\Shv(\BG_a^I/T)_\emptyset$ identifies with
$$\Shv((\BG_a-0)^I/T)\simeq \Shv(\on{pt}/Z_G)$$
\emph{in two different ways}\footnote{They differ by a cohomological shift.}. 

\medskip

In what follows we will use the identification
\begin{equation} \label{e:ident open}
\Shv(\BG_a^I/T)_\emptyset \overset{(\bj_\emptyset)_!}\hookrightarrow \Shv(\BG_a^I/T)\overset{(\jmath/T)^*}\to 
\Shv((\BG_a-0)^I/T)\simeq \Shv(\on{pt}/Z_G).
\end{equation} 

\sssec{} \label{sss:Kir on AI}

Denote
$$\on{exp}_\emptyset^{\boxtimes I}/T:=R_{Z_G}[-r]\in \Shv(\on{pt}/Z_G) 
\overset{\text{\eqref{e:ident open}}}\simeq \Shv(\BG_a^I/T)_\emptyset,$$
where:

\begin{itemize}

\item $R_{Z_G}\in \Shv(\on{pt}/Z_G)$ is the !-direct image of $\sfe\in \Vect=\Shv(\on{pt})$
along $\on{pt}\to \on{pt}/Z_G$;

\smallskip

\item $r=|I|$ is the semi-simple rank of $G$.

\end{itemize}

\medskip

In terms of the above identifications, we have
$$\on{exp}_!^{\boxtimes I}/T=(\bj_\emptyset)_!(\on{exp}_\emptyset^{\boxtimes I}/T),$$

Set
$$\on{exp}_*^{\boxtimes I}/T:=(\bj_\emptyset)_*(\on{exp}_\emptyset^{\boxtimes I}/T).$$

Thus, explicitly, 
$$\on{exp}_!^{\boxtimes I}/T\simeq (\jmath/T)_*(R_{Z_G})[-r]$$
and 
$$\on{exp}_*^{\boxtimes I}/T\simeq (\jmath/T)_!(R_{Z_G})[-r].$$

\medskip

Note that the natural transformation \eqref{e:! to* Kir} gives rise to a map
\begin{equation} \label{e:! to * exp}
\on{exp}_!^{\boxtimes I}/T\to \on{exp}_*^{\boxtimes I}/T.
\end{equation} 

In other words, this is a map 
\begin{equation} \label{e:! to * exp again}
(\jmath/T)_*(R_{Z_G})[-r]\to (\jmath/T)_!(R_{Z_G})[-r],
\end{equation}
note the direction of the arrow!

\sssec{Examples}

Let us analyze the behavior of the map \eqref{e:! to * exp again} explicitly.

\medskip
. 
We consider the pullback of the map \eqref{e:! to * exp again} to $\BG_a^I$ itself. Thus,
we are dealing with the map
\begin{equation} \label{e:! to * exp again again}
\jmath_*(\pi_!(\sfe_T))\to \jmath_!(\pi_!(\sfe_T)),
\end{equation}
where $\pi:T\to \BG^I_m$. 

\medskip

Assume first that $T\to \BG_m^I$ is an isomorphism. Then both sides of \eqref{e:! to * exp again again},
shifted cohomologically by $[r]$, are perverse, and the cosocle (resp., socle) of
$\jmath_*(\sfe_{\BG_m^I})[r]$ (resp., of $\jmath_!(\sfe_{\BG_m^I})[r]$) is the $\delta$-function at $0$,
i.e., $\ul\sfe_{\on{pt}}$. 

\medskip

With these identifications, the map \eqref{e:! to * exp again} is
$$\jmath_*(\sfe_{\BG_m^I})[r] \to \ul\sfe_{\on{pt}}\to \jmath_*(\sfe_{\BG_m^I})[r].$$

\medskip

In the general case, both sides in \eqref{e:! to * exp again again} have additional 
direct factors, given by Kummer sheaves corresponding to non-zero characters of (the finite group) $Z_G$. The map 
\eqref{e:! to * exp again again} is the natural isomorphism on these factors. 

\ssec{The case when there is no exponential sheaf}

In this subsection we will treat the case of \thmref{t:Poinc ! to *} when we work either over a field
of characteristic $0$ or in mixed characteristic.

\sssec{}

We adapt the formalism of \secref{ss:Kir} to 
$\ol{\Bun}_{N,\rho(\omega_X)}$, using the method of Sects. \ref{sss:adapt 1}-\ref{sss:adapt 3}.

\medskip

Thus, we obtain a stratification of the category $\Shv(\ol{\Bun}_{N,\rho(\omega_X)})$ (resp., $\Shv(\ol{\Bun}_{N,\rho(\omega_X)}/T)$)
by sub/quotient categories $\Shv(\ol{\Bun}_{N,\rho(\omega_X)})_J$ (resp., $\Shv(\ol{\Bun}_{N,\rho(\omega_X)}/T)_J$), 
$J\subset I$, and similarly for $\Bun_{N,\rho(\omega_X)}$

\sssec{}

Consider the object
$$\on{exp}_{\chi^I,\emptyset}/T:=(\chi^I/T)^*(\on{exp}_\emptyset^{\boxtimes I}/T)\in 
\Shv(\Bun_{N,\rho(\omega_X)}/T)_\emptyset.$$

\medskip

First, we note:

\begin{lem} \label{l:on deg}
For $J\neq 0$, the projection
$$\Shv(\ol{\Bun}_{N,\rho(\omega_X)})\overset{(\bj_\emptyset)^*=(\bj_\emptyset)^!}\twoheadrightarrow
\Shv(\ol{\Bun}_{N,\rho(\omega_X)})_\emptyset$$
annihilates objects supported on $\ol{\Bun}_{N,\rho(\omega_X)}-\Bun_{N,\rho(\omega_X)}$.
\end{lem}

Hence, we obtain that $\on{exp}_{\chi^I,\emptyset}/T$ uniquely extends to an object of 
$\Shv(\ol{\Bun}_{N,\rho(\omega_X)}/T)_\emptyset$; we denote it by
$$\ol{\on{exp}}_{\chi^I,\emptyset}/T.$$

\sssec{}

Unwinding the constructions, we obtain that the objects $\on{Poinc}^{\on{Vac}}_!$ and $\on{Poinc}^{\on{Vac}}_*$ are
$$(\ol\sfp/T)_!\circ (\bj_\emptyset)_!(\ol{\on{exp}}_{\chi^I,\emptyset}/T) \text{ and }
(\ol\sfp/T)_*\circ (\bj_\emptyset)_*(\ol{\on{exp}}_{\chi^I,\emptyset}/T),$$
respectively, where $(\ol\sfp/T)_!\simeq (\ol\sfp/T)_*$, since the map $\ol\sfp$ is proper.

\medskip

Moreover, the map \eqref{e:Poinc ! to *} is induced by the natural transformation 
\eqref{e:! to* Kir}.

\medskip

Hence, in order to prove \thmref{t:Poinc ! to *}, it suffices to show that for $J\neq \emptyset$, the essential image of the functor
\begin{equation} \label{e:ext Poinc deg}
\Shv(\ol{\Bun}_{N,\rho(\omega_X)}/T)_{J\on{-cl}} \overset{(\ol\bi_J)_!=(\ol\bi_J)_*}\longrightarrow 
\Shv(\ol{\Bun}_{N,\rho(\omega_X)}/T) \overset{(\ol\sfp/T)_!}\longrightarrow \Shv(\Bun_G)
\end{equation}
lies in the essential image of the functor 
$$\Shv(\Bun_M) \overset{\Eis_!}\to \Shv(\Bun_G),$$
where $P\twoheadrightarrow M$ are the standard Levi and parabolic corresponding to $J$.

\medskip

This is done by a manipulation similar to the one used in the proof of \propref{p:rest is Eis}. 
We elaborate on this below. 

\sssec{}

We stratify $\ol{\Bun}_{N,\rho(\omega_X)}$ by  
$$(\ol{\Bun}_{N,\rho(\omega_X)})_\lambda \overset{i_\lambda}\hookrightarrow \ol{\Bun}_{N,\rho(\omega_X)},$$
and consider the corresponding categories
$$\Shv((\ol{\Bun}_{N,\rho(\omega_X)})_\lambda/T)_{J\on{-cl}} \overset{(\ol\bi_J)_!=(\ol\bi_J)_*}\longrightarrow 
\Shv((\ol{\Bun}_{N,\rho(\omega_X)})_\lambda/T).$$

We will show that the composition
\begin{multline} \label{e:ext Poinc deg 1}
\Shv((\ol{\Bun}_{N,\rho(\omega_X)})_\lambda/T)_{J\on{-cl}} \overset{(i_\lambda)_!}\hookrightarrow 
\Shv(\ol{\Bun}_{N,\rho(\omega_X)}/T)_{J\on{-cl}} \overset{(\ol\bi_J)_!=(\ol\bi_J)_*}\longrightarrow  \\
\to \Shv(\ol{\Bun}_{N,\rho(\omega_X)}/T) \overset{(\ol\sfp/T)_!}\longrightarrow \Shv(\Bun_G)
\end{multline}
factors as
$$\Shv((\ol{\Bun}_{N,\rho(\omega_X)})_\lambda/T)_{J\on{-cl}} \to \Shv(\Bun_M) \overset{\Eis_!}\to \Shv(\Bun_G),$$
where $P$ is the standard parabolic corresponding to $J\subset I$.

\medskip

This follows from the Cartesian diagram
$$
\CD
(\ol{\Bun}_{N,\rho(\omega_X)})_\lambda/T \\
@V{\sim}VV \\
(\Bun_B\underset{\Bun_T}\times X^{(\lambda)})/T @>{'\!\sfq_P}>> (\Bun_{B(M)}\underset{\Bun_T}\times X^{(\lambda)})/T \\
@VVV @VVV \\
\Bun_P @>{\sfq_P}>> \Bun_M \\
@V{\sfp_P}VV \\
\Bun_G,
\endCD
$$
using the following observation:

\begin{lem}
The functor of *-pullback along $'\!\sfq_P$ gives rise to an equivalence
$$\Shv((\Bun_{B(M)}\underset{\Bun_T}\times X^{(\lambda)})/T)\to 
\Shv((\ol{\Bun}_{N,\rho(\omega_X)})_\lambda/T)_{J\on{-cl}}.$$
\end{lem}

\qed[\thmref{t:Poinc ! to *}] 

\section{Langlands functor and Eisenstein series} \label{s:L and Eis}

The goal of this section is to prove \thmref{t:L and Eis}. 

\begin{rem} \label{r:diff GLC3}

The proof that we will give applies in any sheaf-theoretic context (e.g., de Rham and Betti). We note, however,
that the construction of the isomorphism of functors stated in \thmref{t:L and Eis}
is a priori different from that in \cite[Theorem 14.2.2]{GLC3}. 

\medskip

The reason that we give a (different) proof here is that some ingredients of the proof in \cite{GLC3}
(specifically, the semi-infinite geometric Satake) have only been developed in the de Rham context.  

\medskip

It is a good (but potentially quite involved) exercise to show that the two isomorphisms actually
agree. 

\end{rem} 

\ssec{Reducing to a statement about \texorpdfstring{$\BL^{\on{restr}}_{G,\on{coarse}}$}{L}}

This step is parallel to \cite[Sect. 14.2.3]{GLC3}. 

\sssec{}

Denote
$$\Eis^{-,\on{spec}}_{\on{coarse}}:=\Upsilon^\vee_{\LS^{\on{restr}}_\cG}\circ \Eis^{-,\on{spec}}_{\on{coarse}}\circ \Xi_{\LS^{\on{restr}}_\cM},\quad
\QCoh(\LS^{\on{restr}}_\cM)\to \QCoh(\LS^{\on{restr}}_\cG).$$

In other words, 
$$\Eis^{-,\on{spec}}_{\on{coarse}}=(\sfp^{-,\on{spec}})_*\circ (\sfq^{-,\on{spec}})^*;$$
note that map $\sfp^{-,\on{spec}}$ is schematic, so the functor $(\sfp^{-,\on{spec}})_*$ is well-behaved. 

\sssec{}

We claim that in order to construct the datum of commutativity for the diagram \eqref{e:L and Eis}, it is enough to 
do so for the diagram
\begin{equation} \label{e:L and Eis coarse}
\CD
\Shv_\Nilp(\Bun_M) @>{\BL^{\on{restr}}_{M,\on{coarse}}}>> \QCoh(\LS^{\on{restr}}_\cM) \\
@V{\Eis^-_{!,\rho_P(\omega_X)}[\delta_{(N^-_P)_{\rho_P(\omega_X)}}]}VV  @VV{\Eis^{-,\on{spec}}_{\on{coarse}}}V \\
 \Shv_\Nilp(\Bun_G) @>>{\BL^{\on{restr}}_{G,\on{coarse}}}> \QCoh(\LS^{\on{restr}}_\cG).
\endCD
\end{equation} 

Indeed, to prove the  commutativity of \eqref{e:L and Eis}, it is enough to show
that the two circuits are isomorphic when evaluated on compact objects of $\Shv_\Nilp(\Bun_G)$.

\medskip

Note that diagram \eqref{e:L and Eis coarse} is obtained from diagram \eqref{e:L and Eis} by composing both
circuits with
$$\Upsilon^\vee_{\LS^{\on{restr}}_\cG}:\IndCoh_\Nilp(\LS^{\on{restr}}_\cG)\to \QCoh(\LS^{\on{restr}}_\cG).$$

Since the functor $\Upsilon^\vee_{\LS^{\on{restr}}_\cG}$ is fully faithful on the eventually coconnective subcategory,
it suffices to show that both circuits in \eqref{e:L and Eis} send compact objects in $\Shv_\Nilp(\Bun_G)$
to eventually coconnective objects in $\IndCoh_\Nilp(\LS^{\on{restr}}_\cG)$.

\sssec{}

We first show this for the anti-clockwise circuit. 

\medskip

The functor $\Eis^-_!$ preserves compactness, and hence so does its translated version 
$\Eis^-_{!,\rho_P(\omega_X)}$. Hence, the required assertion follows from \thmref{t:L sends comp to even coconn}. 

\sssec{}

We now consider the clockwise circuit. The top horizontal arrow sends compact objects to eventually coconnective objects
again by \thmref{t:L sends comp to even coconn} (applied to $M$).

\medskip

Finally, the functor $\Eis^{-,\on{spec}}_{\on{coarse}}$ has a bounded cohomological amplitude on the left,
since the morphism $\sfq^{-,\on{spec}}$ is quasi-smooth.

\ssec{Method of proof}

The proof will be a geometric counterpart of the computation of Whittaker coefficients of
Eisenstein series, coupled with (the parabolic version of) the theory developed in \cite{BG2} and \cite{FH}. 

\sssec{}

Recall that the functor $\Eis^-_!$ extends to a $\QCoh(\LS^{\on{restr}}_\cG)$-linear functor
\begin{equation} \label{e:Eis part enh}
\QCoh(\LS^{\on{restr}}_{\cP^-})\underset{\QCoh(\LS^{\on{restr}}_\cM)}\otimes \Shv_\Nilp(\Bun_M)\to \Shv(\Bun_G),
\end{equation}
see the proof of \propref{p:Eis preserve Nilp}. We will denote the functor in \eqref{e:Eis part enh} by $\Eis_!^{-,\on{part.enh}}$.

\medskip

We will denote by  $\Eis_{!,\rho_P(\omega_X)}^{-,\on{part.enh}}$ the precomposition of $\Eis_!^{-,\on{part.enh}}$
with the corresponding translation functor (this is well-defined as the action of $\QCoh(\LS^{\on{restr}}_\cM)$
on $\Shv_\Nilp(\Bun_M)$ commutes with central translations). 

\begin{rem}

The functor $\Eis_{!,\rho_P(\omega_X)}^{-,\on{part.enh}}$ corresponds under the Langlands functor 
to the functor $\Eis^{-,\on{spec,part.enh}}$ studied in \cite[Sect. 12.6-7]{GLC3}. 

\end{rem} 

\sssec{}

In order to construct the commutativity datum for \eqref{e:L and Eis coarse}, it suffices to do so for the diagram
\begin{equation} \label{e:L and Eis enh}
\CD
\QCoh(\LS^{\on{restr}}_{\cP^-})\underset{\QCoh(\LS^{\on{restr}}_\cM)}\otimes \Shv_\Nilp(\Bun_M) 
@>{\on{Id}\otimes \BL^{\on{restr}}_{M,\on{coarse}}}>> \QCoh(\LS^{\on{restr}}_{\cP^-}) \\
@V{\Eis^{-,\on{part.enh}}_{!,\rho_P(\omega_X)}[\delta_{(N^-_P)_{\rho_P(\omega_X)}}]}VV  @VV{(\sfq^{-,\on{spec}})_*}V \\
 \Shv_\Nilp(\Bun_G) @>>{\BL^{\on{restr}}_{G,\on{coarse}}}> \QCoh(\LS^{\on{restr}}_\cG).
\endCD
\end{equation} 

Since both circuits in \eqref{e:L and Eis enh} are $\QCoh(\LS^{\on{restr}}_\cG)$-linear, in order to construct 
the commutativity datum for \eqref{e:L and Eis enh}, it is enough to do so for its composition with the
functor $\Gamma_!(\LS^{\on{restr}}_\cG,-)$. 

\medskip

Thus, we need to construct the datum of commutativity for the diagram
\begin{equation} \label{e:L and Eis enh Sect}
\CD
\QCoh(\LS^{\on{restr}}_{\cP^-})\underset{\QCoh(\LS^{\on{restr}}_\cM)}\otimes \Shv_\Nilp(\Bun_M) 
@>{\on{Id}\otimes \BL^{\on{restr}}_{M,\on{coarse}}}>> \QCoh(\LS^{\on{restr}}_{\cP^-}) \\
@V{\Eis^{-,\on{part.enh}}_{!,\rho_P(\omega_X)}[\delta_{(N^-_P)_{\rho_P(\omega_X)}}]}VV  @VV{\Gamma_!(\LS^{\on{restr}}_{\cP^-},-)}V \\
 \Shv_\Nilp(\Bun_G) @>>{\on{coeff}_G^{\on{Vac}}}> \Vect.
\endCD
\end{equation} 

\sssec{}

Denote
$$\Omega^{\on{glob}}:=\sfq^{-,\on{spec}}_*(\CO_{\LS^{\on{restr}}_{\cP^-}})\in \on{ComAlg}(\QCoh(\LS^{\on{restr}}_\cM)).$$

The morphism $\sfq^{-,\on{spec}}$ is ``as good as affine" (see \cite[Sect. 12.7.5]{GLC2} for what this means). Hence, we
can identify
$$\QCoh(\LS^{\on{restr}}_{\cP^-})\underset{\QCoh(\LS^{\on{restr}}_\cM)}\otimes \Shv_\Nilp(\Bun_M) 
\simeq \Omega^{\on{glob}}\mod(\Shv_\Nilp(\Bun_M)).$$

\medskip

The datum of the functor $\Eis_!^{-,\on{part.enh}}$ can be interpreted as the action of $\Omega^{\on{glob}}$,
viewed as a monad on $\Shv_\Nilp(\Bun_M)$, on the functor $\Eis^-_!$.

\medskip

Thus, the datum of commutativity of \eqref{e:L and Eis enh Sect} can be stated as an isomorphism of functors
\begin{equation} \label{e:Whit of Eis}
\on{coeff}_G^{\on{Vac}}\circ \Eis^-_{!,\rho_P(\omega_X)}[\delta_{(N^-_P)_{\rho_P(\omega_X)}}]\simeq
\on{coeff}_M^{\on{Vac}}\circ (\Omega^{\on{glob}}\star -)
\end{equation} 
as functors
$$\Shv_\Nilp(\Bun_M) \to \Vect,$$
acted on by the monad $\Omega^{\on{glob}}$. In the above formula, $(-)\star (-)$ denotes the action of
$\QCoh(\LS^{\on{restr}}_\cM)$ on $\Shv_\Nilp(\Bun_M)$. 

\medskip

In what follows we will construct an isomorphism \eqref{e:Whit of Eis}; its compatibility with the action
of $\Omega^{\on{glob}}$ would follow by unwinding the construction of \cite{BG2} 
and \cite{FH}\footnote{More precisely, this follows by combining the Koszul duality statement in the proof of 
Corollary 6.4.1.2 with the proof of Proposition 4.5.4.1 in \cite{FH}.}, see Remark \ref{r:action on Omega}. 

\sssec{}

Recall the category $\Rep(\cM)_\Ran$, which acts on $\Shv(\Bun_M)$. Its action of $\Shv_\Nilp(\Bun_M)$
factors through the localization functor 
$$\Loc_\cM^{\on{spec,restr}}:\Rep(\cM)_\Ran\to \QCoh(\LS^{\on{restr}}_\cM),$$
see \cite[Sect. 12.7.1]{AGKRRV1}.

\medskip

Let $\Omega^{\on{loc}}\in \Rep(\cM)$ be the commutative algebra 
$$\on{C}^\cdot_{\on{Chev}}(\cn^-_P).$$

We attach to it the commutative factorization algebra
$$\on{Fact}(\Omega^{\on{loc}})\in \Rep(\cM)_\Ran,$$
which we view as a (commutative) algebra object with respect to the (symmetric) monoidal 
structure on $\Rep(\cM)_\Ran$, see \cite[Sect. B.10.4]{GLC2}.

\medskip

By \cite[Sect. 12.3.5]{GLC3}, we have
$$\Omega^{\on{glob}}\simeq \Loc_\cG^{\on{spec,restr}}(\on{Fact}(\Omega^{\on{loc}})),$$
as commutative algebra objects in $\QCoh(\LS^{\on{restr}}_\cM)$.

\sssec{}

We can view $\on{Fact}(\Omega^{\on{loc}})$ as a monad acting on $\Shv(\Bun_M)$. We will prove
the following generalization of \eqref{e:Whit of Eis}:

\begin{thm} \label{t:Whit of Eis}
There exists a canonical isomorphism 
$$\on{coeff}_G^{\on{Vac}}\circ \Eis^-_{!,\rho_P(\omega_X)}[\delta_{(N^-_P)_{\rho_P(\omega_X)}}]\simeq
\on{coeff}_M^{\on{Vac}}\circ (\on{Fact}(\Omega^{\on{loc}})\star (-))$$
as functors $\Shv(\Bun_M) \to \Vect$, where $(-)\star (-)$ denotes the monoidal action of $\Rep(\cM)_\Ran$ on $\Shv(\Bun_G)$.
\end{thm} 

\begin{rem}

One can view formula \eqref{e:Whit of Eis} as a geometric counterpart of the classical 
computation of the Whittaker coefficient of Eisenstein series. 

\medskip

The latter says that 
the Whittaker coefficient of Eisenstein series of an automorphic function on $M$
equals the Whittaker coefficient of a particular Hecke operator applied to that
automorphic function. 

\medskip

The Hecke functor $\on{Fact}(\Omega^{\on{loc}})$ is the geometric counterpart of
that classical Hecke functor. 

\end{rem}

\begin{rem}

\thmref{t:Whit of Eis} is equivalent to a (particular case of) \cite[Corollary 10.1.5]{GLC3},
and its proof is parallel to that of \cite[Theorem 10.1.2]{GLC3} in Sect. 10.3 of {\it loc. cit.}

\medskip

The essential difference\footnote{Alluded to in Remark \ref{r:diff GLC3}.}, however, is 
that our key computational ingredient here is \propref{p:Omega}, whereas in \cite{GLC3}
it is Proposition 10.6.8 of {\it loc. cit.}, which uses the local theory developed in that paper,
specifically Corollary 2.5.2 in {\it loc. cit.}.

\end{rem}

\ssec{Proof of \thmref{t:Whit of Eis}}

We will give a proof when $\on{char}(k)$ is positive, i.e., when the Artin-Schreier sheaf exists; 
the proof in the characteristic zero case is completely parallel: one replaces the Artin-Schreier sheaf 
by the procedure of \cite[Sect. 3.3]{GLC1}.

\sssec{} \label{sss:coeff expl}

Denote
$$\Bun_{N,\rho(\omega_X)}:=\Bun_B\underset{\Bun_T}\times \{\rho(\omega_X)\},$$
and denote by 
$$\chi:\Bun_{N,\rho(\omega_X)}\to \BG_a$$
that map
$$\Bun_{N,\rho(\omega_X)}\overset{\chi^I}\to \BG_a^I\overset{\on{sum}}\to \BG_a.$$

\medskip 

Let $\sfp$ denote the projection  
$$\Bun_{N,\rho(\omega_X)}\to \Bun_G.$$

Recall that the functor $\on{coeff}^{\on{Vac}}$ is by defintion
\begin{equation} \label{e:coeff formula}
\on{C}^\cdot(\Bun_{N,\rho(\omega_X)},\sfp^!(-)\sotimes \chi^!(\on{exp}_\omega)),
\end{equation} 
where $\on{exp}_\omega$ denotes the Artin-Schreier sheaf on $\BG_a$, normalized so that
it behaves multiplicatively with respect to the !-pullback. 

\sssec{}

Recall that the functor $\Eis^-_!$ is defined by
\begin{equation} \label{e:Eis formula}
(\sfp^-)_!\circ (\sfq^-)^*(-)[\on{dim.rel}],
\end{equation}
where:

\begin{itemize}

\item $\sfp^-:\Bun_{P^-}\to \Bun_G$;

\smallskip

\item $\sfq^-:\Bun_{P^-}\to \Bun_M$;

\smallskip

\item $\on{dim.rel}$ is the dimension of $\Bun_{P^-}$ over $\Bun_M$ (it depends on the connected
component of $\Bun_M$).

\end{itemize} 

\sssec{}

Let 
$$
\CD
\BunPtm @<{j}<< \Bun_{P^-} \\
@V{\wt\sfp^-}VV \\
\Bun_G
\endCD
$$
denote Drinfeld's relative compactification of $\Bun_{P^-}$ along $\sfp^-$. 
Denote by $\wt\sfq^-$ the map $\BunPtm\to \Bun_M$. 

\medskip

Note that the ULA property of the object $j_!(\ul\sfe_{\Bun_{P^-}})\in \Shv(\BunPtm)$ 
with respect to $\wt\sfq^-$ (see \cite[Theorem 5.1.5]{BG1}) implies that the natural transformation
$$j_!\circ (\wt\sfq^-)^*(-)\simeq  j_!(\ul\sfe_{\Bun_{P^-}})\overset{*}\otimes (\wt\sfq^-)^*(-)\to
j_!(\ul\sfe_{\Bun_{P^-}})\sotimes (\wt\sfq^-)^!(-)[2\dim(\Bun_M)]$$
is an isomorphism.

\medskip

Hence, we can rewrite
\begin{equation} \label{e:Eis via compact}
\Eis^-_!(-)\simeq (\wt\sfp^-)_*\left((\wt\sfq^-)^!(-)\sotimes j_!(\ul\sfe_{\Bun_{P^-}})\right)[2\dim(\Bun_M)+\on{dim.rel}].
\end{equation} 

\sssec{}

Consider the diagram
$$
\xy
(20,0)*+{\Bun_G,}="X";
(20,40)*+{\Bun_{N,\rho(\omega_X)}\underset{\Bun_G}\times \BunPtm}="Y";
(0,20)*+{\Bun_{N,\rho(\omega_X)}}="Z";
(40,20)*+{\BunPtm}="W";
%(20,70)*+{(\Bun_{N,\rho(\omega_X)}\underset{\Bun_G}\times \BunPtm)^{\on{tr}}}="U";
{\ar@{->}_{'\wt\sfp^-} "Y";"Z"};
{\ar@{->}^{'\sfp} "Y";"W"};
{\ar@{->}_{\sfp} "Z";"X"};
{\ar@{->}^{\wt\sfp^-} "W";"X"};
%{\ar@{_{(}->}_j "U";"Y"};  
%{\ar@/_3pc/@{->}_{\,'\wt\sfp^{-,\on{tr}}} "U";"Z"};
%{\ar@/^3pc/@{->}^{\,'\sfp^{\on{tr}}} "U";"W"};
\endxy
$$

By base change and projection formula, we obtain:
\begin{multline} \label{e:coeff of Eis 1}
\on{coeff}_G^{\on{Vac}}\circ \Eis^-_!(-)\simeq \\
\simeq \on{C}^\cdot\left(\Bun_{N,\rho(\omega_X)}\underset{\Bun_G}\times \BunPtm,
\bigl((\,'\sfp)^!\circ ((\wt\sfq^-)^!(-)\sotimes j_!(\ul\sfe_{\Bun_{P^-}}))\bigr)\sotimes \left((\,'\wt\sfp^-)^!\circ \chi^!(\on{exp}_\omega)\right)\right) \\
[2\dim(\Bun_M)+\on{dim.rel}].
\end{multline} 

\sssec{}

Let
$$(\Bun_{N,\rho(\omega_X)}\underset{\Bun_G}\times \BunPtm)^{\on{tr}}\subset \Bun_{N,\rho(\omega_X)}\underset{\Bun_G}\times \BunPtm$$
be the open substack corresponding to the condition that the $N$-reduction and the generalized $P^-$-reduction of the G-bundle are transversal
at the generic point of the curve:

$$
\xy
(20,40)*+{\Bun_{N,\rho(\omega_X)}\underset{\Bun_G}\times \BunPtm}="Y";
(0,20)*+{\Bun_{N,\rho(\omega_X)}}="Z";
(40,20)*+{\BunPtm}="W";
(20,70)*+{(\Bun_{N,\rho(\omega_X)}\underset{\Bun_G}\times \BunPtm)^{\on{tr}}}="U";
{\ar@{->}_{'\wt\sfp^-} "Y";"Z"};
{\ar@{->}^{'\sfp} "Y";"W"};
{\ar@{_{(}->} "U";"Y"};  
{\ar@/_3pc/@{->}_{\,'\wt\sfp^{-,\on{tr}}} "U";"Z"};
{\ar@/^3pc/@{->}^{\,'\sfp^{\on{tr}}} "U";"W"};
\endxy
$$

As in \cite[Lemma 10.6.3]{GLC3}, we have:

\begin{lem}
The natural transformation 
\begin{multline*}
\on{C}^\cdot\left(\Bun_{N,\rho(\omega_X)}\underset{\Bun_G}\times \BunPtm,
\bigl((\,'\sfp)^!\circ ((\wt\sfq^-)^!(-)\sotimes j_!(\ul\sfe_{\Bun_{P^-}}))\bigr)\sotimes \left((\,'\wt\sfp^-)^!\circ \chi^!(\on{exp}_\omega)\right)\right)  \to \\ 
\on{C}^\cdot\left((\Bun_{N,\rho(\omega_X)}\underset{\Bun_G}\times \BunPtm)^{\on{tr}},
\bigl((\,'\sfp^{\on{tr}})^!\circ ((\wt\sfq^-)^!(-)\sotimes j_!(\ul\sfe_{\Bun_{P^-}}))\bigr)\sotimes 
\left((\,'\wt\sfp^{-,\on{tr}})^!\circ \chi^!(\on{exp}_\omega)\right)\right) 
\end{multline*}
is an isomorphism.
\end{lem} 

\medskip

Hence, we obtain that the expression in \eqref{e:coeff of Eis 1} can be rewritten as 
\begin{multline} \label{e:coeff of Eis 2}
\on{C}^\cdot\left((\Bun_{N,\rho(\omega_X)}\underset{\Bun_G}\times \BunPtm)^{\on{tr}},
\bigl((\,'\sfp^{\on{tr}})^!\circ ((\wt\sfq^-)^!(-)\sotimes j_!(\ul\sfe_{\Bun_{P^-}}))\bigr)\sotimes 
\left((\,'\wt\sfp^{-,\on{tr}})^!\circ \chi^!(\on{exp}_\omega)\right)\right) \\
[2\dim(\Bun_M)+\on{dim.rel}].
\end{multline} 

\sssec{} \label{sss:Zast}

Recall the (parabolic) Zastava space 
$$\on{Zast}:=(\Bun_P\underset{\Bun_G}\times \BunPtm)^{\on{tr}},$$
which is the open substack of $\Bun_P\underset{\Bun_G}\times \BunPtm$, corresponding to
the condition that the $P$-reduction and the generalized $P^-$-reduction are transversal 
at the generic point of the curve, see \cite[Sect. 2.2 and Proposition 3.2]{BFGM}. 

\medskip

The stack $\on{Zast}$ is endowed with a map 
$$\fs:\on{Zast}\to \on{Hecke}(M)_{\Ran'},$$
where:

\begin{itemize}

\item $\Ran'$ is the sheafification of the Ran space in the topology of finite surjective maps;

\smallskip

\item $\on{Hecke}(M)_{\Ran'}$ is the corresponding version of the Hecke stack for $M$;

\end{itemize} 

\begin{rem}

In the formulation in \cite{BFGM}, the map $\fs$ rather goes to a version of $\on{Hecke}(M)$ over
the space of colored divisors on $X$, which is the union of schemes of the form
$$X^{(\ul{n})}:=\underset{i}\Pi\, X^{(n_i)}, \quad n_i\in \BZ^{\geq 0},$$
where $i$ runs over (a subset of) the Dynkin diagram of $G$.

\medskip

There is no map from $X^{(\ul{n})}$ to $\Ran$, however, there is one to $\Ran'$; namely it comes 
from the map
$$X^{\ul{n}}:=\underset{i}\Pi\, X^{n_i}\to \Ran,$$
which is invariant with respect to 
$$\Sigma_{\ul{n}}:=\underset{i}\Pi\, \Sigma_{n_i},$$
where $\Sigma_n$ is the symmetric group on $n$ letters.

\end{rem}

\sssec{}

Note that the stack $(\Bun_{N,\rho(\omega_X)}\underset{\Bun_G}\times \BunPtm)^{\on{tr}}$ that appears
in \eqref{e:coeff of Eis 2} is canonically isomorphic to
$$\Bun_{N(M),\rho(\omega_X)}\underset{\Bun_M}\times \on{Zast},$$
where:

\begin{itemize}

\item $N(M)$ is the maximal unipotent subgroup of the Levi $M$;

\smallskip

\item The map $\on{Zast}\to \Bun_M$ is the composition
$$\on{Zast} \overset{\fs}\to \on{Hecke}(M)_{\Ran'} \overset{\hl}\to \Bun_M.$$

\end{itemize} 

\medskip

Under this identification, the map
$$(\Bun_{N,\rho(\omega_X)}\underset{\Bun_G}\times \BunPtm)^{\on{tr}}
\overset{\,'\sfp^{\on{tr}}}\longrightarrow \BunPtm\overset{\wt\sfq^-}\to \Bun_M$$
corresponds to
\begin{multline*} 
\Bun_{N(M),\rho(\omega_X)}\underset{\Bun_M}\times \on{Zast} \overset{\on{id}\times \fs}\longrightarrow \\
\to \Bun_{N(M),\rho(\omega_X)}\underset{\Bun_M,\hl}\times \on{Hecke}(M)_{\Ran'}\to \on{Hecke}(M)_{\Ran'} \overset{\hr}\to \Bun_M.
\end{multline*}

%$$\Bun_{N(M),\rho(\omega_X)}\underset{\Bun_M}\times \on{Zast}\to \on{Zast} \overset{\fs}\to \on{Hecke}(M)_{\Ran'} \overset{\hr}\to \Bun_M.$$

\sssec{}

Let us denote by 
$$\Omega_\chi\in \Shv\Bigl( \Bun_{N(M),\rho(\omega_X)}\underset{\Bun_M,\hl}\times \on{Hecke}(M)_{\Ran'}\Bigr)$$
the object equal to the *-direct image along 
\begin{multline*} 
(\Bun_{N,\rho(\omega_X)}\underset{\Bun_G}\times \BunPtm)^{\on{tr}} \simeq 
\Bun_{N(M),\rho(\omega_X)}\underset{\Bun_M}\times \on{Zast}\overset{\on{id}\times \fs}\longrightarrow \\
\to \Bun_{N(M),\rho(\omega_X)}\underset{\Bun_M,\hl}\times \on{Hecke}(M)_{\Ran'}
\end{multline*} 
of 
$$\left((\,'\sfp^{\on{tr}})^!\circ j_!(\ul\sfe_{\Bun_{P^-}})\right)\sotimes 
\left((\,'\wt\sfp^{-,\on{tr}})^!\circ \chi^!(\on{exp}_\omega)\right)\in \Shv\Bigl((\Bun_{N,\rho(\omega_X)}\underset{\Bun_G}\times \BunPtm)^{\on{tr}}\Bigr).$$

\medskip

Let us denote by $r_2$ the projection
$$\Bun_{N(M),\rho(\omega_X)}\underset{\Bun_M,\hl}\times \on{Hecke}(M)_{\Ran'}\to \on{Hecke}(M)_{\Ran'}.$$
%$\Bun_{N(M),\rho(\omega_X)}$ and $\on{Hecke}(M)_{\Ran'}$, respectively.

\medskip

Applying the projection formula, we obtain that the expression in \eqref{e:coeff of Eis 2} identifies with
\begin{multline} \label{e:coeff of Eis 3}
\on{C}^\cdot\left(\Bun_{N(M),\rho(\omega_X)}\underset{\Bun_M,\hl}\times \on{Hecke}(M)_{\Ran'},(r_2^!\circ \hr^!(-))\sotimes \Omega_\chi\right) \\
[2\dim(\Bun_M)+\on{dim.rel}].
\end{multline} 

\sssec{}

Consider the map 
\begin{multline} \label{e:trans rhoP}
\on{transl}_{\rho_P(\omega_X)}:\Bun_{N(M),\rho_M(\omega_X)}\underset{\Bun_M,\hl}\times \on{Hecke}(M)_{\Ran'}\to \\
\to \Bun_{N(M),\rho(\omega_X)}\underset{\Bun_M,\hl}\times \on{Hecke}(M)_{\Ran'},
\end{multline} 
given by (central) translation by $\rho_P(\omega_X)$. 

\medskip

Let $r_1$ denote the projection 
$$\Bun_{N(M),\rho_M(\omega_X)}\underset{\Bun_M,\hl}\times \on{Hecke}(M)_{\Ran'}\to \Bun_{N(M),\rho_M(\omega_X)}.$$

\medskip

Here is the key computational input in the proof of \thmref{t:L and Eis}:

\begin{prop} \label{p:Omega}
The pullback of $\Omega_\chi$ along \eqref{e:trans rhoP}, viewed as an object of
$$\Shv(\Bun_{N(M),\rho_M(\omega_X)}\underset{\Bun_M,\hl}\times \on{Hecke}(M)_{\Ran'}),$$
shifted cohomologically by
$$[2\dim(\Bun_M)+\on{dim.rel}+\delta_{(N^-_P)_{\rho_P(\omega_X)}}],$$
identifies canonically with
$$(r_1^!\circ \chi_M^!(\on{exp}_\omega))\sotimes (\Sat^{\on{nv}}_M\circ \on{Fact}(\Omega^{\on{loc}})),$$
where:

\begin{itemize} 

\item $\chi_M:\Bun_{N(M),\rho_M(\omega_X)}\to \BG_a$ is the counterpart of the map $\chi$ for $M$.

\smallskip

\item $\Sat^{\on{nv}}_M:\Rep(\cM)\to \Sph_M$ is the naive geometric Satake functor;

\smallskip

\item By a slight abuse of notation, we denote by $\Sat^{\on{nv}}_M\circ \on{Fact}(\Omega^{\on{loc}}))$ the image
of the corresponding object under the equivalence 
$$\Shv(\on{Hecke}(M)_{\Ran})\simeq \Shv(\on{Hecke}(M)_{\Ran'}).$$

\end{itemize}

\end{prop} 

\begin{rem} \label{r:action on Omega}

Recall that in order to deduce \thmref{t:L and Eis} from \thmref{t:Whit of Eis} we also need to know
that the isomorphism \eqref{e:Whit of Eis} is compatible with $\Omega^{\on{glob}}$-actions.

\medskip

This compatibility follows from the corresponding property of the isomorphism of  \propref{p:Omega}.

\end{rem}  

\sssec{}

Let us assume \propref{p:Omega} for a moment and finish the proof of \thmref{t:Whit of Eis}.

\medskip

Indeed, combining \eqref{e:coeff of Eis 3} with \propref{p:Omega}, we obtain that the functor
$$\on{coeff}_G^{\on{Vac}}\circ \Eis^-_{!,\rho_P(\omega_X)}[\delta_{(N^-_P)_{\rho_P(\omega_X)}}]$$
is isomorphic to
%\begin{equation} \label{e:coeff of Eis 4}
$$\on{C}^\cdot\Bigl(\Bun_{N(M),\rho_M(\omega_X)}\underset{\Bun_M,\hl}\times \on{Hecke}(M)_{\Ran'},
(r_2^!\circ \hr^!(-))\sotimes (r_1^!\circ \chi_M^!(\on{exp}_\omega))\sotimes (\Sat^{\on{nv}}_M\circ \on{Fact}(\Omega^{\on{loc}}))\Bigr)$$
%\end{equation} 

By base change and projection formula, the letter expression identifies with
\begin{equation} \label{e:coeff of Eis 4}
\on{C}^\cdot\left(\Bun_{N(M),\rho_M(\omega_X)},\sfp_M^!
\Bigl(\hl_*\bigl(\hr^!(-)\sotimes (\Sat^{\on{nv}}_M\circ \on{Fact}(\Omega^{\on{loc}}))\bigr)\Bigr)
\sotimes \chi_M^!(\on{exp}_\omega)\right),
\end{equation} 
where 
$$\sfp_M:\Bun_{N(M),\rho_M(\omega_X)}\to \Bun_M.$$

\medskip

We have, by definition:
$$\hl_*\left(\hr^!(-)\sotimes (\Sat^{\on{nv}}_M\circ \on{Fact}(\Omega))\right)\simeq  \on{Fact}(\Omega^{\on{loc}})\star (-),$$
as endofunctors of $\Shv(\Bun_G)$.

\medskip

Thus, the expression in \eqref{e:coeff of Eis 4} identifies with
$$\on{C}^\cdot\left(\Bun_{N(M),\rho_M(\omega_X)}, \sfp_M^!\left(\on{Fact}(\Omega^{\on{loc}})\star (-)\right)\sotimes \chi_M^!(\on{exp}_\omega)\right)
\simeq \on{coeff}_M^{\on{Vac}}\left(\on{Fact}(\Omega^{\on{loc}})\star (-)\right),$$
as required.

\qed[\thmref{t:Whit of Eis}]

\ssec{Proof of \propref{p:Omega}}

\sssec{}

Note that the map $\chi$ is a sum of two maps $\chi^M$ and $\chi^{N_P}$, where the former
is the sum over the simple roots in the Dynkin diagram of $M$ and the latter 
is the sum over the other simple roots.

\medskip

Note that the composition
$$\Bun_{N(M),\rho(\omega_X)}\underset{\Bun_M}\times \on{Zast}\simeq 
\Bun_{N,\rho(\omega_X)}\underset{\Bun_G}\times \BunPtm\overset{'\wt\sfp^{-,\on{tr}}}\longrightarrow 
\Bun_{N,\rho(\omega_X)} \overset{\chi^M}\to \BG_a$$
identifies with
$$\Bun_{N(M),\rho(\omega_X)}\underset{\Bun_M}\times \on{Zast}\overset{r_1}\to 
\Bun_{N(M),\rho(\omega_X)}\overset{\chi_M}\to \BG_a.$$

\medskip

Denote by $\Omega_{\chi^{N_P}}$ the object of $\Shv\Bigl( \Bun_{N(M),\rho(\omega_X)}\underset{\Bun_M,\hl}\times \on{Hecke}(M)_{\Ran'}\Bigr)$
defined in the same way as $\Omega_\chi$, but with $\chi$ replaced by $\chi_{N^P}$. 

\medskip

By base change \propref{p:Omega} is equivalent to the isomorphism 
\begin{equation} \label{e:Omega 1}
(\on{transl}_{\rho_P(\omega_X)})^!(\Omega_{\chi^{N_P}})[2\dim(\Bun_M)+\on{dim.rel}+\delta_{(N^-_P)_{\rho_P(\omega_X)}}]\simeq 
\Sat^{\on{nv}}_M\circ \on{Fact}(\Omega^{\on{loc}})
\end{equation}
as objects in $$\Shv(\Bun_{N(M),\rho_M(\omega_X)}\underset{\Bun_M,\hl}\times \on{Hecke}(M)_{\Ran'}).$$

\sssec{}

Let
$$\overset{\circ}{\on{Zast}}\overset{j_{\on{Zast}}}\hookrightarrow \on{Zast}$$
be the open Zastava, i.e., the corresponding open 
$$(\Bun_P\underset{\Bun_G}\times \Bun_{P^-})^{\on{tr}} \subset \Bun_P\underset{\Bun_G}\times \Bun_{P^-}.$$

\medskip

By a slight abuse of notation, let us denote by $'\sfp^{\on{tr}}$ the map
$\on{Zast}\to \BunPtm$ and by $'\overset{\circ}\sfp{}^{\on{tr}}$ the map
$$\overset{\circ}{\on{Zast}}\to  \Bun_{P^-}.$$

\medskip

Since the stacks below involved are smooth, we have 
$$(\,'\overset{\circ}\sfp{}^{\on{tr}})^!(\ul\sfe_{\Bun_{P^-}})\simeq \ul\sfe_{\overset{\circ}{\on{Zast}}}
[2(\dim(\overset{\circ}{\on{Zast}})-\dim(\Bun_{P^-}))].$$

From here we obtain a map
\begin{equation} \label{e:acycl comp}
(j_{\on{Zast}})_!(\ul\sfe_{\overset{\circ}{\on{Zast}}})[2(\dim(\overset{\circ}{\on{Zast}})-\dim(\Bun_{P^-}))]\to
(\,'\sfp^{\on{tr}})^!\circ j_!(\ul\sfe_{\Bun_{P^-}}).
\end{equation}

The next assertion is \cite[Lemma 4.1.10]{Lin}:

\begin{lem} \label{l:Kevin}
The map \eqref{e:acycl comp} is an isomorphism.
\end{lem}

\sssec{}

From \lemref{l:Kevin}, we obtain that the object
$$\Omega_{\chi^{N_P}}[2\dim(\Bun_M)+\on{dim.rel}+\delta_{(N^-_P)_{\rho_P(\omega_X)}}]$$
is isomorphic to 
\begin{equation} \label{e:Omega 2}
(\on{id}\times \fs)_*
\left(\Bigl(r_2^!\circ (j_{\on{Zast}})_!(\ul\sfe_{\overset{\circ}{\on{Zast}}})\Bigr)\sotimes 
\Bigl((\,\wt\sfp^{-,\on{tr}})^!\circ (\chi^{N_P})^!(\on{exp}_\omega)\Bigr)\right)[\dim(\overset{\circ}{\on{Zast}})].
\end{equation} 

The required isomorphism between the pullback of \eqref{e:Omega 2} along $\on{transl}_{\rho_P(\omega_X)}$
and the object $\Sat^{\on{nv}}_M\circ \on{Fact}(\Omega^{\on{loc}})$ follows from \cite[Sect. 4.6.1]{Ra1} (for $P=B$)
and \cite[Theoren 1.4.3.1]{FH} (for an arbitrary parabolic).

\end{document}